\documentclass[10pt,a4paper]{article}
\usepackage[T1]{fontenc}
\usepackage{lmodern}
\usepackage[a4paper,left=20mm,right=20mm,top=25mm,bottom=25mm]{geometry}
\usepackage{authblk}

\usepackage[utf8]{inputenc}
\usepackage{amsmath,amssymb,amsthm,mathtools}
\usepackage[numbers,sort&compress]{natbib}
\usepackage{hyperref}
\hypersetup{hidelinks,pdftitle={Multi-time Markov renewal chains and stratified renewal theorems},pdfauthor={Leonidas Kordalis and Samis Trevezas},pdfsubject={Probability theory: Markov renewal processes}}
\usepackage{microtype}
\usepackage{enumitem}
\usepackage{needspace}

\providecommand{\doi}[1]{\url{https://doi.org/#1}}

\newcommand{\N}{\mathbb N}
\newcommand{\Z}{\mathbb Z}
\newcommand{\R}{\mathbb R}
\newcommand{\Pp}{\mathbb P}
\newcommand{\Ee}{\mathbb E}
\newcommand{\Cov}{\operatorname{Cov}}
\newcommand{\Var}{\operatorname{Var}}
\newcommand{\dist}{\operatorname{dist}}
\newcommand{\argmin}{\operatorname*{arg\,min}}
\newcommand{\argmax}{\operatorname*{arg\,max}}
\newcommand{\one}{\mathbf 1}
\newcommand{\leqd}{\leq_d}
\newcommand{\notleqd}{\not\leq_d}

\newcommand{\dd}{\mathrm d}

\theoremstyle{plain}
\newtheorem{theorem}{Theorem}[section]
\newtheorem{proposition}[theorem]{Proposition}
\newtheorem{lemma}[theorem]{Lemma}
\newtheorem{corollary}[theorem]{Corollary}
\theoremstyle{definition}
\newtheorem{assumption}[theorem]{Assumption}
\newtheorem{definition}[theorem]{Definition}
\newtheorem{remark}[theorem]{Remark}
\newtheorem{example}[theorem]{Example}

\title{Multi-time Markov renewal chains and stratified renewal theorems}
\author[1]{Leonidas Kordalis}
\author[1,2]{Samis Trevezas}
\affil[1]{Department of Mathematics, National and Kapodistrian University of Athens, Athens 15784, Greece}
\affil[2]{MICS Laboratory, CentraleSup\'elec, Universit\'e Paris-Saclay, 91190 Gif-sur-Yvette, France}

\date{}

\begin{document}
\maketitle
\begin{center}
\small
\href{mailto:lkordali@math.uoa.gr}{lkordali@math.uoa.gr}\qquad
\href{mailto:strevezas@math.uoa.gr}{strevezas@math.uoa.gr}
\end{center}

\begin{abstract}
We develop stratified renewal theorems for multi-time observations of Markov-additive
paths on standard Borel state spaces with nonnegative integer vector increments. Horizon
directions with the same active constraints form strata. A compact-uniform inverse expansion yields a global limit in a topology
combining hypi convergence with exact restrictions to each stratum. At Brownian scale, the central
theorem determines the joint terminal-edge law at fixed finite families of critical horizons
about a common direction, under nondegenerate pairwise Gaussian contrasts. One
Gaussian vector selects the exiting coordinates. Selections of the same coordinate
at the same threshold share an edge; different selected pairs carry conditionally
independent equilibrium marks. These marks retain both endpoint states, the full increment, age and
overshoot. Under total-variation convergence of the state marginal to stationarity,
pointwise scalar renewal convergence in total variation suffices for the marked argument. Regenerative conditions provide
explicit criteria. Stationary scalar level processes retain dependence at fixed integer
offsets. Under additional moment conditions, a second-order allocation theorem localizes
optimal budgets and, with the marked hypotheses and nonsingular Gaussian contrasts,
yields a cost-weighted equilibrium-edge mixture. An exact covariance criterion links
uniform tightness of the unperturbed inverse field to the absence of a square-root
allocation loss. Finite, denumerable and continuous modulating environments are
included.
\end{abstract}



\section{Introduction}
\label{sec:introduction}

Several cumulative quantities can govern the evolution of one system:
calendar age, operating time, usage or exposure.  The time-dimensional
perspective of Kordalis and Trevezas~\cite{KordalisTrevezasMRCAlg2025,KordalisTrevezasRenewal2026}
treats these quantities as coordinates of a common renewal horizon and
observes the system through their joint thresholds.  Dependence among the
coordinates is allowed.  A renewal is included only while all thresholds
are respected, and exit occurs as soon as at least one threshold is exceeded.
Varying the thresholds produces a field indexed by a partially ordered set
of time vectors, although all observations concern one renewal sequence.

The resulting field raises two questions.  How does the inverse fluctuate
when the coordinate limiting the count changes?  How are the terminal edges
at several horizons coupled, particularly when the horizons are incomparable?
These questions arise whenever one compares several admissible horizons,
assesses their joint terminal costs, or chooses how to distribute resources.
They require simultaneous control of the active constraints and of the
transitions selected by those constraints.  We develop this joint analysis
for a Markov-modulated renewal sequence on a standard Borel state space.
The three principal limit theorems use this general framework; finite and
denumerable state spaces provide useful specializations.

Write $\N=\{0,1,2,\ldots\}$ and fix $d\geq1$.
Let $(J_n,S_n)_{n\geq0}$ be a Markov-additive process with state component
$J_n$ and nonnegative vector increments $X_n=S_n-S_{n-1}\in\N^d$.
We take $S_0=0$ and use $\leqd$ for coordinatewise order.
At the transition index $n$ this is a one-parameter process.  Its observation
through lower rectangles gives
\begin{equation*}
 N(k)=\sup\{n\geq0:S_n\leqd k\},\qquad
 Z_k=J_{N(k)},\qquad U_k=k-S_{N(k)},\qquad k\in\N^d.
\end{equation*}
The count $N$, state $Z$ and vector age $U$ form the multi-time renewal
field.  Markov additivity specifies the underlying path, and the lower
rectangles organize its observations by the partial order of time vectors.
The conditional law of the future quadrant then describes how information
at one horizon propagates to all larger horizons, including their joint
observations at incomparable corners.

If $\mu\in(0,\infty)^d$ is the mean displacement per transition, the
first-order count in direction $\lambda\in(0,\infty)^d$ has rate
\begin{equation*}
 \rho(\lambda)=\min_{r\leq d}\frac{\lambda_r}{\mu_r},\qquad
 I(\lambda)=\left\{r:\lambda_r/\mu_r=\rho(\lambda)\right\}.
\end{equation*}
Directions with the same active set form a stratum.  The active set changes
when a direction reaches a tie between coordinate constraints.  A single
active coordinate gives a linear
inverse fluctuation; several active coordinates give a minimum of correlated
branches.  This fixed-direction phenomenon is classical, including without
Markov modulation.  The field-level questions concern convergence across these strata and the
joint microscopic information retained at the selected edges.
The overshoot is the amount by which the exiting increment carries a
coordinate beyond its threshold.  At a tie, fluctuations of the cumulative
quantities decide which threshold is exceeded first.  The transition causing
that exit carries information about its states and increment in addition
to the accumulated fluctuation.  Relating these two scales is the central
probabilistic problem of the paper.  Budgets with the same first-order count
can differ in their second-order count losses and terminal-state distributions.

\subsection{Principal results}

The main contribution is a joint limit for the terminal transitions selected
by several competing rectangular horizons, including the dependence caused
by repeated coordinate thresholds.  Its proof combines Gaussian ordering
with pointwise scalar renewal convergence on a standard Borel state space.
A global stratified inverse theorem and a second-order allocation theorem connect this
terminal law to changes in the active constraints and to optimal budgets.
The exact kernel identities and functional-limit criteria support these
three principal results.

The first is the global stratified inverse theorem,
Theorem~\ref{thm:stratified-fclt}.  If the centered additive path has a
continuous functional limit $G$ at a scale $a_t\to\infty$, $a_t=o(t)$,
then a compact-uniform finite-scale expansion describes the inverse over
all directions, including the shrinking neighborhoods where the active
coordinate changes.  For continuous $b$ and thresholds
$\lfloor t\lambda+a_tb(\lambda)\rfloor$,
the limiting field is
\begin{equation*}
 \Phi_{G,b}(\lambda)=\min_{r\in I(\lambda)}
       \frac{b_r(\lambda)-G_r(\rho(\lambda))}{\mu_r}.
\end{equation*}
Global convergence uses the hypi topology of
B{\"u}cher, Segers and Volgushev~\cite{BucherSegersVolgushev2014}, together
with locally uniform convergence on every fixed active-set stratum.
The latter retains the actual value at ties, which ordinary hypi convergence
does not determine.  Normal perturbations identify the completed interface
values.  For unperturbed thresholds $\lfloor t\lambda\rfloor$, at Brownian
scale, with long-run covariance matrix $\Gamma$,
Proposition~\ref{prop:uniform-tightness-dichotomy}
gives an exact criterion for local supremum-norm tightness over the whole
direction space: it holds precisely when
$\Gamma=\sigma^2\mu\mu^{\mathsf T}$ for some $\sigma\geq0$, so all normalized Brownian coordinates
coincide.  Otherwise a transverse tie prevents uniform tightness.

The central marked result, Theorem~\ref{thm:marked-competing-face},
describes several critical horizons on the same path.  Fix a direction
$\lambda^0$ and corners
$k_t(h^a)=\lfloor t\lambda^0+\sqrt t\,h^a\rfloor$, $1\leq a\leq K$.
Their limiting inverse fluctuations and selected faces are functions of one
Gaussian vector.  Write $q_v(x,\dd y)$ for the joint transition kernel
of the next state and increment $v$, and $\pi$ for the stationary state law.
Conditional on this vector, a selected coordinate $r$
carries the equilibrium-edge law
\begin{equation}
 \nu_r(\dd x,\dd y,\{v\},\{a\},\{o\})
 =\frac{\pi(\dd x)q_v(x,\dd y)}{\mu_r}
       \one_{\{0\leq a<v_r,\ o=v_r-a\}}.
 \label{eq:intro-palm}
\end{equation}
This law retains the departure state, arrival state, full vector increment,
and selected-coordinate age and overshoot.  Conditional on the common Gaussian vector, distinct selected pairs
$(r,h_r^a)$ receive independent equilibrium marks, while identical pairs
retain the same mark.  The theorem also gives locally uniform one-horizon
limits.  Its joint law separates the dependence created by the shared
Gaussian fluctuation from the additional dependence created when two
horizons select the same terminal edge.  Corollary~\ref{cor:marked-fixed-offsets}
retains the stationary renewal dependence between coordinate levels at
fixed integer separations.

The marked theorem assumes a Brownian functional limit, total-variation
convergence of the state marginal to stationarity, pointwise total-variation
equilibration of the scalar terminal law, and nondegenerate pairwise
Gaussian contrasts. Singular active covariance is allowed. The scalar
convergence is required only pointwise in the starting state: an induction
controls the distributions of the states at successive passages, so a
uniform estimate over all starting states is unnecessary.
A regenerative argument proves the scalar input while retaining the
departure-state dependence between adjacent blocks.  Return-cycle
conditions give denumerable-state criteria, and finite-state matrix
conditions recover the simpler specialization.  A persistent continuous
environment with state-dependent increments satisfies all three principal
theorems.

The third result, Theorem~\ref{thm:allocation}, connects the inverse theorem
to optimal rectangular budgets with cost vector $c\in(0,\infty)^d$.
Under a Brownian limit and the stated
partial-sum moment bound, every exact maximizer of $\Ee[N(k)]$ subject to
$c\cdot k\leq t$ lies within $O(\sqrt t)$ of $t\mu/(c\cdot\mu)$.
The theorem gives the second-order value and the set of optimal corrections.
The square-root loss in the optimal expected count is strictly positive
unless all normalized coordinate fluctuations coincide.  Thus the same
covariance condition characterizes uniform tightness of the unperturbed inverse field
and absence of a square-root allocation loss from competing constraints.
With nonsingular Gaussian contrasts the correction $h_*$ is unique on
$c\cdot h=0$ and satisfies
\begin{equation}
 \Pp(r\text{ is the selected limiting coordinate at }h_*)
       =\frac{c_r\mu_r}{c\cdot\mu}.
 \label{eq:intro-balance}
\end{equation}
When the marked hypotheses also hold, combining \eqref{eq:intro-palm} and
\eqref{eq:intro-balance} gives the
terminal law at an optimal allocation as the cost-weighted mixture
$\sum_r c_r\mu_r\nu_r/(c\cdot\mu)$.
Under the same nonsingular contrast condition, Proposition
\ref{prop:allocation-curvature} gives an explicit Hessian and quadratic
growth at the optimum.  It converts a uniform objective error into an
allocation error bound; obtaining rates for that objective error remains
a quantitative renewal problem.

The exact observation theory uses the same standard Borel state space.
Its marked-potential identity, residual future-quadrant kernel and
lumpability criterion are structural propositions.  Zero-displacement
clusters and state self-transitions are included under an explicit
transience condition.  The residual kernel specifies the joint conditional law of all future
observations and supplies the exact identities used in the asymptotic
analysis.  Together, the structural and limit results provide a foundation
for simultaneous terminal-risk assessment and for resource decisions based
on several cumulative quantities.

\subsection{Connections with earlier work}

The finite-state formulation of Kordalis and
Trevezas~\cite{KordalisTrevezasMRCAlg2025} supplies the matrix convolution
algebra and multi-index renewal equations, with an explicit interpretation
through several time coordinates.  Their subsequent renewal study
\cite{KordalisTrevezasRenewal2026} considers successive increment vectors
that are independent and identically distributed, while allowing dependence
among the coordinates of each vector. Under proportional horizon growth,
it gives a strong law for the rectangular count, the corresponding
first-order mean asymptotic, and a renewal-reward central limit theorem
centered by the mean reward times the realized number of renewals. It also gives a count central
limit theorem for a uniquely active coordinate and derives the
Gaussian-minimum conclusion at ties under joint marginal-count convergence
and the required directional centering. Section~5 of that work uses the
vector age to retain the terminal censoring event in exact likelihood
inference. Building on these fixed-direction and finite-horizon results,
the present analysis treats changing active sets and the joint terminal
edges at several horizons. Its exact observation theory extends the
modulating state space and specifies the conditional law of the entire
future quadrant.

The role of the age variable also connects with the scalar discrete-time
construction of Trevezas and Limnios~\cite{TrevezasLimnios2011}, who augment
a finite-state semi-Markov chain by its backward recurrence time and retain
the terminal survival factor in exact likelihood inference. Their main
statistical model uses positive sojourns and state changes at renewal epochs.
The vector-age construction here permits state self-transitions and
zero-displacement clusters under the stated transience condition.

Scalar Markov renewal theory provides the underlying passage and potential
methods.  Its finite-state foundations include
Pyke~\cite{Pyke1961}, Pyke and Schaufele~\cite{PykeSchaufele1964} and
{\c C}inlar~\cite{Cinlar1969}; general-state developments include
Jacod~\cite{Jacod1971,Jacod1974}, Kesten~\cite{Kesten1974}, Athreya,
McDonald and Ney~\cite{AthreyaMcDonaldNey1978},
Nummelin~\cite{Nummelin1978} and Alsmeyer~\cite{Alsmeyer1994}.
The scalar equilibrium law used below belongs to this theory.
The Markov-additive formalism of
{\c C}inlar~\cite{Cinlar1972a,Cinlar1972b} and
Ney and Nummelin~\cite{NeyNummelin1987} accommodates the vector path.
The multi-time assertions concern the additional observation structure
obtained by applying every lower-rectangle horizon to that path.

Vector renewal counts also have a substantial classical theory.
Hunter~\cite{Hunter1974a,Hunter1974b} develops bivariate renewal equations
and asymptotics. Kennedy~\cite{Kennedy1971}, Csenki~\cite{Csenki1979} and
Horv{\'a}th~\cite{Horvath1984} establish multidimensional invariance or
strong-approximation results, while Iglehart and Whitt~\cite{IglehartWhitt1971}
and Vervaat~\cite{Vervaat1972} give general inversion principles.
Niculescu~\cite{Niculescu1984} represents the rectangular count as the
minimum of coordinate renewal counts and treats coincident directional
ratios for independent increment vectors under the stated moment and
nondegeneracy conditions. Niculescu and Omey~\cite{NiculescuOmey1985}
establish equivalence of normal-approximation rates for vectors of sums
and their associated passage or renewal variables, extending Ahmad's
bivariate results~\cite{Ahmad1981}. This theory already contains the
fixed-direction Gaussian minimum. The global theorem below describes
the directional field across ties through a uniform finite-scale
approximation and exact restrictions to each active-set stratum.

An important marked-passage comparison is the work of Fuh and
Lai~\cite{FuhLai2001}. For a scalar boundary, their joint limit retains
the terminal Markov state and overshoot jointly with transverse Gaussian
coordinates, under weighted ergodicity, moment and strong nonlattice
assumptions on the ascending ladder process. In particular, the terminal
state and overshoot are asymptotically independent of those Gaussian
coordinates. Fuh~\cite{Fuh2004} develops uniform scalar renewal estimates
and higher-order passage approximations in a nonlattice setting.
D{\"o}ring and Trottner~\cite{DoringTrottner2023} identify invariant laws
and establish convergence and mixing of level-indexed overshoots for
continuous-time Markov-additive processes with finite modulation.
Building on this scalar passage perspective, our marked result determines
the joint terminal observations selected by a finite family of rectangular
horizons. A common Gaussian vector selects the coordinate at each corner,
and equality of the selected coordinate thresholds determines which marks
remain identical. The proof controls successive passage states using
pointwise scalar total-variation convergence and convergence of the state
marginal to stationarity.

The exact terminal quantities have related precedents in the multivariate
first-excess transforms of Dshalalow~\cite{Dshalalow1994,Dshalalow1997,Dshalalow2001}
and Dshalalow and Liew~\cite{DshalalowLiew2006}. White~\cite{White2022}
gives probabilities for arbitrary weak orderings of coordinate crossings,
including simultaneous crossings, for i.i.d. nonnegative vector increments.
These results provide exact descriptions of passage events and associated
observations. The present marked limit links the complete terminal edge
to Gaussian coordinate selection at several critical horizons and retains
the dependence created by repeated selected thresholds.

Complementary estimates for multidimensional renewal potentials and
transition probabilities are developed by Babillot~\cite{Babillot1988},
Uchiyama~\cite{Uchiyama2007} and Guibourg and Herv{\'e}~\cite{GuibourgHerve2013}.
Ferr{\'e}, Herv{\'e} and Ledoux~\cite{FerreHerveLedoux2012} derive local
limits and other distributional approximations under an $L^2$ spectral gap,
and Logachov, Mogulskii, Prokopenko and Yambartsev~\cite{LogachovEtAl2021}
study lattice local limits for finite-state semi-Markov additive functionals.
These results describe potential or endpoint behavior under their
respective ergodicity, moment and arithmetic conditions. In the marked
argument below, scalar coordinate renewal laws supply the terminal-edge
input and the functional limit determines the competing coordinates.

The topology in the inverse result is based on
B{\"u}cher, Segers and Volgushev~\cite{BucherSegersVolgushev2014}; its use
for inverse-type statistical processes also appears in
Zwingmann and Holzmann~\cite{ZwingmannHolzmann2020}.
The trace coordinates record additional interface values for this renewal
problem.  The random-set inverses of Ilienko and
Molchanov~\cite{IlienkoMolchanov2018} concern scalar multiple sums indexed
by $\N^d$.  Here a scalar summation index generates observations indexed
by vector horizons.  This construction complements the indexing structures
studied in general multiparameter Markov theory
\citep{Khoshnevisan2002}, plane-indexed martingale structures
\citep{CairoliWalsh1975}, multidimensional-time random-walk renewals
\citep{NeyWainger1972,MaejimaMori1984,Klesov1991}, set-indexed renewals
\citep{IvanoffMerzbach2006}, spatial Markov renewal models
\citep{DaviesBraun2011}, and semi-Markov reward fields
\citep{SoltaniGhasemi2014}.

The allocation result uses uniform integrability and optimizer localization
in addition to the distributional inverse limit.  Its optimization context
is discussed in Section~\ref{sec:allocation}.  Joint transition rewards
are treated through their displacement--reward covariance, in keeping with
random-time-change theory
\citep{MasudaSumita1991,GlynnWhitt1993,Ball1999,Khorshidian2009,ChatrabgounEtAl2020}.
The finite-horizon convolution algorithms of Kordalis and
Trevezas~\cite{KordalisTrevezasFFT2026} develop the computational side of
single-time semi-Markov renewal equations.  The exact convolution and residual
kernels in the present work identify corresponding quantities for multi-time
observation and connect them to the asymptotic laws.

\subsection{Organization and conventions}

Sections~\ref{sec:kernel}--\ref{sec:input} give the model, the residual
observation law and the functional-limit inputs.
Sections~\ref{sec:stratified}, \ref{sec:marked} and \ref{sec:allocation}
contain the three principal theorems and their proofs.
Section~\ref{sec:discussion} relates the stratified limits to simultaneous
observation, quantitative renewal estimates, statistical allocation and jump limits.
The appendices contain the structural proofs, sufficient probabilistic
conditions and topological details.

We write $\R_+=[0,\infty)$ and $[d]=\{1,\ldots,d\}$.
Vectors are columns, and $e_r$ is the $r$th
coordinate vector in $\R^d$.  For $I\subset[d]$, $v_I$ denotes the
subvector indexed by $I$ and $\Gamma_I$ the corresponding principal
submatrix.  The order $\leq_d$ and floor operation are coordinatewise,
$\|\cdot\|$ is the Euclidean norm on finite-dimensional vector spaces,
and $|v|_1=\sum_r|v_r|$.  We write $\one_A$ for the indicator of $A$,
$\one$ for the constant function or vector of ones, and
$\operatorname{spr}$ for the spectral radius of a matrix.
State spaces carry their specified Borel
$\sigma$-fields; finite and denumerable spaces carry the discrete
$\sigma$-field.  Random elements in nonseparable function spaces are
interpreted in the Hoffmann--J\o rgensen outer-probability sense.

\section{Kernel potential and lower-rectangle observation}
\label{sec:kernel}

Let $(\mathsf X,\mathcal X)$ be a standard Borel space.  A discrete
vector-displacement Markov-renewal kernel is a family of nonnegative kernels
\begin{equation*}
q=(q_v)_{v\in\N^d},\qquad
\sum_{v\in\N^d}q_v(x,\mathsf X)=1.
\end{equation*}
Write $q_0^0=I$ and let $q_0^m$ denote the $m$-fold kernel composition.  The
standing condition for the exact observation theory is uniform transience of
the zero-displacement kernel.

\begin{assumption}
\label{ass:zero-cluster}
The kernel
\begin{equation*}
 R_0=\sum_{m=0}^{\infty}q_0^m
\end{equation*}
is finite and satisfies
\begin{equation}
 C_0:=\sup_{x\in\mathsf X}R_0(x,\mathsf X)<\infty.
 \label{eq:zero-cluster-bound}
\end{equation}
\end{assumption}

The strict case $q_0=0$ has $R_0=I$ and $C_0=1$.  Assumption
\ref{ass:zero-cluster} also permits state changes at zero displacement, but
excludes an infinite zero-displacement cluster.  In a finite state space,
write $Q_v=(q_v(i,\{j\}))_{i,j}$ for every $v$.  The condition is then
equivalent to $\operatorname{spr}(Q_0)<1$.
We impose Assumption~\ref{ass:zero-cluster} throughout this section and
Section~\ref{sec:augmentation}.  Its uniform bound controls the expected
number of transitions up to the next nonzero displacement, from every
initial state.  It is stronger than almost-sure finiteness of each such
waiting time and ensures that every coefficient of the potential below is
a finite kernel.  In fact it gives uniform exponential moments of the
zero-cluster length; Lemma~\ref{lem:zero-cluster-bounds} proves this and the
bounds $\sup_x\psi_k(x,\mathsf X)\leq C_0$ and
$\sup_x\Ee_x[T(k)]\leq C_0(|k|_1+1)$.  No recurrence or moment condition on the nonzero
displacements is needed for these exact identities.

The family $q$ defines the transition kernel
\begin{equation*}
Q(x,\dd y\times\{v\})=q_v(x,\dd y)
\end{equation*}
from $\mathsf X$ to $\mathsf X\times\N^d$, and the embedded state kernel is
$P=\sum_vq_v$.  Given an initial law $\alpha$, the Ionescu--Tulcea theorem
\cite[Theorem~8.24]{Kallenberg2021}
constructs a process $(J_n,S_n)_{n\geq0}$ such that $J_0\sim\alpha$,
$S_0=0$, and, with $X_{n+1}=S_{n+1}-S_n$,
\begin{equation}
\Pp(J_{n+1}\in A,X_{n+1}=v\mid J_0,S_0,\ldots,J_n,S_n)
=q_v(J_n,A).
\label{eq:kernel-transition}
\end{equation}
All increments are nonnegative.  For $k\in\N^d$, set
\begin{equation*}
 T(k)=\inf\{n\geq1:S_n\notleqd k\},\qquad N(k)=T(k)-1.
\end{equation*}
Assumption~\ref{ass:zero-cluster} implies $T(k)<\infty$ almost surely: every
zero cluster is finite, and at most $|k|_1+1$ nonzero increments can occur
before the lower rectangle is left.  If
$\mathcal F_n=\sigma(J_0,S_0,\ldots,J_n,S_n)$, monotonicity gives
\begin{equation*}
 \{T(k)\leq n\}=\{S_n\notleqd k\}\in\mathcal F_n.
\end{equation*}
Thus $N(k)+1=T(k)$ is a stopping time.  In general $N(k)$ is not, because
$\{N(k)\leq n\}=\{S_{n+1}\notleqd k\}$ need not belong to $\mathcal F_n$.
For example, in the one-state i.i.d. model with $d=1$, $k=1$, and
$\Pp(X_1=1)=\Pp(X_1=2)=1/2$, the event $\{N(1)\leq0\}=\{X_1=2\}$ is not
measurable with respect to the trivial initial sigma-field.

\subsection{Multi-index convolution and the potential}

For kernel sequences $A=(A_k)$ and $B=(B_k)$ define
\begin{equation*}
(A*B)_k(x,C)
=\sum_{r\leqd k}\int_{\mathsf X}A_r(x,\dd y)B_{k-r}(y,C).
\end{equation*}
The identity $e$ is concentrated at zero and has coefficient $e_0=I$, the
identity kernel.  Associativity follows from Tonelli's theorem and the
bijection between triples $(a,b,c)$ satisfying $a+b+c=k$ under the two
parenthesizations.  Put $q^{(*0)}=e$ and
$q^{(*n)}=q*q^{(*(n-1))}$.

The next proposition gives the coefficient law, the two renewal equations,
and the reduction of zero-displacement clusters.

\Needspace{0.34\textheight}
\begin{proposition}
\label{prop:potential}
For $n\geq0$, $k\in\N^d$ and $A\in\mathcal X$,
\begin{equation}
q^{(*n)}_k(x,A)=\Pp_x(J_n\in A,S_n=k).
\label{eq:convolution-law}
\end{equation}
The coefficientwise series
\begin{equation*}
\psi_k=\sum_{n=0}^{\infty}q^{(*n)}_k
\end{equation*}
defines a finite kernel.  The sequence $\psi$ is the minimal
nonnegative solution of both renewal equations
\begin{equation}
\psi=e+q*\psi=e+\psi*q.
\label{eq:resolvent}
\end{equation}
For $v\neq0$, define
\begin{equation}
 \widehat q_v=R_0q_v,\qquad \widehat q_0=0,
 \label{eq:collapsed-kernel}
\end{equation}
where juxtaposition denotes kernel composition.  Then $\widehat q$ is a
strict Markov-renewal kernel.  If $\widehat\psi$ is its potential, then
\begin{equation*}
 \psi_k=\widehat\psi_kR_0,\qquad k\in\N^d.
\end{equation*}
\end{proposition}

The proof, including the zero-cluster factorization, is given in
Appendix~\ref{app:field-structure}.

For a measure $\alpha$ on $\mathsf X$, write
\begin{equation*}
(\alpha\psi_k)(A)=\int_{\mathsf X}\alpha(\dd x)\psi_k(x,A).
\end{equation*}
In finite state space, $q_v$ and $\psi_v$ are matrices and
\eqref{eq:resolvent} is precisely a noncommutative multi-index convolution
resolvent, with $R_0=(I-Q_0)^{-1}$.  No commutativity of the coefficient
matrices is used.

\subsection{The marked observation identity}

For $u\in\N^d$ define the rectangular survival function
\begin{equation}
\overline H_u(x)=\sum_{v\notleqd u}q_v(x,\mathsf X).
\label{eq:rect-survival}
\end{equation}
The event $v\notleqd u$ is the complement of a lower rectangle; it is not an
upper-orthant survival event when $d\geq2$.

The next identity links the potential to the full observed terminal edge.

\begin{proposition}
\label{thm:marked-potential}
Let $k,u,v\in\N^d$ with $u\leqd k$.  As a measure in $(x,y)$,
\begin{equation}
 \Pp_\alpha(U_k=u,\,Z_k\in\dd x,\, J_{N(k)+1}\in\dd y,\,X_{N(k)+1}=v) = \one_{\{v\notleqd u\}} (\alpha\psi_{k-u})(\dd x)q_v(x,\dd y).
 \label{eq:marked-potential}
\end{equation}
Consequently,
\begin{equation}
\Pp_\alpha(Z_k\in\dd x,U_k=u)
=(\alpha\psi_{k-u})(\dd x)\overline H_u(x),
\qquad u\leqd k.
\label{eq:age-potential}
\end{equation}
\end{proposition}

\begin{proof}
Fix measurable $A,B\subseteq\mathsf X$.  For $n\geq0$, the event
\begin{equation*}
\{S_n=k-u,\,J_n\in A,\,J_{n+1}\in B,\,X_{n+1}=v,
  \,v\notleqd u\}
\end{equation*}
implies $S_n\leqd k$ and
$S_{n+1}=k-u+v\notleqd k$.  Nonnegativity prevents every later $S_m$
from re-entering the lower rectangle.  Hence $N(k)=n$ and $U_k=u$.
Conversely, on the event in the left side of \eqref{eq:marked-potential},
putting $n=N(k)$ gives $S_n=k-u$ and $v\notleqd u$.  The events indexed by
$n$ are disjoint.  The Markov property at each deterministic $n$ therefore
gives
\begin{equation*}
 \Pp_\alpha(U_k=u,Z_k\in A,J_{N(k)+1}\in B,X_{N(k)+1}=v) =\one_{\{v\notleqd u\}} \sum_{n=0}^{\infty} \int_A (\alpha q^{(*n)}_{k-u})(\dd x)q_v(x,B),
\end{equation*}
and the sum is \eqref{eq:marked-potential}.  Integrate over $y$, sum over $v$, and use
\eqref{eq:rect-survival} to obtain \eqref{eq:age-potential}.
\end{proof}

Summing the marked identity gives the normalization of the observation law.

For every $k\in\N^d$,
\begin{equation*}
 \sum_{u\leqd k}\int_{\mathsf X}
 (\alpha\psi_{k-u})(\dd x)\overline H_u(x)=1,
\end{equation*}
because the events $\{U_k=u\}$ partition the sample space.

Multiplying the potential by the rectangular survival function gives the
joint law of the state and vector age at an arbitrary multi-time corner.
Retaining the next edge in
\eqref{eq:marked-potential} will be essential in
Section~\ref{sec:marked}.

The zero-displacement factorization also preserves the original transition
count when the cycle lengths are retained.  If $\tau_m$ is the index of the
$m$th positive edge, $L_m=\tau_m-\tau_{m-1}$, and $M(k)$ counts positive
edges inside the rectangle, then
\begin{equation*}
 N(k)=\sum_{m=1}^{M(k)+1}L_m-1.
\end{equation*}
Proposition~\ref{prop:positive-edge-embedding} in
Appendix~\ref{app:field-structure} gives the complete marked embedding and
its invariant-distribution and drift relations.  The resulting fluctuation
correction is stated in Corollary~\ref{cor:raw-count-zero-increments}.

\section{The augmented multi-time field}
\label{sec:augmentation}

The conditional future of $Z_k$ generally depends on the distance from the
current renewal point to $k$.  Adjoining the vector age gives a conditional
law for the entire future quadrant.  This law determines both the one-point
transitions and their joint dependence at incomparable corners.

Let
\begin{equation*}
\mathcal E=\{(x,u)\in\mathsf X\times\N^d:\overline H_u(x)>0\}.
\end{equation*}
For $(x,u)\in\mathcal E$, define the residual first-edge kernel
\begin{equation}
q^{\langle u\rangle}_v(x,\dd y)
=\frac{\one_{\{v\notleqd u\}}q_v(x,\dd y)}{\overline H_u(x)}.
\label{eq:residual-kernel}
\end{equation}

\subsection{Residual-start laws}

Start a Markov-renewal path at state $x$ from the virtual renewal location
$-u$.  Write $\widetilde J_0=x$ and
$R_j=-u+\sum_{\ell=1}^jV_\ell$ for its state and virtual location after
$j$ edges.  Draw its first state-increment pair from
\eqref{eq:residual-kernel}, and use the original kernel $q$ after that first
edge.  Observing this path at corners $h\in\N^d$ gives a field
$\mathcal Y^{x,u}=(\mathcal Y^{x,u}_h)_{h\in\N^d}$ with
$\mathcal Y^{x,u}_0=(x,u)$.

Let $\mathcal N_h^{x,u}$ be the number of new renewal edges included at the
residual corner $h$: if $V_j$ are the successive increments of the
residual-start path, then
\begin{equation*}
 \mathcal N_h^{x,u}
 =\sup\left\{j\geq0:-u+\sum_{\ell=1}^jV_\ell\leqd h\right\}.
\end{equation*}
Explicitly,
\begin{equation}
 \mathcal Y_h^{x,u}
 =\left(\widetilde J_{\mathcal N_h^{x,u}},
        h-R_{\mathcal N_h^{x,u}}\right).
 \label{eq:residual-field}
\end{equation}
The first-edge censoring implies $\mathcal N_0^{x,u}=0$.  Formula
\eqref{eq:residual-field} includes any later zero-displacement edges once
their location lies in the enlarged rectangle.  The inverse indices are
measurable, since they are determined by countably many coordinate
inequalities, and are finite almost surely by
Assumption~\ref{ass:zero-cluster}.  At each possible last index, the
deterministic-time argument in the proof of Proposition~\ref{thm:marked-potential}
shows that an observation with $\overline H_{h-R_j}(\widetilde J_j)=0$ has
probability zero.  This also applies at index zero, since the residual
first-edge law is absolutely continuous with respect to the original
first-edge law.  Hence the field takes values in $\mathcal E$ simultaneously
at all corners, whose index set is countable.  The measurable path
construction on the standard Borel space therefore gives a Markov kernel
\begin{equation*}
 (x,u)\longmapsto
 \mathsf R_{x,u}:=\mathcal L(\mathcal Y^{x,u})
 \quad\text{from }\mathcal E\text{ to }\mathcal E^{\N^d}.
\end{equation*}
We call it the residual future-quadrant kernel.  For $n\in\N$,
denote by $\mathsf R^N_{n,x,u}$ the law on
$(\N\times\mathcal E)^{\N^d}$ of
\begin{equation*}
 \bigl(n+\mathcal N_h^{x,u},\mathcal Y_h^{x,u}\bigr)_{h\in\N^d}.
\end{equation*}
This construction also defines a Markov kernel carrying the absolute renewal
count.

Write
\begin{equation*}
 \mathcal G_k^Y=\sigma\{Y_m=(Z_m,U_m):m\leqd k\},\qquad
 \mathcal G_k^{NY}=\sigma\{N(m),Z_m,U_m:m\leqd k\}.
\end{equation*}
Lemma~\ref{lem:censoring} in Appendix~\ref{app:field-structure} proves that,
conditional on the last included state and vector age, the lower-rectangle
past contributes only the censoring event for the first unseen edge.
This justifies the following conditional law without treating $N(k)$ as a
stopping time.

\begin{proposition}
\label{thm:future-quadrant}
For every $k\in\N^d$ and every bounded Borel function $F$ on
$\mathcal E^{\N^d}$,
\begin{equation}
\Ee_\alpha\left[
F\bigl((Y_{k+h})_{h\in\N^d}\bigr)\mid\mathcal G_k^Y
\right]
=\mathsf R_{Y_k}F
\quad\text{a.s.}
\label{eq:future-quadrant}
\end{equation}
The same identity holds under every initial distribution on $\mathcal E$ when
the field is constructed by mixing the residual-start laws.

The conditional law also retains the original transition count.
For every bounded Borel function $F$ on
$(\N\times\mathcal E)^{\N^d}$,
\begin{equation*}
 \Ee_\alpha\!\left[
 F\bigl((N(k+h),Y_{k+h})_{h\in\N^d}\bigr)
 \mid\mathcal G_k^{NY}\right]
 =\mathsf R^N_{N(k),Z_k,U_k}F
 \quad\text{a.s.}
\end{equation*}
\end{proposition}

The proof is given in Appendix~\ref{app:field-structure}.

\subsection{One-point semigroup and incomparable indices}

Let $K_t(y,B)$ be the law of $Y_t$ for a process initialized by
$J_0=y,S_0=0$.  The path-space construction and almost-sure finiteness of the
exit time make this kernel measurable in $y$.  Define
\begin{equation}
 P_h^Y((x,u),B) =\frac{\overline H_{u+h}(x)}{\overline H_u(x)} \one_B(x,u+h) + \frac{1}{\overline H_u(x)} \sum_{\substack{v\leqd u+h\\v\notleqd u}} \int_{\mathsf X}q_v(x,\dd y) K_{u+h-v}(y,B).
 \label{eq:augmented-transition}
\end{equation}
The two terms partition residual paths according to whether the first unseen
increment is still outside $[0,u+h]_d$ or enters the enlarged rectangle.

The one-point kernels form a commuting semigroup.

\begin{proposition}
\label{prop:one-point-semigroup}
The map in \eqref{eq:augmented-transition} is a Markov kernel on
$\mathcal E$.  For bounded measurable $f$,
\begin{equation}
\Ee[f(Y_{k+h})\mid\mathcal G_k^Y]=P_h^Yf(Y_k),
\label{eq:one-point-markov}
\end{equation}
and
\begin{equation}
P_0^Y=\mathrm{Id},\qquad
P_{h+r}^Y=P_h^YP_r^Y=P_r^YP_h^Y,\qquad h,r\in\N^d.
\label{eq:commuting-semigroup}
\end{equation}
\end{proposition}

See Appendix~\ref{app:field-structure} for the proof.

\begin{remark}
The semigroup in \eqref{eq:commuting-semigroup} determines transitions between
comparable indices.  Couplings at incomparable indices are instead specified
by the path-space kernel.  For example, on $\{0,1\}$ let $P$ have identical
Bernoulli$(1/2)$ rows and put $\widetilde P_0=I$ and
$\widetilde P_h=P$ for $h\neq0$.  An i.i.d. Bernoulli field and
the field $V_{(i,j)}=B_{i+j}$ have this same one-point semigroup
$(\widetilde P_h)$, although
their incomparable-index couplings differ.  The path-space kernel
$\mathsf R$ in Proposition~\ref{thm:future-quadrant} specifies the
future-quadrant law of the present model.
\end{remark}

\subsection{Strong lumpability of the state projection}

Let
\begin{equation*}
\Lambda_h((x,u),A)=P_h^Y((x,u),A\times\N^d).
\end{equation*}
The dependence of $\Lambda_h$ on $u$ determines whether the age coordinate
can be omitted while preserving a common transition law for $Z_k$.  Put
\begin{equation*}
 \mathsf X_{\mathcal E}=\{x\in\mathsf X:\text{there is }u\in\N^d
 \text{ such that }(x,u)\in\mathcal E\}.
\end{equation*}
Since $\overline H_u(x)\leq\overline H_0(x)$ for every $u$, this is the
Borel set $\{x:\overline H_0(x)>0\}$.  It contains every realized value of
$Z_k$, and $(x,0)\in\mathcal E$ for every $x\in\mathsf X_{\mathcal E}$.

The following proposition applies the classical
Markov-function, or strong-lumpability, criterion of
Burke and Rosenblatt~\cite{BurkeRosenblatt1958}, with one common transition
law required for every initial distribution.

\begin{proposition}
\label{thm:lumpability}
The following statements are equivalent.

\begin{enumerate}[label=\textup{(L\arabic*)},leftmargin=2.2em]
\item There is one Markov semigroup $(P_h^Z)_{h\in\N^d}$ on
$\mathsf X_{\mathcal E}$
such that, under every residual-start initial law on $\mathcal E$,
\begin{equation*}
\Ee[f(Z_{k+h})\mid\mathcal G_k^Z]=P_h^Zf(Z_k),
\qquad
\mathcal G_k^Z=\sigma\{Z_m:m\leqd k\}.
\end{equation*}

\item For every $h$, $x$, and all $u,u'$ with $(x,u),(x,u')\in\mathcal E$,
\begin{equation}
\Lambda_h((x,u),\cdot)=\Lambda_h((x,u'),\cdot).
\label{eq:lumpability}
\end{equation}
\end{enumerate}

When they hold,
$P_h^Z(x,A)=\Lambda_h((x,0),A)$, the coordinate kernels commute, and it is
enough to verify \eqref{eq:lumpability} for $h=e_1,\ldots,e_d$.
\end{proposition}

The proof is given in Appendix~\ref{app:field-structure}.

Markov additivity of the embedded skeleton does not by itself imply the
Markov property for the state projection $Z$.
Proposition~\ref{thm:lumpability} gives the exact condition under which the
vector age can be omitted from the common one-point transition law.

\section{Probabilistic input: functional and Brownian limits}
\label{sec:input}

Apart from the zero-cluster transience condition in
Assumption~\ref{ass:zero-cluster}, the structural identities of
Sections~\ref{sec:kernel} and \ref{sec:augmentation} require neither
recurrence of the embedded state chain nor moment assumptions.  The inverse
results require a functional limit for the original Markov-additive path.
This formulation separates the geometry of rectangular
inversion from any particular method of proving an invariance principle.

Fix an initial law and a vector
$\mu=(\mu_1,\ldots,\mu_d)\in(0,\infty)^d$.  For real $t>0$ and $s\geq0$ put
\begin{equation}
G_t(s)=\frac{S_{\lfloor ts\rfloor}-\lfloor ts\rfloor\mu}{a_t}.
\label{eq:def-Gt}
\end{equation}
Here and below a process on $[0,\infty)$ is viewed in
$D_{\mathrm{loc}}$, the space carrying the Skorokhod $J_1$ topology on every
compact time interval.
We use the following continuous-limit input throughout the inverse analysis.

\begin{assumption}
\label{ass:general-fclt}
There is a deterministic positive scale $a_t$ such that
\begin{equation}
a_t\longrightarrow\infty,
\qquad \frac{a_t}{t}\longrightarrow0,
\label{eq:scale}
\end{equation}
and a process $G$ with almost surely continuous paths for which
\begin{equation*}
G_t\Rightarrow G
\quad\text{in }D_{\mathrm{loc}}([0,\infty),\R^d).
\end{equation*}
\end{assumption}

No self-similarity, independent-increment property, or regular variation of
$a_t$ is part of Assumption~\ref{ass:general-fclt}.  The normalization in
\eqref{eq:def-Gt} is always $a_t$, even when the path is evaluated at time
$ts$.  This is the form needed by the deterministic inverse map in
Section~\ref{sec:stratified}.
Strict positivity of the drift coordinates makes each ratio
$\lambda_r/\mu_r$ finite at a positive corner and permits comparison of
all possible exiting faces on the same linear time scale.  Continuity of
the limit controls individual increments and the fluctuations over short
time intervals; these properties are used in the inverse expansions.
The assumptions do not require a nonsingular covariance matrix or a finite
state space.

Lemma~\ref{lem:fclt-consequences} in
Appendix~\ref{app:probabilistic-input} records the resulting uniform law of
large numbers, stochastic equicontinuity and negligible maximal increment.
These are the only discretization consequences used in the inverse proof.

\subsection{Joint transition rewards}

Let $g:\mathsf X\times\mathsf X\times\N^d\to\R^p$ be a measurable
transition reward, with $p\geq1$, and define
\begin{equation*}
A_n(g)=\sum_{j=1}^n g(J_{j-1},J_j,X_j),
\qquad A_0(g)=0.
\end{equation*}
For a deterministic vector $\bar g\in\R^p$, set
\begin{equation*}
H_t(s)=\frac{A_{\lfloor ts\rfloor}(g)-\lfloor ts\rfloor\bar g}{a_t}.
\end{equation*}

For transition rewards, we require the corresponding joint continuous
functional limit.

\begin{assumption}
\label{ass:joint-fclt}
Assumption~\ref{ass:general-fclt} holds and, for a process $H$ with almost
surely continuous paths,
\begin{equation*}
(G_t,H_t)\Rightarrow(G,H)
\quad\text{in }
D_{\mathrm{loc}}([0,\infty),\R^{d+p}).
\end{equation*}
\end{assumption}

The reward results later in the paper use Assumption~\ref{ass:joint-fclt},
whereas the inverse-count results use only its displacement marginal.  In
particular, the inverse theorem is not tied to Gaussian fluctuations.

At the central-limit scale, the preceding input takes the following Brownian
form.

\begin{assumption}
\label{ass:brownian-fclt}
Assumption~\ref{ass:joint-fclt} holds with $a_t=\sqrt t$ and
$(G,H)=(B_S,B_A)$, a centered Brownian motion in $\R^{d+p}$.  Thus, for a
symmetric nonnegative definite matrix
\begin{equation*}
\boldsymbol\Gamma=
\begin{pmatrix}
\Gamma&\Gamma_{S,A}\\
\Gamma_{A,S}&\Gamma_A
\end{pmatrix},
\end{equation*}
we have
\begin{equation*}
\Cov\left(
\begin{pmatrix}B_S(s)\\B_A(s)\end{pmatrix},
\begin{pmatrix}B_S(u)\\B_A(u)\end{pmatrix}
\right)
=(s\wedge u)\boldsymbol\Gamma.
\end{equation*}
\end{assumption}

The matrix $\boldsymbol\Gamma$ is the long-run covariance per original
transition, not the covariance of an arbitrarily chosen regenerative block.
The distinction matters for rewards, as Proposition~\ref{prop:regen-fclt}
in Appendix~\ref{app:probabilistic-input} makes explicit.

Propositions~\ref{prop:finite-state-fclt} and \ref{prop:regen-fclt} give
finite-state and general-state sufficient conditions in
Appendix~\ref{app:probabilistic-input}.  The latter treats rewards at
regeneration boundaries explicitly.  When the functional limit is proved
through positive edges, Corollary~\ref{cor:raw-count-zero-increments}
retains the fluctuation of the original transition count contributed by the
zero-cluster lengths.  These sufficient conditions are separate from the inverse
geometry and are not additional hypotheses of
Assumption~\ref{ass:general-fclt}.

\section{The stratified rectangular inverse}
\label{sec:stratified}

The global invariance principle describes rectangular inversion over the
entire direction space, including changes in the active coordinates.  We
first establish a deterministic inverse expansion and then apply the
probabilistic input of Section~\ref{sec:input}.

\subsection{Rate strata and the trace-enhanced state space}

Directions will be normalized to the open simplex
\begin{equation*}
 \Delta^\circ=\left\{\lambda\in(0,\infty)^d:
                   \sum_{r=1}^d\lambda_r=1\right\}.
\end{equation*}
The extension to arbitrary compact subsets of the positive orthant is obtained
by adjoining the radial coordinate; see Appendix~\ref{app:auxiliary}.  Recall
\begin{equation*}
 \rho(\lambda)=\min_r\frac{\lambda_r}{\mu_r},\qquad
 d_r(\lambda)=\frac{\lambda_r}{\mu_r}-\rho(\lambda),\qquad
 I(\lambda)=\{r:d_r(\lambda)=0\}.
\end{equation*}
For nonempty $I\subset[d]:=\{1,\ldots,d\}$ put
\begin{equation*}
 \mathcal D_I=\{\lambda\in\Delta^\circ:I(\lambda)=I\}.
\end{equation*}
These are relatively open polyhedral cells of different dimensions, not in
general faces of the simplex.  For the trace coordinates choose the compact
exhaustions
\begin{equation*}
 \mathcal D_{I,m}
 =\left\{\lambda\in\mathcal D_I:
       \min_r\lambda_r\ge m^{-1},\quad
       \min_{s\notin I}d_s(\lambda)\ge m^{-1}\right\},
\end{equation*}
where the last minimum is omitted when $I=[d]$.  Empty members are ignored.

For a locally bounded function $f$ on a metric space, write
\begin{equation*}
 f_\wedge(x)=\sup_{\varepsilon>0}
       \inf\{f(y):\operatorname{dist}(x,y)<\varepsilon\},\qquad
 f_\vee(x)=\inf_{\varepsilon>0}
       \sup\{f(y):\operatorname{dist}(x,y)<\varepsilon\}
\end{equation*}
for its lower and upper semicontinuous hulls.  Hypi convergence on the locally
compact space $\Delta^\circ$ means simultaneous Fell convergence of
$\operatorname{epi}(f_\wedge)$ and $\operatorname{hypo}(f_\vee)$ in
$\Delta^\circ\times\R$; see B{\"u}cher, Segers and Volgushev~\cite{BucherSegersVolgushev2014}.  We use the
global hypi semimetric $d_{\mathrm{hypi}}^\Delta$ defined in
Appendix~\ref{app:auxiliary}.  It acts on representatives of locally bounded
functions; no quotient is taken.

We use the following trace-enhanced notion of stratified hypi convergence.

\begin{definition}
\label{def:stratified-topology}
A sequence of locally bounded functions $f_n:\Delta^\circ\to\R$ converges
to $f$ in the \emph{trace-enhanced stratified hypi topology}, written
$f_n\to_{\mathrm{str}}f$, when
\begin{enumerate}[label=(\roman*)]
\item $f_n\to f$ in the global hypi topology on $\Delta^\circ$;
\item for every nonempty $I\subset[d]$ and every $m$,
\begin{equation*}
 \sup_{\lambda\in\mathcal D_{I,m}}
       |f_n(\lambda)-f(\lambda)|\longrightarrow0.
\end{equation*}
\end{enumerate}
Equivalently, this is the topology generated by the global hypi semimetric
and all fixed-stratum restrictions.  The trace coordinates separate any two
function representatives because the strata partition $\Delta^\circ$; their weighted
sum with the semimetric is therefore a genuine metric
$d_{\mathrm{str}}$, given explicitly in Appendix~\ref{app:auxiliary}.
\end{definition}

Ordinary hypi convergence records the two completed envelopes but not the
value assigned on a discontinuity set.  The second coordinate in
Definition~\ref{def:stratified-topology} retains precisely that information.
There is no conflict between the coordinates: all traces are restrictions of
the same prelimit field.  The image generated below by continuous driving
paths is a separable subspace, since it is a continuous image of a countable
product of separable spaces of continuous functions.  Thus its random limits
are Borel measurable.  For the full function space we keep the outer
probability convention stated in the introduction.

The next lemma identifies the limiting active coordinates and their
semicontinuous envelopes.

\begin{lemma}
\label{lem:active-collapse}
Let $c_n\to\infty$, and let $u_{n,r},u_r:\Delta^\circ\to\R$ be
locally bounded functions satisfying $u_{n,r}\to u_r$ locally uniformly,
with every $u_r$ continuous.  Define
\begin{equation}
 F_n(\lambda)=\min_{1\le r\le d}
       \{c_nd_r(\lambda)+u_{n,r}(\lambda)\},
 \qquad
 F(\lambda)=\min_{r\in I(\lambda)}u_r(\lambda).
 \label{eq:active-collapse}
\end{equation}
Then $F_n\to_{\mathrm{str}}F$.  Moreover, at every
$\lambda\in\Delta^\circ$,
\begin{equation}
 F_\wedge(\lambda)=\min_{r\in I(\lambda)}u_r(\lambda),
 \qquad
 F_\vee(\lambda)=\max_{r\in I(\lambda)}u_r(\lambda).
 \label{eq:active-hulls}
\end{equation}
\end{lemma}

\begin{proof}
Fix a compact subset $C$ of $\mathcal D_I$.  The inactive gaps have a
positive minimum on $C$.  Hence their penalized terms tend uniformly to
$+\infty$, while the terms indexed by $I$ converge uniformly.  This proves
the trace assertion.

We verify the two global sequential conditions for epi/hypograph convergence
on $\Delta^\circ$.  Suppose $\lambda_n\to\lambda$ and pass to a subsequence along
which a minimizing index in \eqref{eq:active-collapse} is constant, say $r$.
If $r\notin I(\lambda)$, continuity gives
$c_nd_r(\lambda_n)\to\infty$.  Otherwise the penalty is nonnegative.  Thus
\begin{equation*}
 \liminf_nF_n(\lambda_n)
 \ge \min_{s\in I(\lambda)}u_s(\lambda).
\end{equation*}
Taking the constant recovery sequence $\lambda_n=\lambda$ proves recovery
for the lower envelope.

For the upper bound, choose $r_n\in I(\lambda_n)$.  Along a subsequence
$r_n=r\in I(\lambda)$, and
\begin{equation*}
 \limsup_nF_n(\lambda_n)
 \le \lim_nu_{n,r}(\lambda_n)
 \le\max_{s\in I(\lambda)}u_s(\lambda).
\end{equation*}
It remains to recover the upper endpoint.  Let $r\in I(\lambda)$ maximize
$u_s(\lambda)$ over $I(\lambda)$.  If $I(\lambda)=\{r\}$ (in particular,
if $d=1$), the constant recovery sequence already suffices.  Otherwise
$d\geq2$; define a tangent vector $w$ by
\begin{equation*}
 \frac{w_r}{\mu_r}=-1,\qquad
 \frac{w_s}{\mu_s}=
 \frac{\mu_r}{\sum_{j\ne r}\mu_j}\quad(s\ne r).
\end{equation*}
Then $\sum_sw_s=0$, and for every sufficiently small $\varepsilon>0$,
$\lambda+\varepsilon w\in\Delta^\circ$ and $r$ is its unique active
coordinate.  Choose $\varepsilon_n\downarrow0$ with
$c_n\varepsilon_n\to\infty$.  At
$\lambda_n=\lambda+\varepsilon_nw$, every competing penalty diverges and
$F_n(\lambda_n)\to u_r(\lambda)$.  This is the hypograph recovery sequence.
The recovery points remain in $\Delta^\circ$.  The sequential characterization
of Fell convergence now gives global hypi
convergence, and the two recovery calculations give
\eqref{eq:active-hulls}.
\end{proof}

The tangent perturbation in the preceding proof accounts for unequal
coordinate drifts while remaining within the simplex.

\Needspace{0.32\textheight}
\subsection{A uniform first-order expansion for lower-rectangle inversion}

Put
\begin{equation*}
 \rho_{\max}=\sup_{\lambda\in\Delta^\circ}\rho(\lambda)
 =\frac1{\sum_{r=1}^d\mu_r},
\end{equation*}
and fix $T>\rho_{\max}$.  For a continuous coordinatewise nondecreasing path
$x:[0,T]\to\R^d$ and a threshold map
$\eta:\Delta^\circ\to\R^d$, set, with $\sup\varnothing=0$,
\begin{equation*}
 \mathcal R(x,\eta)(\lambda)
 =\sup\{s\in[0,T]:x(s)\leqd\eta(\lambda)\}.
\end{equation*}
At the base point $x^0(s)=s\mu$ and $\eta^0(\lambda)=\lambda$, this equals
$\rho(\lambda)$.

The next proposition gives a sequential first-order expansion in the
trace-enhanced stratified metric.  Its conclusion is uniform on compact
subsets of every fixed stratum and simultaneously records the global hypi
limit across changes of active set.

\begin{proposition}
\label{thm:inverse-expansion}
Let $\varepsilon_n\downarrow0$, let
\begin{equation*}
 x_n(s)=s\mu+\varepsilon_ng_n(s),\qquad
 \eta_n(\lambda)=\lambda+\varepsilon_nb_n(\lambda),
\end{equation*}
where $x_n$ is coordinatewise nondecreasing,
$g_n\in C([0,T],\R^d)$, $g_n\to g$ uniformly on $[0,T]$, and the locally
bounded maps $b_n$ converge to $b$ locally uniformly on $\Delta^\circ$.
Assume that $g$ and $b$ are continuous.  Then
\begin{equation*}
 d_{\mathrm{str}}\!\left(
 \frac{\mathcal R(x_n,\eta_n)-\rho}{\varepsilon_n},
 \Phi_{g,b}\right)\longrightarrow0,
 \qquad
 \Phi_{g,b}(\lambda)
 =\min_{r\in I(\lambda)}
       \frac{b_r(\lambda)-g_r(\rho(\lambda))}{\mu_r}.
\end{equation*}
\end{proposition}

\begin{proof}
Fix a compact $K\subset\Delta^\circ$ and write
$\mu_* =\min_r\mu_r$.  For a continuous vector-valued function $h$ on
$[0,T]$, set
\begin{equation*}
 \omega_h(\delta)=\sup_{\substack{s,u\in[0,T]\\|s-u|\leq\delta}}
                    \|h(s)-h(u)\|_\infty.
\end{equation*}
Define
\begin{equation*}
 a_{n,r}(\lambda)=\frac{d_r(\lambda)}{\varepsilon_n}
       +\frac{b_{n,r}(\lambda)-g_{n,r}(\rho(\lambda))}{\mu_r},
 \qquad A_n(\lambda)=\min_ra_{n,r}(\lambda).
\end{equation*}
Nonnegative gaps, at least one of which is zero, give
\begin{equation*}
 \sup_{\lambda\in K}|A_n(\lambda)|
 \leq\mu_*^{-1}\bigl(\|b_n\|_{\infty,K}
                         +\|g_n\|_{\infty,[0,T]}\bigr).
\end{equation*}
The right side is bounded for all sufficiently large $n$.  Fix $M$ exceeding
that bound and put $L=M+2$.  Since $\inf_K\rho>0$ and $\sup_K\rho<T$,
all times $\rho(\lambda)+\varepsilon_nz$ with $\lambda\in K$ and
$|z|\leq L$ lie in $[0,T]$ for large $n$.  Direct subtraction gives
\begin{equation}
 \frac{x_{n,r}(\rho(\lambda)+\varepsilon_nz)-\eta_{n,r}(\lambda)}
      {\varepsilon_n\mu_r}
 =z-a_{n,r}(\lambda)
  +\frac{g_{n,r}(\rho(\lambda)+\varepsilon_nz)
              -g_{n,r}(\rho(\lambda))}{\mu_r}.
 \label{eq:det-sandwich-expansion}
\end{equation}
The last term is bounded in absolute value by
$e_n=\mu_*^{-1}\omega_{g_n}(\varepsilon_nL)$, uniformly in the displayed
parameters.  Uniform convergence $g_n\to g$ and continuity of $g$ imply
$e_n\to0$.  Put $\delta_n=2e_n+\varepsilon_n$; eventually $\delta_n<1$.
At $z=A_n(\lambda)-\delta_n$, every coordinate in
\eqref{eq:det-sandwich-expansion} is strictly negative.  At
$z=A_n(\lambda)+\delta_n$, a minimizing coordinate is strictly positive.
Monotonicity of $x_n$ therefore gives
\begin{equation*}
 \rho(\lambda)+\varepsilon_n(A_n(\lambda)-\delta_n)
 \leq\mathcal R(x_n,\eta_n)(\lambda)
 \leq\rho(\lambda)+\varepsilon_n(A_n(\lambda)+\delta_n).
\end{equation*}
Consequently the normalized inverse differs from $A_n$ by at most
$\delta_n$, uniformly on $K$.  This argument applies to every compact
$K\subset\Delta^\circ$.

Lemma~\ref{lem:active-collapse}, with $c_n=\varepsilon_n^{-1}$ and
$u_{n,r}=(b_{n,r}-g_{n,r}\circ\rho)/\mu_r$, gives
$A_n\to_{\mathrm{str}}\Phi_{g,b}$.  The local-uniform perturbation
lemma, Lemma~\ref{lem:stratified-local-perturbation}, transfers this
convergence to the normalized inverse and proves the assertion.
\end{proof}

\begin{remark}
The target topology is part of the proposition.  Under the nonzero-contrast
hypothesis of Proposition~\ref{thm:global-nontight}, a transverse tie
precludes asymptotic tightness in the supremum norm.  The first-order
expansion is retained in ordinary hypi topology, but that topology does not
in general identify its exact value on $\mathcal D_I$, $|I|>1$.
The trace enhancement retains these exact stratum values together with
the global hypi limit.
\end{remark}

\subsection{The global invariance principle}

Let $a_t$ be the scale in Assumption~\ref{ass:general-fclt}.  For a
deterministic, locally bounded Borel threshold field
$k_t:\Delta^\circ\to\N^d$ define
\begin{equation*}
 b_t(\lambda)=\frac{k_t(\lambda)-t\lambda}{a_t},\qquad
 X_t(\lambda)=\frac{N(k_t(\lambda))-t\rho(\lambda)}{a_t}.
\end{equation*}

The following global stratified invariance principle transfers the path
limit to the entire directional inverse field.

\begin{theorem}
\label{thm:stratified-fclt}
Under Assumption~\ref{ass:general-fclt}, let $b_t\to b$ locally uniformly
on $\Delta^\circ$, with $b$ continuous.  For every compact
$K\subset\Delta^\circ$, the following uniform expansion holds:
\begin{equation}
 \sup_{\lambda\in K}\left|X_t(\lambda)
 -\min_{1\leq r\leq d}\left\{
 \frac{t}{a_t}d_r(\lambda)
 +\frac{b_{t,r}(\lambda)-G_{t,r}(\rho(\lambda))}{\mu_r}
 \right\}\right|
 \xrightarrow{\Pp}0.
 \label{eq:finite-scale-uniform-inverse}
\end{equation}
Then
\begin{equation}
 X_t\Rightarrow_{\mathrm{str}}\Phi_{G,b},
 \qquad
 \Phi_{G,b}(\lambda)=
 \min_{r\in I(\lambda)}
 \frac{b_r(\lambda)-G_r(\rho(\lambda))}{\mu_r}.
 \label{eq:global-stratified-fclt}
\end{equation}
More explicitly, the convergence is global hypi on $\Delta^\circ$
and is locally uniform on every $\mathcal D_I$.  Conditional on a realization of
$G$, its completed value set at $\lambda$ is
\begin{equation}
 \left[
 \min_{r\in I(\lambda)}
       \frac{b_r(\lambda)-G_r(\rho(\lambda))}{\mu_r},
 \ 
 \max_{r\in I(\lambda)}
       \frac{b_r(\lambda)-G_r(\rho(\lambda))}{\mu_r}
 \right].
 \label{eq:completed-fiber}
\end{equation}
At a fixed $\lambda^0$, with $I_0=I(\lambda^0)$ and
$\rho_0=\rho(\lambda^0)$, the critical chart satisfies, for every compact
$H\subset\R^d$,
\begin{equation}
 \left(
 \frac{N(\lfloor t\lambda^0+a_th\rfloor)-t\rho_0}{a_t}
 \right)_{h\in H}
 \Rightarrow
 \left(
 \min_{r\in I_0}\frac{h_r-G_r(\rho_0)}{\mu_r}
 \right)_{h\in H}
 \quad\text{in }\ell^\infty(H).
 \label{eq:critical-chart-general}
\end{equation}
\end{theorem}

\begin{proof}
Let $\widetilde S$ be the linear interpolation of $(S_n)$ and set
\begin{equation*}
 \widehat G_t(s)=\frac{\widetilde S(ts)-ts\mu}{a_t},
 \qquad \varepsilon_t=\frac{a_t}{t},\qquad
 F_t(\lambda)=\min_{r\leq d}
 \left\{\varepsilon_t^{-1}d_r(\lambda)
       +\frac{b_{t,r}(\lambda)-G_{t,r}(\rho(\lambda))}{\mu_r}\right\}.
\end{equation*}
By Lemma~\ref{lem:fclt-consequences}, for every finite $T$,
\begin{equation}
 \|\widehat G_t-G_t\|_{\infty,[0,T]}
 \leq a_t^{-1}\max_{1\leq n\leq\lceil tT\rceil}
                         \|X_n-\mu\|_\infty
 \xrightarrow{\Pp}0.
 \label{eq:interpolated-path-comparison}
\end{equation}
Thus $\widehat G_t\Rightarrow G$ in $C([0,T],\R^d)$; in particular its
supremum norm is tight and
$\omega_{\widehat G_t}(\eta_t)\to0$ in probability whenever
$\eta_t\downarrow0$.

Fix a compact $K\subset\Delta^\circ$ and $T>\rho_{\max}$.
Use the notation $\mu_*$ and $\omega$ from the preceding proof.  On the event
\begin{equation*}
 \mu_*^{-1}\bigl(\|b_t\|_{\infty,K}
                      +\|\widehat G_t\|_{\infty,[0,T]}\bigr)\leq M,
\end{equation*}
the active minimum formed with $\widehat G_t$ is bounded by $M$.  Put
$L=M+2$ and
$\delta_t=2\mu_*^{-1}\omega_{\widehat G_t}(\varepsilon_tL)+a_t^{-1}$.
For fixed $M$, eventually
$\varepsilon_tL<\min\{\inf_K\rho,T-\sup_K\rho\}$, and
$\Pp(\delta_t>1)\to0$.  The two pathwise comparisons in
\eqref{eq:det-sandwich-expansion} apply with this positive buffer
$\delta_t$.  They locate the interpolated inverse strictly before $T$ and
bound its normalized error by $\delta_t$ uniformly on $K$.
Lemma~\ref{lem:interpolation-error} and
\eqref{eq:interpolated-path-comparison} now give, on these events,
\begin{equation*}
 \sup_{\lambda\in K}|X_t(\lambda)-F_t(\lambda)|
 \leq\delta_t+a_t^{-1}
        +\mu_*^{-1}\|\widehat G_t-G_t\|_{\infty,[0,T]}.
\end{equation*}
The right side tends to zero in probability for fixed $M$.  Tightness of
$\|\widehat G_t\|_\infty$ and the deterministic local boundedness of $b_t$
allow $M\to\infty$, proving
\begin{equation}
 \sup_{\lambda\in K}|X_t(\lambda)-F_t(\lambda)|
       \xrightarrow{\Pp}0.
 \label{eq:local-uniform-XF}
\end{equation}
This is \eqref{eq:finite-scale-uniform-inverse}.

On a countable compact exhaustion of $\Delta^\circ$, a subsequence
argument turns \eqref{eq:local-uniform-XF} into simultaneous almost-sure
local uniform approximation along a further subsequence.
Lemma~\ref{lem:stratified-local-perturbation} then yields
\begin{equation}
 d_{\mathrm{str}}(X_t,F_t)\xrightarrow{\Pp}0.
 \label{eq:dstr-XF}
\end{equation}
For any sequence $t_n\to\infty$, the Skorokhod representation theorem in
$D_{\mathrm{loc}}$, together with continuity of $G$, gives versions for
which $G_{t_n}\to G$ locally uniformly almost surely.
Lemma~\ref{lem:active-collapse} applies on this representation with
$c_n=t_n/a_{t_n}$ and
$u_{n,r}=(b_{t_n,r}-G_{t_n,r}\circ\rho)/\mu_r$.  Hence
$F_{t_n}\to\Phi_{G,b}$ almost surely in $d_{\mathrm{str}}$.
The fields involved are Borel random elements with separable ranges, and
the limiting law is tight, as established below in
Appendix~\ref{app:auxiliary}.  The representation and
\eqref{eq:dstr-XF} therefore give
\eqref{eq:global-stratified-fclt}.  Formula
\eqref{eq:completed-fiber} follows from the two hulls in
Lemma~\ref{lem:active-collapse}.

Finally, let $H\subset\R^d$ be compact.  Apply the same pathwise comparisons
at the fixed time $\rho_0$, with thresholds
$\lfloor t\lambda^0+a_th\rfloor$.  The rounding term is bounded by
$a_t^{-1}$ in each normalized coordinate.  The inactive gaps are fixed
positive constants multiplied by $t/a_t$, whereas $h$ is bounded and
$G_t(\rho_0)$ is tight.  They are therefore absent from the minimum with
probability tending to one, uniformly in $h\in H$.  The modulus bound gives
\begin{equation*}
 \sup_{h\in H}\left|
 \frac{N(\lfloor t\lambda^0+a_th\rfloor)-t\rho_0}{a_t}
 -\min_{r\in I_0}\frac{h_r-G_{t,r}(\rho_0)}{\mu_r}
 \right|\xrightarrow{\Pp}0.
\end{equation*}
Evaluation at $\rho_0$ is continuous at continuous paths, and the map
$z\mapsto(h\mapsto\min_{r\in I_0}(h_r-z_r)/\mu_r)$ is Lipschitz
from $\R^d$ into $C(H)$.  These facts prove
\eqref{eq:critical-chart-general}.
\end{proof}

The finite-scale expansion identifies the critical scale at which
constraints can compete.  On a compact set of directions, the fluctuation terms
are bounded in probability, whereas coordinate $r$ carries the deterministic
penalty $(t/a_t)d_r(\lambda)$.  A coordinate whose gap is much larger than
$a_t/t$ consequently drops out of the minimum with probability tending to
one.  Gaps of order $a_t/t$ can remain comparable with the path fluctuation and
produce the critical profiles in \eqref{eq:critical-chart-general}.  Thus
the stratification records both the limiting active constraints and the
scale at which a small change of horizon can change the exiting coordinate.
This scale is also the one on which allocation corrections and uncertainty
in the estimated drift must be compared.

The joint continuous functional limit also yields the associated reward fields.

\begin{corollary}
\label{cor:brownian-reward-field}
Under Assumption~\ref{ass:joint-fclt}, let $b_t\to b$ locally uniformly
on $\Delta^\circ$, with $b$ continuous.  Let $A_n$ and $H$ be the reward sum and limiting reward component
in that assumption.  Then the following convergences hold in
the product of $p$ copies of the trace-enhanced metric $d_{\mathrm{str}}$,
and hence coordinatewise locally uniformly on every stratum and globally in
the hypi topology:
\begin{eqnarray}
 \frac{A_{N(k_t(\lambda))}-t\rho(\lambda)\bar g}{a_t}
 &\Rightarrow& H(\rho(\lambda))+\bar g\,\Phi_{G,b}(\lambda),
 \label{eq:reward-deterministic-centering}\\
 \frac{A_{N(k_t(\lambda))}-\bar gN(k_t(\lambda))}{a_t}
 &\Rightarrow& H(\rho(\lambda)).
 \label{eq:reward-random-centering}
\end{eqnarray}
The assertions hold jointly with the inverse field.
\end{corollary}

\begin{proof}
The exact decomposition
\begin{equation*}
 \frac{A_{N(k_t)}-t\rho\bar g}{a_t}
 =H_t(N(k_t)/t)+
   \bar g\frac{N(k_t)-t\rho}{a_t}
\end{equation*}
reduces \eqref{eq:reward-deterministic-centering} to
Theorem~\ref{thm:stratified-fclt}.  Applying the same decomposition after
subtracting $\bar gN(k_t)$ proves \eqref{eq:reward-random-centering}.  The
inverse localization in the theorem's proof
gives $N(k_t)/t-\rho=O_{\Pp}(a_t/t)$ locally uniformly.  Stochastic
equicontinuity of $H_t$ therefore permits replacement of the random time by
$\rho(\lambda)$.  Addition of a continuous function and multiplication by a
fixed scalar are continuous for the hypi coordinate and for every uniform
trace coordinate.  A negative scalar exchanges the epigraph and hypograph
coordinates.  Applying these operations to each reward coordinate proves
the stated product-topology convergence, jointly with the inverse field.
\end{proof}

\subsection{Normal fans and completed fibers}

Formula \eqref{eq:completed-fiber} has a direct geometric interpretation.
Fix $\lambda^0\in\Delta^\circ$ and put
$I_0=I(\lambda^0)$, $\rho_0=\rho(\lambda^0)$.  For a deterministic continuous
path $g$ and a continuous threshold perturbation $b$, define
$a_r=\{b_r(\lambda^0)-g_r(\rho_0)\}/\mu_r$ and, on the tangent space
$T\Delta=\{h:\sum_rh_r=0\}$,
\begin{equation}
 \Psi_{\lambda^0,g,b}(h)=
 \min_{r\in I_0}
 \left\{a_r+\frac{h_r}{\mu_r}
       -\min_{s\in I_0}\frac{h_s}{\mu_s}\right\}.
 \label{eq:normal-fan}
\end{equation}

The next proposition shows that normal blow-ups fill the completed graph.

\begin{proposition}
\label{prop:normal-fiber}
Under the hypotheses of Proposition~\ref{thm:inverse-expansion}, if
$\lambda_n=\lambda^0+\varepsilon_nh_n\in\Delta^\circ$ with
$h_n\to h\in T\Delta$, then
\begin{equation}
 \frac{\mathcal R(x_n,\eta_n)(\lambda_n)-\rho(\lambda_n)}
      {\varepsilon_n}
 \longrightarrow\Psi_{\lambda^0,g,b}(h).
 \label{eq:normal-blow-up-limit}
\end{equation}
Moreover
\begin{equation*}
 \overline{\{\Psi_{\lambda^0,g,b}(h):h\in T\Delta\}}
 =\left[\min_{r\in I_0}a_r,\max_{r\in I_0}a_r\right],
\end{equation*}
which is exactly the vertical fiber of the completed hypi graph at
$\lambda^0$.
\end{proposition}

\begin{proof}
For $r\in I_0$,
\begin{equation*}
 \frac{d_r(\lambda^0+\varepsilon_nh_n)}{\varepsilon_n}
 \longrightarrow
 \frac{h_r}{\mu_r}-\min_{s\in I_0}\frac{h_s}{\mu_s},
\end{equation*}
whereas all coordinates outside $I_0$ have diverging penalties.  Substitution
in the inverse bounds proved in Proposition~\ref{thm:inverse-expansion}, together with
$b_n(\lambda_n)\to b(\lambda^0)$ and
$g_n(\rho(\lambda_n))\to g(\rho_0)$, yields
\eqref{eq:normal-blow-up-limit}, whose right-hand side is
\eqref{eq:normal-fan}.

Every summand in \eqref{eq:normal-fan} has a nonnegative penalty and at
least one penalty is zero, so its value lies between $\min_{r\in I_0}a_r$ and
$\max_{r\in I_0}a_r$.  Conversely fix $r\in I_0$ and prescribe, modulo a common
constant, $h_r/\mu_r=0$ and $h_s/\mu_s=L$ for
$s\in I_0\setminus\{r\}$.  The unused coordinates, or a common shift when
$I_0=[d]$, make $\sum_sh_s=0$ without changing the displayed differences.
As $L\uparrow\infty$, the value tends to $a_r$.  Taking $r$ with maximal
$a_r$ and varying $L$ continuously from zero gives every value between the
minimum and maximum by the intermediate value theorem.  Finally,
Lemma~\ref{lem:active-collapse} identifies these endpoints with the lower and
upper hypi hulls.
\end{proof}

The upper hull arises as the limit along normal approaches to the tie at the
same order as the stochastic fluctuation.

\subsection{Coordinate contrasts and uniform tightness}

Nondegenerate fluctuations across a tie can prevent tightness in the uniform
topology.  To make this precise, let $u\mapsto\lambda(u)\in\Delta^\circ$
be continuous for $0\le u\le u_0$ and satisfy
\begin{equation}
 I(\lambda(0))=\{1,2\},\quad I(\lambda(u))=\{1\}\ (u>0),
 \quad d_2(\lambda(u))=cu+o(u),\quad c>0,
 \label{eq:transverse-path}
\end{equation}
while all remaining gaps stay positive near zero.

The following proposition gives the corresponding failure of tightness.

\begin{proposition}
\label{thm:global-nontight}
Under Assumption~\ref{ass:general-fclt} with $a_t=\sqrt t$ and
$G=B_S$ a centered Brownian motion with covariance matrix $\Gamma$ per unit
time, let the path satisfy
\eqref{eq:transverse-path} and assume
\begin{equation*}
 \Var\left(
 \frac{B_{S,1}(\rho_0)}{\mu_1}
 -\frac{B_{S,2}(\rho_0)}{\mu_2}\right)>0,
 \qquad \rho_0=\rho(\lambda(0)).
\end{equation*}
Then
\begin{equation*}
 u\longmapsto
 \frac{N(\lfloor t\lambda(u)\rfloor)-t\rho(\lambda(u))}{\sqrt t}
\end{equation*}
is not asymptotically tight in $\ell^\infty([0,u_0])$.
\end{proposition}

\begin{proof}
Write $X_t$ for the field in the statement.  Let $\widehat G_t$ be the
interpolated centered path used in the proof of
Theorem~\ref{thm:stratified-fclt}, and define
\begin{equation*}
 Y_t(u)=\min_{r\leq d}\left\{\sqrt t\,d_r(\lambda(u))
       -\frac{\widehat G_{t,r}(\rho(\lambda(u)))}{\mu_r}\right\}.
\end{equation*}
Each $Y_t$ is a Borel random element of $C([0,u_0])$: the interpolation
depends continuously on finitely many increments over the relevant time
interval.  The compact-uniform inverse expansion,
\eqref{eq:interpolated-path-comparison}, and the uniformly
$O(t^{-1/2})$ rounding error give
$\|X_t-Y_t\|_\infty\to0$ in outer probability.

If $X_t$ were asymptotically tight in $\ell^\infty([0,u_0])$, asymptotic
equivalence would make $Y_t$ asymptotically tight there.  Since
$C([0,u_0])$ is a closed separable subspace, the Borel laws of $Y_t$ would
then have a weakly convergent subsequence with a continuous limit $Y$.
Theorem~\ref{thm:stratified-fclt} identifies its finite-dimensional laws.
Writing
\begin{equation*}
 A=-\frac{B_{S,1}(\rho_0)}{\mu_1},\qquad
 B=-\frac{B_{S,2}(\rho_0)}{\mu_2},\qquad
 A(u)=-\frac{B_{S,1}(\rho(\lambda(u)))}{\mu_1},
\end{equation*}
we obtain, for every $u>0$,
$(Y(u),Y(0))\stackrel{d}{=}(A(u),\min(A,B))$.
Choose $u_j\downarrow0$.  Continuity of $Y$ gives
$Y(u_j)-Y(0)\to0$ almost surely, whereas Brownian continuity gives
$A(u_j)-\min(A,B)\to A-\min(A,B)$ almost surely.
Equality of the corresponding laws forces $A\leq B$ almost surely.
The assumed positive contrast variance makes $A-B$ a nondegenerate centered
Gaussian variable, for which $\Pp(A>B)=1/2$.  This contradiction proves
the assertion.
\end{proof}

The preceding result gives an exact criterion for the adequacy of the
uniform topology over the whole direction space.  It shows that uniform
tightness is exceptional: all coordinates must have the same normalized
Brownian fluctuation.

\begin{proposition}
\label{prop:uniform-tightness-dichotomy}
Under the displacement Brownian limit in Proposition~\ref{thm:global-nontight}, define
\begin{equation*}
 X_t^0(\lambda)
 =\frac{N(\lfloor t\lambda\rfloor)-t\rho(\lambda)}{\sqrt t},
 \qquad \lambda\in\Delta^\circ.
\end{equation*}
The following statements are equivalent:
\begin{enumerate}[label=\textup{(\roman*)}]
\item For every compact $K\subset\Delta^\circ$, the restrictions
$X_t^0|_K$ are asymptotically tight in $\ell^\infty(K)$.
\item For every $r,s\in[d]$,
\begin{equation*}
 \left(\frac{e_r}{\mu_r}-\frac{e_s}{\mu_s}\right)^{\!\mathsf T}
 \Gamma
 \left(\frac{e_r}{\mu_r}-\frac{e_s}{\mu_s}\right)=0.
\end{equation*}
\item There is a constant $\sigma\geq0$ such that
$\Gamma=\sigma^2\mu\mu^{\mathsf T}$.
\end{enumerate}
When these conditions hold, for a standard one-dimensional Brownian motion
$W$ and every compact $K\subset\Delta^\circ$,
\begin{equation*}
 X_t^0|_K\Rightarrow
 \bigl(-\sigma W(\rho(\lambda))\bigr)_{\lambda\in K}
 \quad\text{in }\ell^\infty(K).
\end{equation*}
If they fail, there is a compact transverse curve meeting a two-coordinate
tie on which asymptotic tightness in the supremum norm fails.
\end{proposition}

\begin{proof}
Let $D=\operatorname{diag}(\mu_1,\ldots,\mu_d)$ and
$C=D^{-1}\Gamma D^{-1}$.  Since $C$ is nonnegative definite, condition
\textup{(ii)} implies $C(e_r-e_s)=0$ for every $r,s$.  The vectors
$e_r-e_s$ span the orthogonal complement of $\one$, so symmetry gives
$C=\sigma^2\one\one^{\mathsf T}$ for some $\sigma\geq0$.
This proves \textup{(iii)}; the converse is immediate.  This argument also
covers $d=1$, when the condition on contrasts is vacuous.

Assume \textup{(iii)}.  The continuous Gaussian process $B_S$ then has
the same law as $(\sigma\mu_rW)_{r\leq d}$.  Equivalently, all normalized
coordinate processes agree almost surely, simultaneously at every time:
equality at rational times follows from their zero contrast variances, and
continuity extends it to all times.  In a Skorokhod representation, write
$G_t\to B_S$ uniformly on $[0,T]$, where $T>\rho_{\max}$.  For a compact
$K\subset\Delta^\circ$, put
\begin{equation*}
 F_t^0(\lambda)=\min_{r\leq d}
 \left\{\sqrt t\,d_r(\lambda)
 +\frac{b_{t,r}(\lambda)-G_{t,r}(\rho(\lambda))}{\mu_r}\right\},
 \qquad
 b_t(\lambda)=\frac{\lfloor t\lambda\rfloor-t\lambda}{\sqrt t}.
\end{equation*}
The identity $\min_rd_r(\lambda)=0$ and the Lipschitz property of a
finite minimum imply
\begin{equation*}
 \sup_{\lambda\in K}
 |F_t^0(\lambda)+\sigma W(\rho(\lambda))|
 \leq\max_{r\leq d}\mu_r^{-1}
 \left\{t^{-1/2}+\sup_{0\leq s\leq T}
 |G_{t,r}(s)-B_{S,r}(s)|\right\}\longrightarrow0.
\end{equation*}
Apply this representation along an arbitrary sequence $t\to\infty$.
The compact-uniform expansion
\eqref{eq:finite-scale-uniform-inverse} transfers the convergence from
$F_t^0$ to $X_t^0$, proving the stated uniform convergence and
\textup{(i)}.  The limit has continuous paths and is a Borel random element
of the separable space $C(K)$, so its law is tight.

Conversely, suppose \textup{(ii)} fails for coordinates $r\ne s$.
For $d=2$, set $\tau=(\mu_r+\mu_s)^{-1}$ and
$\lambda^0=\tau\mu$.  For $d>2$, choose
$0<\tau<(\sum_j\mu_j)^{-1}$ and set
\begin{equation*}
 \lambda_r^0=\tau\mu_r,\qquad
 \lambda_s^0=\tau\mu_s,\qquad
 \lambda_j^0=\tau\mu_j+
 \frac{1-\tau\sum_i\mu_i}{d-2}\quad(j\notin\{r,s\}).
\end{equation*}
Then $\lambda^0\in\Delta^\circ$, $I(\lambda^0)=\{r,s\}$ and
$\rho(\lambda^0)=\tau$.  For sufficiently small $u_0>0$, the curve
$\lambda(u)=\lambda^0+u(e_s-e_r)$, $0\leq u\leq u_0$, remains in
$\Delta^\circ$ and satisfies
\begin{equation*}
 I(\lambda(u))=\{r\}\quad(u>0),\qquad
 d_s(\lambda(u))=u(\mu_r^{-1}+\mu_s^{-1}).
\end{equation*}
All remaining gaps stay strictly positive.  Moreover,
\begin{equation*}
 \Var\left(\frac{B_{S,r}(\tau)}{\mu_r}
           -\frac{B_{S,s}(\tau)}{\mu_s}\right)
 =\tau(e_r/\mu_r-e_s/\mu_s)^{\mathsf T}
       \Gamma(e_r/\mu_r-e_s/\mu_s)>0.
\end{equation*}
Proposition~\ref{thm:global-nontight}, after relabeling the coordinates,
therefore gives failure of asymptotic tightness on this curve.  Its image
$K$ is a compact subset of $\Delta^\circ$, and pullback along the curve is
a continuous map from $\ell^\infty(K)$ to
$\ell^\infty([0,u_0])$.  Hence \textup{(i)} cannot hold.
\end{proof}

\section{The terminal edge at a competing boundary}
\label{sec:marked}

The functional theorem describes the macroscopic inverse fluctuation.  To
describe the transition on which the process leaves the lower rectangle,
one must also retain its departure and arrival states, its full vector
increment, and the age and overshoot in the coordinate that caused the exit.
The overshoot is the amount by which that coordinate exceeds its boundary.
This section gives a joint limit for these variables at a critical
interface.

Scalar Gaussian--overshoot separation is established by
Fuh and Lai~\cite{FuhLai2001}.  Multivariate first-excess transforms
\citep{Dshalalow1994,Dshalalow1997,Dshalalow2001,DshalalowLiew2006}
and renewal observations of marked Poisson processes
\citep{BacotDshalalow2001} provide related terminal descriptions.
Here several rectangular horizons share Gaussian face selection and,
when selected coordinate thresholds coincide, the terminal edge itself.
The theorem determines this joint dependence using pointwise scalar
renewal convergence.

\subsection{General state spaces and scalar renewal input}

Let $(\mathsf X,\mathcal X)$ be the standard Borel state space of
Section~\ref{sec:kernel}, and let $P=\sum_vq_v$ have invariant probability
$\pi$.  In this section the stationary drift is finite and strictly positive:
\begin{equation*}
 \mu=\int_{\mathsf X}\pi(\dd x)
          \sum_{v\in\N^d}\int_{\mathsf X}v\,q_v(x,\dd y)
       \in(0,\infty)^d.
\end{equation*}
For every initial state, assume that each coordinate $S_{n,r}$ tends to
infinity almost surely. These are standing hypotheses throughout this
section; in particular, every scalar passage below is finite.
Proposition~\ref{prop:scalar-regeneration} gives sufficient regenerative
conditions for these drift and passage properties.
Fix a critical direction $\lambda^0\in(0,\infty)^d$ and put
\begin{equation*}
 \rho=\min_{r\leq d}\frac{\lambda^0_r}{\mu_r},\qquad
 I=\{r:\lambda^0_r/\mu_r=\rho\},\qquad m=|I|.
\end{equation*}
For $r\in I$ and $k\in\N$, define the scalar passage and its full mark by
\begin{equation*}
 T_r(k)=\inf\{n\geq1:S_{n,r}>k\},\qquad
 V_r(k)=\bigl(J_{T_r(k)-1},J_{T_r(k)},X_{T_r(k)},
                k-S_{T_r(k)-1,r},S_{T_r(k),r}-k\bigr).
\end{equation*}
The mark space is
$\mathcal M=\mathsf X\times\mathsf X\times\N^d\times\N\times\N$.
Define the coordinate equilibrium-edge law on this space by
\begin{equation}
 \nu_r(\dd x,\dd y,\{v\},\{a\},\{o\})
 =\frac{\pi(\dd x)q_v(x,\dd y)}{\mu_r}
      \one_{\{0\leq a<v_r,\ o=v_r-a\}}.
 \label{eq:coordinate-palm-law}
\end{equation}
Equivalently, for bounded measurable $f$,
\begin{equation*}
 \nu_r(f)=\frac1{\mu_r}\int_{\mathsf X}\pi(\dd x)
       \sum_v\int_{\mathsf X}q_v(x,\dd y)
                 \sum_{a=0}^{v_r-1}f(x,y,v,a,v_r-a).
\end{equation*}
Taking $f=1$ gives one.  Thus the stationary edge is size-biased by its
$r$th increment; conditionally on that edge, its integer age is uniform
on $\{0,\ldots,v_r-1\}$.  The law retains the full transition kernel,
including both endpoint states.

Write $\|\xi-\eta\|_{\rm TV}=\sup_A|\xi(A)-\eta(A)|$ for probability
measures on a common measurable space.  The scalar input is the pointwise
convergence of $\mathcal L_x(V_r(k))$ to $\nu_r$.  It is stated explicitly
so that any scalar Markov renewal theorem supplying it can be used.
Proposition~\ref{prop:scalar-regeneration} proves sufficient conditions on
a general state space; Corollary~\ref{cor:finite-marked-input} gives the
finite-state specialization.

\Needspace{6\baselineskip}
\begin{assumption}
\label{ass:marked-renewal}
For the initial law $\alpha$, the following conditions hold.
\begin{enumerate}[label=\textup{(\roman*)}]
\item $\|\alpha P^n-\pi\|_{\rm TV}\to0$ as $n\to\infty$.
\item Under $\Pp_\alpha$, for a centered Brownian motion $B_S$,
\begin{equation*}
 \frac{S_{\lfloor ts\rfloor}-\lfloor ts\rfloor\mu}{\sqrt t}
   \Rightarrow B_S(s)
 \quad\text{in }D_{\rm loc}([0,\infty),\R^d),
 \qquad \Cov(B_S(s))=s\Gamma.
\end{equation*}
For distinct $r,q\in I$, put $a_{rq}=e_r/\mu_r-e_q/\mu_q\in\R^d$ and assume
$a_{rq}^{\mathsf T}\Gamma a_{rq}>0$.
\item For every $x\in\mathsf X$ and $r\in I$,
\begin{equation}
 \delta_r(x,b):=\sup_{k\in\N,\ k\geq b}
       \|\mathcal L_x(V_r(k))-\nu_r\|_{\rm TV}
       \longrightarrow0\qquad(b\to\infty).
 \label{eq:scalar-pointwise-tv}
\end{equation}
\end{enumerate}
\end{assumption}

Condition (i) holds for every initial law when $P$ is positive Harris
recurrent and aperiodic; only the displayed convergence for the chosen
$\alpha$ is used.  The Brownian limit supplies short-window
equicontinuity, and the contrast condition makes the selected coordinate
unique with probability one.  It allows singular $\Gamma_I$ and is empty
when $|I|=1$.  Condition (iii) is scalar renewal equilibration
with the full edge retained as a mark.  It accommodates zero coordinate
increments.  No supremum over $x$, rate of convergence, or separate
moment bound at passage states is imposed.  Lemma~\ref{lem:scalar-edge-equilibrium}
explains why pointwise equilibration suffices at the successive random
states of the joint proof.  A full-dimensional local limit theorem is not needed.

\subsection{The critical chart and its observation space}

For $h\in\R^d$ and $t>0$, let
\begin{equation*}
 k_t(h)=\left\lfloor t\lambda^0+\sqrt t\,h\right\rfloor,
 \qquad T_t(h)=N(k_t(h))+1.
\end{equation*}
All vector floors are coordinatewise, and $t$ is sufficiently large that
$k_t(h)\in\N^d$; this holds uniformly for $h$ in any fixed compact set.
Write
$N_t=N(k_t(h))$ and $T_t=T_t(h)$ when no confusion is possible.  The set of
coordinates crossed by the terminal edge is
\begin{equation*}
 C_t(h)=\{q:S_{T_t,q}>k_{t,q}(h)\}.
\end{equation*}
It is nonempty.  Define $W_t(h)=\min C_t(h)$; the choice on a multiple
crossing will be immaterial because multiple crossings have vanishing
probability.  The macroscopic variables retained in the theorem are
\begin{equation*}
 \Xi_t(h)=\frac{N_t-t\rho}{\sqrt t},\qquad
 L_t(h)=\frac{S_{N_t,I}-N_t\mu_I}{\sqrt t},\qquad
 D_t(h)=\frac{k_{t,I}(h)-S_{N_t,I}}{\sqrt t}.
\end{equation*}
Thus $D_t(h)\geq0$ coordinatewise.  On $\{W_t=r\}$ define the terminal
mark
\begin{equation*}
 M_t(h)=\bigl(J_{N_t},J_{T_t},X_{T_t},A_t,O_t\bigr),
 \quad
 A_t=k_{t,r}(h)-S_{N_t,r},\quad
 O_t=S_{T_t,r}-k_{t,r}(h).
\end{equation*}
Both $A_t$ and $O_t$ are integer-valued, $A_t\geq0$, $O_t\geq1$, and
$A_t+O_t=X_{T_t,r}$.

For $g\in\R^I$ and $r\in I$, define
\begin{equation}
 \zeta_r(h,g)=\frac{h_r-g_r}{\mu_r},\qquad
 \Xi_h(g)=\min_{r\in I}\zeta_r(h,g),\qquad
 W_h(g)=\min\argmin_{r\in I}\zeta_r(h,g),
 \label{eq:gaussian-selector}
\end{equation}
and
\begin{equation*}
 D_h(g)=(h_q-g_q-\mu_q\Xi_h(g))_{q\in I}.
\end{equation*}
Let $G$ have the centered normal law $N_m(0,\rho\Gamma_I)$, and write
$\Xi_h=\Xi_h(G)$, $W_h=W_h(G)$ and $D_h=D_h(G)$.
The definition selects the smallest index on a tie.  The contrast condition in Assumption~\ref{ass:marked-renewal} implies that all pairwise contrasts in
\eqref{eq:gaussian-selector} are nondegenerate, so the minimizing index is
almost surely unique.  Moreover, $D_h\geq0$ and $D_{h,W_h}=0$.

Choose a bounded complete separable metric $d_{\mathsf X}$ generating
$\mathcal X$; such a choice is possible for every standard Borel space.
Equip $\mathcal M$ with the product topology, using the discrete topology
only for its integer coordinates, and a bounded compatible complete metric
$d_{\mathcal M}$.  Put
\begin{equation*}
 \mathcal Z=\R^m\times\R\times\R^m,\qquad
 \mathcal Y=\mathcal Z\times\mathcal M,\qquad
 d_{\mathcal Y}((z,u),(z',u'))=(1\wedge\|z-z'\|)+d_{\mathcal M}(u,u').
\end{equation*}
These spaces are Polish.  The scalar total-variation condition is
independent of the choice of compatible topology.
For finite measures $\xi$ and $\eta$ on
$\mathcal Y$, write
\begin{equation*}
 d_{\mathrm{BL}}(\xi,\eta)
 =\sup_{\substack{\|F\|_\infty\leq1\\
                   \operatorname{Lip}(F)\leq1}}
 \left|\int F\,\dd\xi-\int F\,\dd\eta\right|.
\end{equation*}
For $r\in I$, define the face-indexed subprobability measures
\begin{equation*}
 \mathsf K^r_{t,h}(B)
 =\Pp_\alpha\bigl((L_t(h),\Xi_t(h),D_t(h),M_t(h))\in B,
                    W_t(h)=r\bigr),
\end{equation*}
\begin{equation*}
 \int_{\mathcal Y}F\,\dd\mathsf K^r_h
 =\Ee\!\left[\one_{\{W_h=r\}}
       \int_{\mathcal M}F(G,\Xi_h,D_h,u)\,\nu_r(\dd u)\right].
\end{equation*}
Their masses are the probabilities of selecting face $r$.

\subsection{A joint theorem for the multi-time field}

The following result combines locally uniform one-point laws with the joint
law at any fixed finite family of corners.  The corners may be incomparable
in the coordinatewise order; all observations refer to the same realization
of the Markov-additive process.

For fixed corners \(h^1,\ldots,h^K\in\mathbb R^d\), define the finite set
of coordinate thresholds
\begin{equation*}
 \mathcal B=\{(r,h_r^a):r\in I,\ 1\leq a\leq K\}.
\end{equation*}
An identical pair occurs only once in \(\mathcal B\).  Let
\((U_{r,z})_{(r,z)\in\mathcal B}\) be mutually independent, independent of
\(G\sim N_m(0,\rho\Gamma_I)\), with \(U_{r,z}\) having law \(\nu_r\).
For brevity write
\(\mathcal V_t(h)=(L_t(h),\Xi_t(h),D_t(h),W_t(h))\) and
\(\mathcal V_h=(G,\Xi_h,D_h,W_h)\).

\begin{theorem}
\label{thm:marked-competing-face}
\label{thm:joint-multitime-marks}
Under Assumption~\ref{ass:marked-renewal} for the initial law $\alpha$,
the following conclusions hold on the standard Borel space $\mathsf X$.
For every fixed finite family \(h^1,\ldots,h^K\),
\begin{equation}
 \bigl(\mathcal V_t(h^a),M_t(h^a)\bigr)_{a=1}^K
 \Rightarrow
 \bigl(\mathcal V_{h^a},U_{W_{h^a},h^a_{W_{h^a}}}\bigr)_{a=1}^K.
 \label{eq:joint-multitime-terminal-law}
\end{equation}
For every compact \(H\subset\mathbb R^d\),
\begin{equation*}
 \sup_{h\in H}\Pp_\alpha(C_t(h)\cap I^c\neq\varnothing)
 +\sup_{h\in H}\Pp_\alpha(|C_t(h)|\geq2)\longrightarrow0.
\end{equation*}
Moreover, for every \(r\in I\),
\begin{equation}
 \sup_{h\in H}d_{\mathrm{BL}}(\mathsf K^r_{t,h},\mathsf K^r_h)
 \longrightarrow0.
 \label{eq:uniform-bl-marked-kernels}
\end{equation}
\end{theorem}

In \eqref{eq:joint-multitime-terminal-law}, conditional on \(G\), two
limiting marks are identical when their corners
select the same coordinate and have the same threshold in that coordinate.
Conditional on \(G\), marks attached to distinct selected
coordinate--threshold pairs are independent, even though the selectors
and inverse fluctuations remain dependent.  This specifies the
finite-dimensional terminal law of the multi-time field, including at
incomparable corners.  The theorem does not assert uniformity over families
whose distinct coordinate thresholds coalesce as \(t\) varies.

The product-test formulation below is often convenient in applications and
records explicitly the asymptotic factorization conditional on the selected
face.

\begin{corollary}
\label{cor:marked-product-tests}
Under Assumption~\ref{ass:marked-renewal}, let $r\in I$, let $f$ be a bounded
measurable function on $\mathcal M$, and let $\Phi$ be a bounded Lipschitz function on
$\mathcal Z$.  For the initial law $\alpha$ in that assumption and every
compact $H\subset\R^d$,
\begin{equation}
 \sup_{h\in H}\Big| \Ee_\alpha\!\left[ \Phi(L_t(h),\Xi_t(h),D_t(h))f(M_t(h)) \one_{\{W_t(h)=r\}} \right] -\Ee\!\left[ \Phi(G,\Xi_h,D_h)\one_{\{W_h=r\}} \right]\nu_r(f)\Big|\longrightarrow0.
 \label{eq:selector-conditioned-limit}
\end{equation}
\end{corollary}

\begin{proof}
Fix a sequence $h_t\to h$ and write
$z_u(g)=(g,\Xi_u(g),u_I-g-\mu_I\Xi_u(g))$.
The theorem's proof gives
\begin{equation*}
 (L_t(h_t),\Xi_t(h_t),D_t(h_t))-z_{h_t}(G_t^0)
       \xrightarrow{\Pp}0,\qquad
 \Pp(W_t(h_t)\ne W_{h_t}(G_t^0))\longrightarrow0.
\end{equation*}
Since $\Phi$ is bounded and Lipschitz and $f$ is bounded, these
replacements change the expectation by $o(1)$.  Apply the scalar-array
factorization estimate with
$\psi_t(g)=\Phi(z_{h_t}(g))\one_{\{W_{h_t}(g)=r\}}$, the mark test
$f$ in coordinate $r$, and test one in all other coordinates.
The error bound is uniform over tests with a common supremum bound, so
the expectation becomes $\Ee[\psi_t(G_t^0)]\nu_r(f)+o(1)$.
Joint convergence $(h_t,G_t^0)\Rightarrow(h,G)$ and continuity outside
the Gaussian-null tie hyperplanes give
\begin{equation*}
 \Ee[\psi_t(G_t^0)]\longrightarrow
 \Ee[\Phi(G,\Xi_h,D_h)\one_{\{W_h=r\}}].
\end{equation*}
The same null-boundary argument makes the last expectation continuous
in $h$.  A violating sequence in a compact $H$ would have a convergent
subsequence and contradict this sequential conclusion, proving uniformity.
\end{proof}

The limiting expectation in \eqref{eq:selector-conditioned-limit} is
continuous in $h$: selectors and remaining budgets converge off the
Gaussian-null tie hyperplanes, so bounded convergence applies.  The
uniform statement therefore includes moving sequences $h_t$ in a compact
critical chart.  Conditional on $W_h=r$, the terminal mark has law
$\nu_r$; it depends on the Gaussian fluctuation through the selected
coordinate.  This is joint weak convergence and asymptotic factorization,
not R\'enyi stable convergence relative to the underlying $\sigma$-field.

\subsection{Proof of the joint theorem}

The following elementary transfer principle replaces uniformity over
initial states.  It applies even when the residual level depends on the
state from which the process restarts.

\begin{lemma}
\label{lem:scalar-edge-equilibrium}
Let $Y_t$ be $\mathsf X$-valued and let $b_t$ tend to infinity in
probability.  Suppose that a finite measure $\xi$ and numbers $e_t\to0$
satisfy
\begin{equation*}
 \Pp(Y_t\in A)\leq\xi(A)+e_t\qquad(A\in\mathcal X).
\end{equation*}
Under \eqref{eq:scalar-pointwise-tv}, for any $I$-valued random index $R_t$,
\begin{equation*}
 \Ee\bigl[\one_{\{b_t\geq0\}}\delta_{R_t}(Y_t,b_t)\bigr]
       +\Pp(b_t<0)\longrightarrow0.
\end{equation*}
\end{lemma}

\begin{proof}
Set $\delta_*(x,L)=\max_{r\in I}\delta_r(x,L)$.  Total variation
between kernels on a standard Borel space is measurable, since the
supremum defining it can be taken over a countable generating algebra.
Thus $\delta_*$ is measurable, is bounded by one and decreases pointwise
to zero.  The setwise bound in the statement also holds for integrals
of measurable functions taking values in $[0,1]$, by integrating their level sets.  For any integer $L\geq0$, the expression to be estimated is
at most
\begin{equation*}
 \Pp(b_t<L)+\xi(\delta_*(\cdot,L))+e_t.
\end{equation*}
First let $t\to\infty$ and then $L\to\infty$.  Dominated convergence
under the finite measure $\xi$ proves the assertion.  No independence
between $Y_t$, $b_t$ and $R_t$ is required.
\end{proof}

\begin{proof}[Proof of Theorem~\ref{thm:joint-multitime-marks}]
We first treat a fixed family of corners. Fix $\beta\in(1/2,1)$ and set
\begin{equation*}
 n_t^-=\lfloor t\rho-t^\beta\rfloor,\qquad
 n_t^+=\lfloor t\rho+t^\beta\rfloor,\qquad
 G_t^0=\frac{S_{n_t^-,I}-n_t^-\mu_I}{\sqrt t}.
\end{equation*}
We take $t$ sufficiently large that $n_t^-\geq0$ and use the natural
filtration $(\mathcal F_n)$. The functional limit and continuity of its
Brownian limit imply
\begin{equation*}
 G_t^0\Rightarrow G,\qquad
 R_t:=\max_{n_t^-\leq n\leq n_t^+}
 \frac{|S_{n,I}-S_{n_t^-,I}-(n-n_t^-)\mu_I|_\infty}{\sqrt t}
 \xrightarrow{\Pp}0.
\end{equation*}
Indeed, both endpoints divided by $t$ tend to $\rho$, and the modulus of
continuity of the centered partial-sum process on an interval of length
$2t^{\beta-1}+2/t$ tends to zero in probability.

For each distinct label $b=(r,z)\in\mathcal B$, put
\begin{equation*}
 k_b(t)=\lfloor t\rho\mu_r+\sqrt t\,z\rfloor,\quad
 T_b(t)=T_r(k_b(t)),\quad
 \eta_b(g)=\frac{z-g_r}{\mu_r},\quad B=|\mathcal B|.
\end{equation*}
Write $\mu_* =\min_{r\in I}\mu_r$ and $\mu^*=\max_{r\in I}\mu_r$.
For $n_t^-\leq n\leq n_t^+$, the definition of $R_t$ gives
$S_{n,r}=n\mu_r+\sqrt t\,G_{t,r}^0+\sqrt t\,e_{n,r}$ with
$|e_{n,r}|\leq R_t$. Set $c_t=(R_t+2t^{-1/2})/\mu_*$.
The integers immediately below
$t\rho+\sqrt t(\eta_b(G_t^0)-c_t)$ and immediately above
$t\rho+\sqrt t(\eta_b(G_t^0)+c_t)$ lie in the common window with
probability tending to one, simultaneously for all labels. This follows
from tightness of $G_t^0$, $c_t\to0$ in probability and
$t^\beta/\sqrt t\to\infty$.
At the first integer the coordinate is at most $k_b(t)$; at the second
it is strictly greater than $k_b(t)$. Monotonicity therefore proves
\begin{equation*}
 \Delta_t:=\max_{b\in\mathcal B}
 \left|\frac{T_b(t)-t\rho}{\sqrt t}-\eta_b(G_t^0)\right|
 \xrightarrow{\Pp}0,
 \qquad
 \Pp(n_t^-<T_b(t)\leq n_t^+\text{ for all }b)\longrightarrow1.
\end{equation*}
In fact, on the event where these comparison integers belong to the
window, $\Delta_t\leq c_t+t^{-1/2}$.

Define
\begin{equation*}
 d_t=\min_{b\ne c}|\eta_b(G_t^0)-\eta_c(G_t^0)|,
 \qquad d(g)=\min_{b\ne c}|\eta_b(g)-\eta_c(g)|,
\end{equation*}
with both minima equal to $+\infty$ when $B=1$. Two distinct labels in
the same coordinate have a nonzero deterministic difference. For
different coordinates their difference at $G$ is a Gaussian variable
with positive variance by Assumption~\ref{ass:marked-renewal}(ii).
Consequently $d(G)>0$ almost surely and
\begin{equation*}
 \lim_{\varepsilon\downarrow0}\limsup_{t\to\infty}
 \Pp(d_t\leq\varepsilon)=0.
\end{equation*}
Order the labels as $b_1,\ldots,b_B$, where $b_p=(r_p,z_p)$, by
increasing $\eta_b(G_t^0)$, with deterministic tie-breaking. The
permutation is $\mathcal F_{n_t^-}$-measurable. Let $\mathcal E_t$ be
the event that every passage lies in the common window and that
$2\Delta_t<d_t$. On this event the actual passages are strictly ordered
as $T_{b_1}(t)<\cdots<T_{b_B}(t)$. For each $\varepsilon>0$,
\begin{equation*}
 \Pp(\mathcal E_t^c)
 \leq\Pp(\text{some passage is outside the window})
       +\Pp(d_t\leq\varepsilon)+\Pp(2\Delta_t\geq\varepsilon).
\end{equation*}
Taking first $t\to\infty$ and then $\varepsilon\downarrow0$ shows that
$\Pp(\mathcal E_t^c)\to0$.

We next carry out the conditioning without restricting the probability
space to $\mathcal E_t$. Define $\sigma_0=n_t^-$ and, recursively,
\begin{equation*}
 \sigma_p=\inf\{n\geq\sigma_{p-1}:S_{n,r_p}>k_{b_p}(t)\},
 \qquad \ell_p=k_{b_p}(t)-S_{\sigma_{p-1},r_p}.
\end{equation*}
All $\sigma_p$ are finite stopping times. To verify the stopping-time
property, for $n\geq n_t^-$ expand $\{\sigma_p\leq n\}$ as the finite
union, over possible labels $b=(r,z)$, of
\begin{equation*}
 \{b_p=b\}\cap\{\sigma_{p-1}\leq n\}
                  \cap\{S_{n,r}>k_b(t)\}.
\end{equation*}
Each event belongs to $\mathcal F_n$; finiteness follows from the
standing coordinate-divergence assumption. On $\mathcal E_t$,
$\sigma_p=T_{b_p}(t)$ for every $p$.

The first residual satisfies
$\ell_1=\mu_{r_1}t^\beta+O_{\Pp}(\sqrt t)$. For $p\geq2$, on
$\mathcal E_t$ the preceding path estimates give the explicit bound
\begin{equation*}
 \left|\frac{\ell_p}{\sqrt t}
   -\mu_{r_p}\{\eta_{b_p}(G_t^0)-\eta_{b_{p-1}}(G_t^0)\}\right|
 \leq R_t+\mu^*\Delta_t+t^{-1/2}.
\end{equation*}
Since the difference in braces is at least $d_t$, these inequalities
imply $\ell_p\to\infty$ in probability for every $p$. More explicitly,
for fixed $L$ and $\varepsilon>0$, and all sufficiently large $t$,
\begin{equation*}
 \Pp(\ell_p<L)\leq\Pp(\mathcal E_t^c)+\Pp(d_t\leq\varepsilon)
 +\Pp(R_t+\mu^*\Delta_t+t^{-1/2}\geq\mu_*\varepsilon/2)
 \quad(p\geq2).
\end{equation*}
The claimed divergence follows by the same successive limits.

On $\{\ell_p\geq0\}$, let $\widetilde V_p$ be the scalar terminal
mark generated after $\sigma_{p-1}$ at residual level $\ell_p$; it is
the mark at $\sigma_p$. On $\{\ell_p<0\}$ assign a fixed dummy mark.
These variables are $\mathcal F_{\sigma_p}$-measurable. Define
\begin{equation*}
 \varepsilon_{p,t}
 =\one_{\{\ell_p\geq0\}}
       \delta_{r_p}(J_{\sigma_{p-1}},\ell_p)
       +\one_{\{\ell_p<0\}},\qquad
 q_{p,t}=\Ee[\varepsilon_{p,t}].
\end{equation*}
The strong Markov property, applied separately to the finitely many
possible labels, yields
\begin{equation*}
 \left\|\mathcal L(\widetilde V_p\mid\mathcal F_{\sigma_{p-1}})
                   -\nu_{r_p}\right\|_{\rm TV}
 \leq\varepsilon_{p,t}\quad\text{almost surely}.
\end{equation*}
This conditional bound applies also to bounded test functions selected
measurably using $\mathcal F_{\sigma_{p-1}}$.

Here is the needed induction controlling the passage states. Let
$\nu_r^+$ be the arrival-state marginal of $\nu_r$, and put
$\xi=\pi+\sum_{r\in I}\nu_r^+$. Set
$e_{0,t}=\|\alpha P^{n_t^-}-\pi\|_{\rm TV}$ and
$e_{p,t}=2q_{p,t}$ for $p\geq1$.
If
$\Pp(J_{\sigma_{p-1}}\in A)\leq\xi(A)+e_{p-1,t}$ for all $A$, then,
with $\delta_*(x,L)=\max_{r\in I}\delta_r(x,L)$, every integer $L\geq0$
satisfies
\begin{equation*}
 q_{p,t}\leq\Pp(\ell_p<L)+\xi(\delta_*(\cdot,L))+e_{p-1,t}.
\end{equation*}
Indeed, on $\{\ell_p\geq L\}$ monotonicity in the residual level
bounds the error by $\delta_*(J_{\sigma_{p-1}},L)$; the domination
inequality extends to $[0,1]$-valued functions by integration of level
sets. The initial domination follows from the definition of $e_{0,t}$.
The conditional bound and the arrival-state coordinate of the mark give
\begin{equation*}
 \Pp(J_{\sigma_p}\in A)
 \leq\Ee[\nu_{r_p}^+(A)]+q_{p,t}+\Pp(\ell_p<0)
 \leq\xi(A)+2q_{p,t}.
\end{equation*}
At each stage, take $t\to\infty$ in the preceding error bound and then
$L\to\infty$. Pointwise scalar equilibration and dominated convergence
under the finite measure $\xi$ show inductively that
$q_{p,t}\to0$ for every $p$. Thus no uniform renewal bound over starting
states has been used.

We now establish factorization while accounting for the random order.
First take $\|\psi\|_\infty\leq1$ and
$\|f_b\|_\infty\leq1$, with all these tests bounded measurable.
For each deterministic permutation $\tau$ of $\mathcal B$, let
$A_{\tau,t}\in\mathcal F_{n_t^-}$ denote the event that the selected
order equals $\tau$. In
\begin{equation*}
 \Ee\left[\one_{A_{\tau,t}}\psi(G_t^0)
                      \prod_{p=1}^B f_{\tau(p)}(\widetilde V_p)\right]
\end{equation*}
remove factors successively for $p=B,B-1,\ldots,1$, using conditional
expectation given $\mathcal F_{\sigma_{p-1}}$. The earlier factors and
$\one_{A_{\tau,t}}\psi(G_t^0)$ are measurable for that sigma-field.
The error at stage $p$ is at most
$2\Ee[\one_{A_{\tau,t}}\varepsilon_{p,t}]$; the replacement constant
is $\nu_{r_{\tau(p)}}(f_{\tau(p)})$. Summing over the disjoint
permutation events, the product of these constants is independent of
their order. Replacing the sequential marks by the original scalar
marks costs at most $2\Pp(\mathcal E_t^c)$. Consequently
\begin{equation*}
 \left|\Ee\left[\psi(G_t^0)
        \prod_{b\in\mathcal B}f_b(V_{r_b}(k_b(t)))\right]
       -\Ee[\psi(G_t^0)]\prod_{b\in\mathcal B}\nu_{r_b}(f_b)\right|
 \leq2\sum_{p=1}^Bq_{p,t}+2\Pp(\mathcal E_t^c)\longrightarrow0.
\end{equation*}
General bounded tests follow by division by their suprema, with zero
tests treated separately.  All conditional expectations above are taken
with respect to stopping-time sigma-fields before restricting to
$\mathcal E_t$.

Each scalar mark marginal converges in total variation: its distance
from $\nu_{r_b}$ is at most
$\int\delta_{r_b}(x,k_b(t))\,\alpha(\dd x)\to0$.
It is therefore tight on the chosen Polish mark space. Joint tightness
of the finite scalar array and $G_t^0$, followed by the preceding
factorization for bounded continuous tests, proves
\begin{equation*}
 \bigl(G_t^0,(V_{r_b}(k_b(t)))_{b\in\mathcal B}\bigr)
 \Rightarrow\bigl(G,(U_b)_{b\in\mathcal B}\bigr).
\end{equation*}
Products of bounded continuous functions determine the probability law
on this finite product of Polish spaces.

For each corner the active exit is the minimum of its scalar passages.
Outside the Gaussian-null tie hyperplanes its limiting minimum is
unique. For $q\notin I$, the functional law of large numbers gives
$S_{n_t^+,q}/t\to\rho\mu_q<\lambda_q^0$ in probability. Monotonicity
then excludes every inactive coordinate from all the terminal crossings
under consideration with probability tending to one. The full exit is
therefore the selected scalar passage and has exactly its scalar mark.
The path bound, applied also at $N_t=T_t-1$, gives jointly for all $a$
\begin{equation*}
 L_t(h^a)-G_t^0\xrightarrow{\Pp}0,\qquad
 \Xi_t(h^a)-\Xi_{h^a}(G_t^0)\xrightarrow{\Pp}0,
\end{equation*}
and the exact definition gives
$D_t(h^a)=h_I^a-L_t(h^a)-\mu_I\Xi_t(h^a)+O(t^{-1/2})$.
The selection map from the Gaussian vector and scalar array is continuous
away from the tie hyperplanes. The continuous mapping theorem proves
\eqref{eq:joint-multitime-terminal-law}; repeated selected labels are
mapped to the same array entry. The factorization estimate also holds
for bounded measurable mark tests. Splitting according to the finitely
many Gaussian selector patterns, whose boundaries have Gaussian measure
zero, gives the corresponding product-test limits for bounded measurable mark tests
and bounded continuous macroscopic tests.

Finally let $h_t\to h$. Repeat the argument with the single corner
$h_t$: there is one label per active coordinate, its threshold parameter
converges, and the limiting Gaussian passage gaps remain nonzero almost
surely. Every estimate above is unchanged. It follows that
$\mathsf K^r_{t,h_t}\Rightarrow\mathsf K^r_h$, while inactive and
multiple crossing probabilities tend to zero. These are finite measures
on a Polish space, so weak convergence is equivalent to convergence in
$d_{\rm BL}$. Also $\mathsf K^r_{h_t}\Rightarrow\mathsf K^r_h$:
for almost every $G$ the selector is locally constant at $h$, and the
remaining macroscopic coordinates vary continuously, so bounded
convergence applies to every bounded continuous test. If any asserted
uniform limit failed on a compact $H$, select points with a fixed positive
error and extract a convergent subsequence in $H$. The just established
sequential conclusions, and the triangle inequality for $d_{\rm BL}$,
give a contradiction. This proves all the uniform assertions.
\end{proof}
\subsection{Local dependence of terminal observations}

The joint theorem gives a direct decomposition of terminal covariance.
For bounded measurable real functions $f_1,f_2$ on $\mathcal M$, put
$m_a(G)=\nu_{W_{h^a}}(f_a)$ and
$c_r=\nu_r(f_1f_2)-\nu_r(f_1)\nu_r(f_2)$.  Let
\begin{equation*}
 \mathcal C=\{W_{h^1}=W_{h^2},\quad
                  h^1_{W_{h^1}}=h^2_{W_{h^2}}\}.
\end{equation*}
Then
\begin{equation}
 \lim_{t\to\infty}\Cov\bigl(f_1(M_t(h^1)),f_2(M_t(h^2))\bigr)
 =\Cov\bigl(m_1(G),m_2(G)\bigr)
    +\Ee[\one_{\mathcal C}c_{W_{h^1}}].
 \label{eq:terminal-covariance}
\end{equation}
Indeed, the bounded-measurable product-test conclusion established in
the proof of Theorem~\ref{thm:marked-competing-face} gives convergence of
both moments and products.  Conditional on $G$, distinct
selected pairs have independent marks, and a shared pair has covariance
$c_{W_{h^1}}$.  The law of total covariance gives the formula.
Its first term measures dependence through the common Gaussian fluctuation;
the second measures the additional dependence from observing one terminal
edge twice.

Fixed integer offsets retain the dependence between nearby levels.  For
each \(r\in I\), write
\(\mathcal M_r=\{(x,y,v,a,o)\in\mathcal M:0\leq a<v_r,\ o=v_r-a\}\).
On this space define the level-transition kernel
\begin{equation*}
 K_r(u,\cdot)=
 \begin{cases}
  \delta_{(x,y,v,a+1,o-1)},&o>1,\\
  \mathcal L_y(V_r(0)),&o=1,
 \end{cases}
 \qquad u=(x,y,v,a,o).
\end{equation*}
The first case retains the same edge; the second starts the next passage
from its arrival state.  The following result uses a stationary version
of this scalar level process for every distinct coordinate threshold.

\begin{corollary}
\label{cor:marked-fixed-offsets}
Under Assumption~\ref{ass:marked-renewal}, the probability \(\nu_r\) is
invariant for \(K_r\).  For fixed \(h^1,\ldots,h^K\in\R^d\) and
\(\ell^1,\ldots,\ell^K\in\Z^d\), let
\(\mathcal V_t^{\ell^a}(h^a),M_t^{\ell^a}(h^a)\) denote the variables
of Theorem~\ref{thm:joint-multitime-marks} evaluated at
\(k_t(h^a)+\ell^a\).
For each \(b=(r,z)\in\mathcal B\), let
\((\mathcal U_b(m))_{m\in\Z}\) be a stationary Markov chain with
kernel \(K_r\) and marginal \(\nu_r\).  Take these chains mutually
independent and independent of \(G\).  Then, for the initial law $\alpha$
in Assumption~\ref{ass:marked-renewal},
\begin{equation*}
 \bigl(\mathcal V_t^{\ell^a}(h^a),M_t^{\ell^a}(h^a)\bigr)_{a=1}^K
 \Rightarrow
 \bigl(\mathcal V_{h^a},
   \mathcal U_{(W_{h^a},h^a_{W_{h^a}})}
                    (\ell^a_{W_{h^a}})\bigr)_{a=1}^K.
\end{equation*}
\end{corollary}

\begin{proof}
The strong Markov property at \(T_r(k)\) shows that
\((V_r(k))_{k\geq0}\) is a Markov chain with kernel \(K_r\).
Indeed, an overshoot exceeding one retains the current edge at the next
level; an overshoot equal to one places that level at the arrival
position.  The pointwise convergence in
\eqref{eq:scalar-pointwise-tv} and contraction of total variation under a
Markov kernel therefore give
\(\nu_rK_r=\nu_r\).  This invariant probability defines a two-sided
stationary chain, denoted here by \(\mathcal U_r\).
For fixed integers \(m_-\leq m_+\) and integer \(k\) with
\(k+m_-\geq0\), contraction under the finite path kernel gives
\begin{equation*}
 \left\|\mathcal L_x\bigl((V_r(k+m))_{m=m_-}^{m_+}\bigr)
       -\mathcal L\bigl((\mathcal U_r(m))_{m=m_-}^{m_+}\bigr)
 \right\|_{\rm TV}
 \leq\delta_r(x,k+m_-),\qquad x\in\mathsf X.
\end{equation*}

For each \(b=(r,z)\), group the offsets
\(\{\ell_r^a:h_r^a=z\}\) and denote their smallest and largest values
by \(m_b^-\) and \(m_b^+\).  Use the early time \(n_t^-\), Gaussian
approximation \(G_t^0\), levels \(k_b(t)\), and
\(\mathcal F_{n_t^-}\)-measurable ordering \(b_1,\ldots,b_B\) from
the proof of Theorem~\ref{thm:joint-multitime-marks}, writing
\(b_p=(r_p,z_p)\).
Its passage approximation applies to all finitely many offset levels.
Consequently, with probability tending to one, every group's passages
follow \(n_t^-\) and the entire groups occur in this order.

Starting from \(\sigma_0=n_t^-\), define \(\sigma_p\) as the first
time at or after \(\sigma_{p-1}\) that coordinate \(r_p\) exceeds
\(k_{b_p}(t)+m_{b_p}^+\).  These are stopping times on every path
and are finite almost surely by the standing assumption that every
coordinate tends to infinity.
The residual distance to the group's smallest level is
\(d_p=k_{b_p}(t)+m_{b_p}^- -S_{\sigma_{p-1},r_p}\).
The first residual is \(\mu_{r_1}t^\beta+O_{\Pp}(\sqrt t)\);
subsequent residuals divided by \(\sqrt t\) have the same positive
limiting gaps as in the theorem's proof.  Thus every \(d_p\) tends to
infinity in probability.  When \(d_p\geq0\), collect all scalar marks
at levels \(k_{b_p}(t)+m\), \(m_{b_p}^-\leq m\leq m_{b_p}^+\), from
the path after \(\sigma_{p-1}\); otherwise assign a fixed dummy array.

Conditionally on \(\mathcal F_{\sigma_{p-1}}\), the collected array
differs from its stationary window law in total variation by at most
$\delta_{r_p}(J_{\sigma_{p-1}},d_p)$ on $\{d_p\geq0\}$.
The last mark of the stationary window has arrival law $\nu_{r_p}^+$,
by invariance of $\nu_{r_p}$.  Therefore the same induction with
$\xi=\pi+\sum_{r\in I}\nu_r^+$ and
Lemma~\ref{lem:scalar-edge-equilibrium} shows that all expected errors
vanish.  Removing bounded measurable window tests backwards, from group
$B$ to group 1, gives independent stationary windows jointly with
$G_t^0\Rightarrow G$.  Agreement with the original arrays on the
event of correctly ordered groups is used only after this conditioning
argument.  Finally, fixed integer offsets leave the macroscopic limits
and Gaussian selectors unchanged.  Selecting the appropriate coordinate
and offset in each window proves the assertion.
\end{proof}

\subsection{Regenerative sufficient conditions}

The scalar equilibrium law is part of classical Markov renewal theory
\citep{Nummelin1978,Alsmeyer1994}.  The next proposition gives a sufficient
construction for the total-variation input used here.  Its proof keeps
both endpoint states and handles the dependence of a departure state on
the preceding regeneration block.

\begin{proposition}
\label{prop:scalar-regeneration}
Let $P$ be positive Harris recurrent with invariant probability $\pi$.
Suppose the joint state--increment kernel has a split atom as in
Section~\ref{subsec:regenerative-fclt}, reached almost surely from every
initial state.  Write $H_0<H_1<\cdots$ for its entrance and return times,
$\mathcal L_j=H_j-H_{j-1}$ and
$\mathcal V_j=S_{H_j}-S_{H_{j-1}}$ for $j\geq1$.
Assume $\Ee[\mathcal L_1+|\mathcal V_1|_1]<\infty$ and
$\mu=\Ee[\mathcal V_1]/\Ee[\mathcal L_1]\in(0,\infty)^d$.
For each $r\in I$, the positive support of $Y_j=\mathcal V_{j,r}$
has greatest common divisor one.  Then the drift agrees with the
stationary drift above, $S_n/n\to\mu$ almost surely from every initial
state, and \eqref{eq:scalar-pointwise-tv} holds.
\end{proposition}

\begin{proof}
At the split atom the outgoing joint law is the same for all its
representatives.  Consequently the successive block innovations $E_j$
are independent and identically distributed, and independent of the
history through $H_0$.  Here $E_j$ records the length, increments, and
states strictly after $H_{j-1}$ through $H_j$; it excludes the departure
representative at $H_{j-1}$.  Thus $(\mathcal L_j,\mathcal V_j)$ are
functions of $E_j$, whereas the full edge marks in block $j\geq2$ are
functions of $(E_{j-1},E_j)$.  This is the boundary dependence described
in Proposition~\ref{prop:regen-fclt}.
Write $\ell=\Ee[\mathcal L_1]$.  The ordinary strong laws for the
i.i.d. pairs $(\mathcal L_j,\mathcal V_j)$ give
\begin{equation*}
 \frac{H_m}{m}\longrightarrow\ell,\qquad
 \frac{S_{H_m}}m\longrightarrow\Ee[\mathcal V_1]
 \quad\hbox{almost surely}.
\end{equation*}
The finite initial segment contributes zero to these limits.  Finite mean gives
$\sum_m\Pp(\mathcal L_m>\epsilon m)<\infty$ for each $\epsilon>0$;
hence $\mathcal L_m/m\to0$ almost surely and $H_{m+1}/H_m\to1$.
For $H_m\leq n<H_{m+1}$, monotonicity sandwiches $S_n$ between
$S_{H_m}$ and $S_{H_{m+1}}$.  It follows that $S_n/n\to\mu$
almost surely from every initial state.

We identify the stationary edge occupation measure explicitly.  For a
bounded measurable transition function $F$, put
\begin{equation*}
 QF(x)=\sum_v\int F(x,y,v)q_v(x,\dd y),\qquad
 R_j(F)=\sum_{n=H_{j-1}+1}^{H_j}F(J_{n-1},J_n,X_n).
\end{equation*}
The variables $R_j(F)$, $j\geq2$, form a stationary integrable
$1$-dependent sequence, because they are functions of
$(E_{j-1},E_j)$ and $|R_j(F)|\leq\|F\|_\infty\mathcal L_j$.
Their even and odd subsequences are each i.i.d.; hence
\begin{equation*}
 \frac{\sum_{j=2}^m R_j(F)}{H_m-H_1}
       \longrightarrow\frac{\Ee[R_2(F)]}{\ell}
       \quad\hbox{almost surely}.
\end{equation*}
On the original path, the variables
$F(J_{n-1},J_n,X_n)-QF(J_{n-1})$ are bounded martingale
differences.  Their averages tend to zero almost surely, since the sum
of their second moments divided by $n^2$ is finite.  The positive
Harris ergodic theorem applied to the bounded function $QF$ therefore
gives the same limit as
\begin{equation*}
 \int\pi(\dd x)\sum_v\int F(x,y,v)q_v(x,\dd y).
\end{equation*}
The pathwise assertion remains valid along the split return times.
Thus
\begin{equation*}
 \Ee[R_2(F)]
 =\ell\int\pi(\dd x)\sum_v\int F(x,y,v)q_v(x,\dd y).
\end{equation*}
Applying this identity to nonnegative truncations and using monotone
convergence extends it to all nonnegative transition functions.  In
particular, taking $F(x,y,v)=v_r$ gives
\begin{equation*}
 \int\pi(\dd x)\sum_v\int v_rq_v(x,\dd y)
      =\frac{\Ee[\mathcal V_{1,r}]}{\ell}=\mu_r<\infty.
\end{equation*}
Finally, if $f$ is a bounded measurable full-mark test, then
$F(x,y,v)=\sum_{a=0}^{v_r-1}f(x,y,v,a,v_r-a)$ satisfies
$|F(x,y,v)|\leq\|f\|_\infty v_r$.  Applying the identity to its
positive and negative parts proves the full-edge occupation formula
used below.

Fix $r\in I$ and put $\gamma_r=\Ee[Y_1]=\ell\mu_r$.  The ordinary renewal mass
\begin{equation*}
 u(k)=\sum_{j\geq0}\Pp(Y_1+\cdots+Y_j=k),\qquad u(k)=0\quad(k<0),
\end{equation*}
is bounded and tends to $\gamma_r^{-1}$ as $k\to\infty$.
Indeed, if $p=\Pp(Y_1=0)$, then $p<1$; deleting zeros multiplies
the renewal mass for the positive conditional law by $(1-p)^{-1}$.
That positive-increment renewal mass is bounded by one, and the span-one
arithmetic renewal theorem of
Erd{\H o}s, Feller and Pollard~\cite{ErdosFellerPollard1949} gives the
asserted limit with the original mean $\gamma_r$.

Define finite measures $B_l$, $l\in\N$, from two successive innovations
$E_1,E_2$.  For an edge $e$ in the second block, let $s(e)$ be the scalar
displacement in that block before $e$, let $v(e)$ be its vector increment,
and let $m(e,a)$ be its full mark with age $a$ and overshoot $v_r(e)-a$.
The first departure state in this block is the terminal state of $E_1$.
Set
\begin{equation*}
 B_l(A)=\Ee\left[\sum_{e\text{ in block }2}
        \sum_{a=0}^{v_r(e)-1}
        \one_{\{Y_1+s(e)+a=l\}}\one_{\{m(e,a)\in A\}}\right].
\end{equation*}
Then $\sum_lB_l(\mathcal M)=\Ee[Y_2]=\gamma_r<\infty$.
For bounded measurable $f$, the transition reward
$\sum_{a=0}^{v_r-1}f(x,y,v,a,v_r-a)$ is integrable over a cycle.
The cycle occupation identity, retaining its departure representative,
gives
\begin{equation*}
 \sum_l B_l(f)
 =\ell\int\pi(\dd x)\sum_v\int q_v(x,\dd y)
                  \sum_{a=0}^{v_r-1}f(x,y,v,a,v_r-a)
 =\gamma_r\nu_r(f).
\end{equation*}
This is the occupation formula proved above for integrable full-edge
rewards; it does not require independence of adjacent full-mark blocks.

For initial state $x$, put $D=S_{H_0,r}$ and
$u_x(k)=\Ee_x[u(k-D)]$.  The delay $D$ is finite almost surely, so
$u_x(k)\to\gamma_r^{-1}$ and $\sup_k u_x(k)\leq(1-p)^{-1}$.
Let $R_x(k,\cdot)$ be the contribution to the terminal law from an edge
at or before $H_1$.  Its mass is at most
$\Pp_x(S_{H_1,r}>k)$, which tends to zero.
For passages in block $j+2$, the displacement before block $j+1$ is
$D+Y_1+\cdots+Y_j$, independently of $(E_{j+1},E_{j+2})$.
Summing over $j\geq0$ consequently gives the measure identity
\begin{equation*}
 \mathcal L_x(V_r(k))=R_x(k,\cdot)+\sum_{l\geq0}u_x(k-l)B_l.
\end{equation*}
Its distance in total variation from $\nu_r$ is bounded by
\begin{equation*}
 R_x(k,\mathcal M)+\sum_{l\geq0}
       |u_x(k-l)-\gamma_r^{-1}|B_l(\mathcal M),
\end{equation*}
which tends to zero by dominated convergence.  Taking the supremum over
all integer $k\geq b$ gives \eqref{eq:scalar-pointwise-tv}.
\end{proof}

The joint-atom hypothesis concerns the law of both the next state and
increment.  Splitting the state kernel alone need not give independent
increment blocks; the correction and 1-dependence in
Alsmeyer~\cite{Alsmeyer1994} explain this distinction.  Proposition
\ref{prop:scalar-regeneration} states its joint regeneration requirement
explicitly.  More general scalar renewal constructions can also be used
through \eqref{eq:scalar-pointwise-tv}.
If $P$ is aperiodic, its positive Harris recurrence supplies
Assumption~\ref{ass:marked-renewal}(i).  The cycle second moment of
Proposition~\ref{prop:regen-fclt}, with zero reward, supplies part (ii),
subject to its separate contrast condition.  The scalar
part needs only the first moments in the proposition above.

\subsection{Denumerable and finite state spaces}

For a denumerable irreducible positive recurrent state chain, each state
is an atom for the joint kernel.  Choose one state $i$ and its return time
$\tau_i$.  The cycle means are
$\Ee_i[\tau_i]=1/\pi_i$ and
$\Ee_i[S_{\tau_i,r}]=\mu_r/\pi_i$.
Thus finite stationary drift and
$\gcd\{k\geq1:\Pp_i(S_{\tau_i,r}=k)>0\}=1$ for every $r\in I$
imply scalar total-variation convergence by
Proposition~\ref{prop:scalar-regeneration}.
Aperiodicity of the state chain, the regenerative cycle second moment,
and the Gaussian contrast requirement then give the marked theorem for
every initial law.  No finiteness of the
state space is involved.  The zero-cluster bound of
Assumption~\ref{ass:zero-cluster} remains a separate sufficient condition
for the exact kernel identities.

For a finite state space $E$, write $q_{ij}(v)=q_v(i,\{j\})$.
The following matrix conditions make the general hypotheses explicit.

\begin{corollary}
\label{cor:finite-marked-input}
Let $P$ be irreducible and aperiodic on $E$, let $\mu\in(0,\infty)^d$,
and assume
\begin{equation}
 \max_{i\in E}\sum_{j\in E}\sum_{v\in\N^d}q_{ij}(v)|v|_1^2<\infty.
 \label{eq:marked-second-moment}
\end{equation}
For $r\in I$, put
$Q_r(\theta)=[\sum_vq_{ij}(v)e^{\mathrm i\theta v_r}]_{i,j\in E}$
and assume
\begin{equation}
 \det(I-Q_r(\theta))\neq0
       \qquad(\theta\not\equiv0\pmod{2\pi}).
 \label{eq:scalar-renewal-arithmetic}
\end{equation}
Let $\Gamma$ be the displacement covariance in
Proposition~\ref{prop:finite-state-fclt}.  If the pairwise contrasts in
Assumption~\ref{ass:marked-renewal}(ii) have positive variance, all parts
of that assumption hold for every initial law on $E$.
Moreover, $\sup_{i\in E}\delta_r(i,b)\to0$.
\end{corollary}

\begin{proof}
Finite-state ergodicity gives part (i), and
Proposition~\ref{prop:finite-state-fclt}, using
\eqref{eq:marked-second-moment}, gives the Brownian limit in part (ii).  To establish part (iii), fix $i\in E$, let $\tau_i$ be its
first return time, and put $Y=S_{\tau_i,r}$.  If the positive support
of $Y$ had greatest common divisor greater than one, some
$\theta\not\equiv0\pmod{2\pi}$ would satisfy
$e^{\mathrm i\theta Y}=1$ almost surely.  Set $a_i=1$ and, for $j\ne i$,
$a_j=\Ee_j[e^{\mathrm i\theta S_{\eta_i,r}}]$, where $\eta_i$ is the
first hitting time of $i$.  First-step conditioning, using the first
return at $i$, gives $Q_r(\theta)a=a$.  This contradicts
\eqref{eq:scalar-renewal-arithmetic}.  The return displacement therefore
has span one, and Proposition~\ref{prop:scalar-regeneration} applies.
The finite number of starting states gives the final uniform statement.
\end{proof}

The scalar condition \eqref{eq:scalar-renewal-arithmetic} excludes
nontrivial periods, including state-dependent lattice offsets; a
deterministic unit increment satisfies it.  It places no full space--time
lattice restriction and requires no exponential moments.
These finite-state hypotheses also imply $\operatorname{spr}(Q_0)<1$:
a closed zero-displacement class would be a closed class of $P$, contrary
to irreducibility and positive drift.  The same argument applies to the
matrix of transitions with zero increment in any fixed coordinate.

\subsection{A continuous modulating environment}

The general-state hypotheses apply to a persistent continuous environment.
Let $\mathsf X=[0,1]$, let $\kappa$ be Lebesgue probability, and set
\begin{equation*}
 P(x,\dd y)=(1-\varepsilon)\delta_x(\dd y)+\varepsilon\kappa(\dd y),
 \qquad 0<\varepsilon<1.
\end{equation*}
For measurable $p_r:[0,1]\to[\eta,1-\eta]$, $0<\eta<1/2$, define
\begin{equation*}
 q_v(x,\dd y)=P(x,\dd y)
       \prod_{r=1}^d p_r(x)^{v_r-1}(1-p_r(x))^{2-v_r},
       \qquad v\in\{1,2\}^d,
\end{equation*}
and set $q_v=0$ otherwise.  Nonconstant $p_r$ make the cumulative
increments depend on the continuous environment.  Here $\pi=\kappa$,
$\mu_r=1+\int p_r\,\dd\kappa$, and
$\|\alpha P^n-\kappa\|_{\rm TV}\leq(1-\varepsilon)^n$ for every $\alpha$.

The joint kernel minorizes
$\varepsilon\eta^d\kappa(\dd y)\delta_{\mathbf1}(\dd v)$ uniformly in $x$.
The associated split construction has geometric regeneration lengths.
An immediate return to its atom has positive probability and displacement
$\mathbf1$, so every coordinate return displacement has span one.
Bounded increments give all cycle moments, and the initial delayed
segment is almost surely finite.  The regenerative proposition and
functional limit therefore apply for every initial law.
Conditional on the full state path, the increment noises
$\epsilon_n=X_n-\Ee[X_n\mid(J_j)_{j\geq0}]$ are independent,
have conditional mean zero and conditional covariance at least
$\eta(1-\eta)I_d$.  They are orthogonal to all square-integrable
functions of the state path.  Hence, in stationarity,
\begin{equation*}
 \Cov_\kappa(S_n)\succeq n\eta(1-\eta)I_d.
\end{equation*}
For the bounded drift function $b(x)=\Ee_x[X_1]-\mu$,
$u=b/\varepsilon$ solves $u-Pu=b$.  The corresponding martingale
differences $X_n-\mu+u(J_n)-u(J_{n-1})$ are bounded.
The Burkholder estimate in Section~\ref{sec:allocation} gives
\begin{equation*}
 \sup_{n\geq1}\Ee_\kappa\!\left[
       \left\|n^{-1/2}(S_n-n\mu)\right\|^4\right]<\infty.
\end{equation*}
The squares are therefore uniformly integrable, so the functional limit
implies $n^{-1}\Cov_\kappa(S_n)\to\Gamma$.  Consequently
$\Gamma\succeq\eta(1-\eta)I_d$, and every required contrast is nondegenerate.

The same bounded martingale differences and bounded coboundary imply the
allocation moment condition of Section~\ref{sec:allocation} for every
$p_0>2$, by the moment estimate proved there.
All three principal theorems consequently cover this continuous-state
model, including its terminal law and optimal-allocation mixture.
Replacing $\kappa$ by a probability on a denumerable or finite space gives
corresponding specializations of the same construction.

\subsection{Consequences and examples}

The fixed-offset dependence already occurs in a one-state scalar model with
$\Pp(X=1)=\Pp(X=2)=1/2$.  Omitting the constant state coordinates, each
of the three marks $(1,0,1)$, $(2,0,2)$ and $(2,1,1)$ has equilibrium
probability $1/3$.  If level $k$ has mark $(2,0,2)$, level $k+1$ has mark
$(2,1,1)$ with certainty.  Corollary~\ref{cor:marked-fixed-offsets} therefore
gives limiting joint probability $1/3$ for this pair, with covariance
$2/9$ between the corresponding indicators.  The kernel $K_1$ records the
persistence of the same edge across the two observations.

For a three-coordinate critical chart with all coordinates active, consider
\(h^1=(0,1,2)\) and \(h^2=(0,2,1)\).  These corners are incomparable.
If both Gaussian selectors choose coordinate 1, their limiting terminal
marks are identical, since both observe exactly the same coordinate-1
threshold.  On every other selector combination their selected pairs are
distinct, and the marks are conditionally independent given \(G\).
Thus the limiting coupling records which constraint the two observations
have in common, not merely their separate face-selection probabilities.

The weaker assumptions include dependence constraints that rule out a
full-dimensional local central limit theorem.  For example, take one
state and \(X=(B,2-B)\), where
\(\Pp(B=0)=\Pp(B=2)=1/4\) and
\(\Pp(B=1)=1/2\).  Then \(\mu=(1,1)\) and
\begin{equation*}
 \Gamma=\frac12\begin{pmatrix}1&-1\\-1&1\end{pmatrix}.
\end{equation*}
Each scalar characteristic function has modulus less than one away from
zero, and the contrast variance is two.  The active covariance is singular
and the joint characteristic function has modulus one along
\(\theta_1=\theta_2\), so a nondegenerate full-dimensional local limit theorem cannot apply
when both coordinates are active.  Nevertheless the joint multi-time theorem
applies.  The example represents two complementary cumulative quantities
whose sum increases by exactly two at each renewal.

The law of $(G,\Xi_h,D_h,W_h)$ uses only $\mu_I$ and $\Gamma_I$.  In contrast,
$\nu_r$ retains every weight $\pi_iq_{ij}(v)$ with $v_r>0$.  Hence two
kernels can have
the same drift and long-run covariance, and therefore the same macroscopic
inverse limit, while having different terminal-state, increment, age or
overshoot laws.  The marked theorem therefore contains information that
cannot be recovered from the functional central limit theorem.

The distinction can already be seen in one-state models with bounded
increments.

\begin{example}
\label{ex:same-moments-different-edge-laws}
Let $d=2$, let $X_2=1+B$ with $B$ Bernoulli with parameter $1/2$, and take
$X_2$ independent of $X_1$.  In Model A, let
\begin{equation*}
 \Pp(X_1=1)=\frac14,\qquad
 \Pp(X_1=2)=\frac12,\qquad
 \Pp(X_1=3)=\frac14.
\end{equation*}
In Model B, let
\begin{equation*}
 \Pp(X_1=1)=\frac16,\qquad
 \Pp(X_1=2)=\frac34,\qquad
 \Pp(X_1=4)=\frac1{12}.
\end{equation*}
Both i.i.d. models have
\begin{equation*}
 \mu=(2,3/2),\qquad \Gamma=\operatorname{diag}(1/2,1/4).
\end{equation*}
They satisfy Corollary~\ref{cor:finite-marked-input}: finite support gives the
second moment, each scalar increment law has span one, and the covariance
matrix is positive definite.  Taking
$\lambda^0=(2,3/2)$ gives $\rho=1$, makes both coordinates active, and
produces the same
Gaussian selector in the two models.  Their equilibrium-edge laws are
different.  For example, the $X_1$ marginal under $\nu_1$ is
\begin{equation*}
 \left(\frac18,\frac12,\frac38\right)
 \quad\text{on }\{1,2,3\}\quad\text{in Model A},
 \qquad
 \left(\frac1{12},\frac34,\frac16\right)
 \quad\text{on }\{1,2,4\}\quad\text{in Model B}.
\end{equation*}
Thus the two processes have identical drift and long-run covariance but
different terminal marked limits.
\end{example}

\section{Second-order allocation and a kernel-sensitive example}
\label{sec:allocation}

The preceding limit theorems describe asymptotically how a rectangular budget
should be distributed among the coordinates.  The
first-order allocation equalizes the deterministic crossing times.  At the
next order the allocation has to compensate for their joint fluctuations; the
optimal correction is characterized by the probabilities of the competing
faces.  Stochastic-programming sensitivity and optimizer asymptotics provide
a broad abstract context \citep{Shapiro1991,Shapiro1993}.  Sequential resource
allocation and fair dynamic rationing provide application-level comparisons
\citep{LienIravaniSmilowitz2014,ManshadiNiazadehRodilitz2023}.  Here the decision variable is the rectangular renewal threshold,
and the limiting objective is the expected minimum of the competing
passage fluctuations.  The results below establish square-root localization of the renewal
optimizer and a uniform expansion of the expected count.  Combined with
Theorem~\ref{thm:marked-competing-face}, they identify the limiting terminal
edge law at the optimal allocation.

Fix a cost vector \(c=(c_1,\ldots,c_d)\in(0,\infty)^d\), and, for real
\(t>0\), put
\begin{equation*}
 \mathcal K_t=\{k\in\N^d:c\mathbin\cdot k\leq t\},
 \qquad
 V_t=\max_{k\in\mathcal K_t}\Ee[N(k)].
\end{equation*}
All expectations in this section are taken under one fixed initial law.
We use only the displacement Brownian limit
\begin{equation}
 G_t(s)=\frac{S_{\lfloor ts\rfloor}-\lfloor ts\rfloor\mu}{\sqrt t}
 \Rightarrow B_S(s)
 \quad\text{in }D_{\mathrm{loc}}([0,\infty),\R^d),
 \label{eq:allocation-displacement-fclt}
\end{equation}
where $B_S$ is centered Brownian motion with
$\Cov(B_S(s))=s\Gamma$.  This is the displacement marginal of
Assumption~\ref{ass:brownian-fclt}; no assumption on transition rewards is
needed in this section.  Define
\begin{equation*}
 C_\mu=c\mathbin\cdot\mu,
 \qquad
 \rho_*=C_\mu^{-1},
 \qquad
 \lambda_*=\rho_*\mu,
 \qquad
 p_r=\frac{c_r\mu_r}{C_\mu}.
\end{equation*}
Thus \(p=(p_1,\ldots,p_d)\) is a probability vector.  The elementary bound
\begin{equation*}
 \min_r\frac{k_r}{\mu_r}
 \leq \frac{c\cdot k}{c\cdot\mu}
 \leq t\rho_*
\end{equation*}
shows that \(\lambda_*\) is the unique continuous first-order balance point.
To pass from distributional convergence to an expansion of expectations we
make one additional, explicit moment assumption.  The following condition is
the moment input used throughout the allocation analysis.

\begin{assumption}
\label{ass:allocation-moment}
For some \(p_0>2\) and a finite constant \(C_{p_0}\),
\begin{equation}
 \Ee[\lVert S_n-n\mu\rVert_\infty^{p_0}]
 \leq C_{p_0}n^{p_0/2},
 \qquad n\geq1.
 \label{eq:allocation-moment}
\end{equation}
\end{assumption}

This condition supplies tail control in addition to the weak invariance
principle.  One sufficient criterion on a general state space is a bounded
measurable solution $u$ of
\begin{equation*}
 u-Pu=b,\qquad b(x)=\Ee_x[X_1]-\mu,
 \qquad \sup_x\Ee_x[|X_1|_1^{p_0}]<\infty.
\end{equation*}
Indeed,
$D_n=X_n-\mu+u(J_n)-u(J_{n-1})$ are martingale differences, and
$S_n-n\mu=\sum_{j=1}^nD_j+u(J_0)-u(J_n)$.
Their $p_0$th moments are uniformly bounded.  Applying Burkholder's
square-function inequality \citep{Burkholder1973} to each coordinate gives
\begin{equation*}
 \Ee\left[\left|\sum_{j=1}^nD_{j,r}\right|^{p_0}\right]
 \leq C_{p_0}\Ee\left[\left(\sum_{j=1}^nD_{j,r}^2\right)^{p_0/2}\right]
 \leq C_{p_0}n^{p_0/2-1}\sum_{j=1}^n\Ee[|D_{j,r}|^{p_0}]
 \leq C'_{p_0}n^{p_0/2}.
\end{equation*}
The bounded remainder yields \eqref{eq:allocation-moment} for every initial law.  In a finite irreducible state space the Poisson
equation has a bounded solution, so a finite increment moment of order
$p_0>2$ suffices.  The continuous-state example in
Section~\ref{sec:marked} also has such a solution.  More generally,
Assumption~\ref{ass:allocation-moment} is imposed directly on the original
displacement sums.  It is sufficient for the uniform integrability and
optimizer localization proved below.  Singular Brownian covariance is
allowed; a separate condition on Gaussian contrasts gives uniqueness and
the selected-face probability formula.

\subsection{An inverse-moment bound}

For real corners \(k\in[0,\infty)^d\), set
\(N(k)=N(\lfloor k\rfloor)\) and write
\begin{equation*}
 \varrho(k)=\min_{1\leq r\leq d}\frac{k_r}{\mu_r}.
\end{equation*}

The next lemma supplies uniform integrability for the centered inverse.

\begin{lemma}
\label{lem:allocation-inverse-moment}
Under Assumption~\ref{ass:allocation-moment}, for every
\(q\in(1,p_0/2)\) there is a finite \(C_q\), independent of \(k\), such that
\begin{equation}
 \Ee[\lvert N(k)-\varrho(k)\rvert^q]
 \leq C_q\{1+\varrho(k)\}^{q/2}.
 \label{eq:allocation-inverse-moment}
\end{equation}
In particular,
\begin{equation}
 \bigl|\Ee[N(k)]-\varrho(k)\bigr|
 \leq C\sqrt{1+\varrho(k)}.
 \label{eq:allocation-first-moment}
\end{equation}
\end{lemma}

\begin{proof}
Put \(r=\varrho(k)\), and choose \(j\) with \(k_j=r\mu_j\).  If \(x\geq1\)
and \(n=\lceil r+x\rceil\), monotonicity of the partial sums gives
\begin{equation*}
 \Pp(N(k)-r\geq x)
 \leq \Pp(S_n\leq_d k)
 \leq
 \Pp(S_{n,j}-n\mu_j\leq-(n-r)\mu_j).
\end{equation*}
Markov's inequality and \eqref{eq:allocation-moment} therefore imply, after
changing constants,
\begin{equation}
 \Pp(N(k)-r\geq x)
 \leq
 \begin{cases}
  C(r+1)^{p_0/2}x^{-p_0},&1\leq x\leq r+1,\\
  Cx^{-p_0/2},&x>r+1,
 \end{cases}
 \label{eq:allocation-upper-tail}
\end{equation}
where each right-hand side may be truncated at one.

For the other tail, let \(2\leq x\leq r\) and
\(n=\lfloor r-x\rfloor+1\).  The event \(r-N(k)\geq x\) entails
\(S_n\not\leq_d k\).  Since \(k_s\geq r\mu_s\) for every \(s\), and
\(r-n\geq x-1\geq x/2\),
\begin{equation}
 \Pp(r-N(k)\geq x)
 \leq\sum_{s=1}^d
   \Pp(S_{n,s}-n\mu_s>(r-n)\mu_s)
 \leq Cr^{p_0/2}x^{-p_0}.
 \label{eq:allocation-lower-tail}
\end{equation}
The lower-tail probability is zero when \(x>r\).

Use the identity
\(\Ee[Y^q]=\int_0^\infty qx^{q-1}\Pp(Y>x)\,\dd x\) for \(Y\geq0\).
For both tails, integration up to \(\sqrt{r+1}\) costs at most
\(C(r+1)^{q/2}\).  On the interval from \(\sqrt{r+1}\) to \(r+1\),
\eqref{eq:allocation-upper-tail} and
\eqref{eq:allocation-lower-tail} give
\begin{equation*}
 C(r+1)^{p_0/2}
 \int_{\sqrt{r+1}}^{r+1}x^{q-p_0-1}\,\dd x
 \leq C(r+1)^{q/2}.
\end{equation*}
Finally, the upper tail beyond \(r+1\) is integrable precisely because
\(q<p_0/2\).  This proves \eqref{eq:allocation-inverse-moment};
\eqref{eq:allocation-first-moment} follows from Lyapunov's inequality.
\end{proof}

\subsection{Uniform expansion of the expected inverse}

For the Brownian motion in \eqref{eq:allocation-displacement-fclt} and
\(h\in\R^d\), define
\begin{equation*}
 F(h)=\Ee\left[
   \min_{1\leq r\leq d}
   \frac{h_r-B_{S,r}(\rho_*)}{\mu_r}
 \right].
\end{equation*}

The following proposition upgrades the local inverse limit to a locally
uniform expansion of its expectation.

\begin{proposition}
\label{prop:allocation-expectation}
Under \eqref{eq:allocation-displacement-fclt} and
Assumption~\ref{ass:allocation-moment}, for every \(M<\infty\),
\begin{equation}
 \sup_{\substack{k\in\N^d\\
       \lVert(k-t\lambda_*)/\sqrt t\rVert_\infty\leq M}}
 \left|
   \frac{\Ee[N(k)]-t\rho_*}{\sqrt t}
   -F\left(\frac{k-t\lambda_*}{\sqrt t}\right)
 \right|\longrightarrow0.
 \label{eq:allocation-uniform-expectation}
\end{equation}
\end{proposition}

\begin{proof}
Write $\rho=\rho_*$, $\mu_{\min}=\min_r\mu_r$ and
$k_t(h)=\lfloor t\rho\mu+\sqrt t\,h\rfloor$.  For
$\|h\|_\infty\leq M$, put
\begin{equation*}
 A_t(h)=\min_r\frac{h_r-G_{t,r}(\rho)}{\mu_r},
 \qquad Z_t(h)=\frac{N(k_t(h))-t\rho}{\sqrt t}.
\end{equation*}
We first prove the uniform inverse approximation directly.  The
functional central limit theorem, whose limit has continuous paths, gives
$\sup_{\|h\|_\infty\leq M}|A_t(h)|=O_{\Pp}(1)$ and, for every fixed
$L<\infty$,
\begin{equation*}
 \omega_t(L)=\sup_{|s-\rho|\leq(L+2)/\sqrt t}
       \|G_t(s)-G_t(\rho)\|_\infty\xrightarrow{\Pp}0.
\end{equation*}
The interval in this supremum is contained in $[0,\infty)$ for all
sufficiently large $t$.  Fix $0<\delta<1$ and work on the event
\begin{equation*}
 \sup_{\|h\|_\infty\leq M}|A_t(h)|\leq L,
 \qquad \omega_t(L)+t^{-1/2}<\tfrac12\delta\mu_{\min}.
\end{equation*}
Define the integer indices
\begin{equation*}
 n_-(h)=\lfloor t\rho+\sqrt t\{A_t(h)-\delta\}\rfloor,
 \qquad
 n_+(h)=\lceil t\rho+\sqrt t\{A_t(h)+\delta\}\rceil.
\end{equation*}
They are nonnegative for all sufficiently large $t$, uniformly in $h$.
Writing $k_t(h)=t\rho\mu+\sqrt t\,h+e_t(h)$, where
$-1<e_{t,r}(h)\leq0$, the exact identity
\begin{equation*}
 \frac{S_{n,r}-k_{t,r}(h)}{\sqrt t}
 =\frac{n-t\rho}{\sqrt t}\mu_r+G_{t,r}(n/t)-h_r
       -\frac{e_{t,r}(h)}{\sqrt t}
\end{equation*}
shows that $S_{n_-(h)}\leq_d k_t(h)$ in every coordinate.  At $n_+(h)$
the inequality fails in any coordinate attaining the minimum defining
$A_t(h)$.  Indeed, the corresponding deterministic margins have absolute
size at least $\delta\mu_{\min}$, whereas the fluctuation and rounding
errors are smaller than half this quantity.  Monotonicity gives
$n_-(h)\leq N(k_t(h))<n_+(h)$ and hence
\begin{equation*}
 \sup_{\|h\|_\infty\leq M}|Z_t(h)-A_t(h)|
 \leq\delta+t^{-1/2}
\end{equation*}
on this event.  First choose $L$ large, then let $t\to\infty$ and
$\delta\downarrow0$.  We obtain
\begin{equation}
 \sup_{\|h\|_\infty\leq M}|Z_t(h)-A_t(h)|
       \xrightarrow{\Pp}0.
 \label{eq:allocation-local-in-probability}
\end{equation}

To pass to expectations, let
$k_t^\pm=\lfloor t\lambda_*\pm M\sqrt t\,\one\rfloor$.
All corners $k_t(h)$ lie between these two corners.  Therefore
\begin{equation*}
 \sup_{\|h\|_\infty\leq M}|Z_t(h)|
 \leq t^{-1/2}\max\{|N(k_t^-)-t\rho|,
                         |N(k_t^+)-t\rho|\}.
\end{equation*}
Choose $q\in(1,p_0/2)$.  Since
$\varrho(k_t^\pm)=t\rho+O(\sqrt t)$,
Lemma~\ref{lem:allocation-inverse-moment} implies that the right-hand
side has uniformly bounded $q$th moment.  Assumption~\ref{ass:allocation-moment}
also gives $\sup_t\Ee[\|G_t(\rho)\|_\infty^{p_0}]<\infty$ for large $t$.
It follows from
$|A_t(h)|\leq\mu_{\min}^{-1}(M+\|G_t(\rho)\|_\infty)$ that the
supremum in \eqref{eq:allocation-local-in-probability} is uniformly
integrable.  Its expectation consequently tends to zero.

Finally, $G_t(\rho)\Rightarrow B_S(\rho)$.  The same moment bound gives,
for each fixed $h$,
$\Ee[A_t(h)]\to F(h)$.  Both maps $h\mapsto\Ee[A_t(h)]$ and
$h\mapsto F(h)$ are $\mu_{\min}^{-1}$-Lipschitz in the supremum norm.
For a finite $\eta$-net $H_\eta$ of the box $\|h\|_\infty\leq M$,
\begin{equation*}
 \sup_{\|h\|_\infty\leq M}|\Ee[A_t(h)]-F(h)|
 \leq\max_{h\in H_\eta}|\Ee[A_t(h)]-F(h)|
          +2\eta\mu_{\min}^{-1}.
\end{equation*}
Letting $t\to\infty$ and then $\eta\downarrow0$ proves uniform
convergence.  Every integer corner in
\eqref{eq:allocation-uniform-expectation} is represented exactly by
$h=(k-t\lambda_*)/\sqrt t$.  Combining the two expectation estimates
proves the proposition.
\end{proof}

\subsection{Value and optimizer expansions}

Introduce the limiting feasible half-space and its optimizer set,
\begin{equation*}
 \mathcal H=\{h\in\R^d:c\cdot h\leq0\},
 \qquad
 F_*=\max_{h\in\mathcal H}F(h),
 \qquad
 \mathcal H_*=\argmax_{h\in\mathcal H}F(h).
\end{equation*}

The next theorem gives the second-order value and optimizer expansions.

\begin{theorem}
\label{thm:allocation}
Under \eqref{eq:allocation-displacement-fclt} and
Assumption~\ref{ass:allocation-moment}, the set \(\mathcal H_*\) is nonempty and
compact, and
\begin{equation}
 V_t=t\rho_*+\sqrt t\,F_*+o(\sqrt t).
 \label{eq:allocation-value}
\end{equation}
Every sequence of exact maximizers
\(\widehat k_t\in\argmax_{k\in\mathcal K_t}\Ee[N(k)]\) satisfies
\begin{equation}
 \dist\left(
  \frac{\widehat k_t-t\lambda_*}{\sqrt t},\mathcal H_*
 \right)\longrightarrow0.
 \label{eq:allocation-argmax}
\end{equation}
If \(\mathcal H_*=\{h_*\}\), then
\begin{equation}
 \widehat k_t=t\lambda_*+\sqrt t\,h_*+o(\sqrt t).
 \label{eq:allocation-unique}
\end{equation}
\end{theorem}

\begin{proof}
For a feasible \(k\), set
\begin{equation*}
 q_t(k)=t\rho_*-\varrho(k)\geq0,
 \qquad
 \delta(k)=k-t\lambda_*.
\end{equation*}
If \(\delta_r(k)<0\), then
\(-\delta_r(k)\leq\mu_rq_t(k)\).  If instead
\(\delta_j(k)>0\), feasibility and \(c\cdot t\lambda_*=t\) give
\begin{equation*}
 c_j\delta_j(k)
 \leq\sum_{\{r:\delta_r(k)<0\}}c_r\{-\delta_r(k)\}
 \leq(c\cdot\mu)q_t(k).
\end{equation*}
Thus, with a constant depending only on \(c\) and \(\mu\),
\begin{equation}
 \lVert k-t\lambda_*\rVert_\infty\leq Cq_t(k).
 \label{eq:allocation-deficit-controls-distance}
\end{equation}

The reference corner \(k_t^0=\lfloor t\lambda_*\rfloor\) is feasible and
satisfies \(\varrho(k_t^0)=t\rho_*+O(1)\).  By
\eqref{eq:allocation-first-moment}, an exact optimizer obeys
\begin{equation*}
 t\rho_*-C\sqrt t
 \leq \Ee[N(k_t^0)]
 \leq \Ee[N(\widehat k_t)]
 \leq t\rho_*-q_t(\widehat k_t)+C\sqrt t.
\end{equation*}
Consequently \(q_t(\widehat k_t)=O(\sqrt t)\), and
\eqref{eq:allocation-deficit-controls-distance} proves the required
square-root localization of every exact optimizer.

The map \(F\) is continuous and satisfies
\begin{equation*}
 F(h+s\mu)=F(h)+s.
\end{equation*}
Hence its maximum over \(\mathcal H\) is attained, if at all, on
\(c\cdot h=0\).  On that hyperplane set \(a_r=h_r/\mu_r\); then
\(p\cdot a=0\).  As \(\lVert a\rVert\to\infty\) within this
hyperplane, \(\min_ra_r\to-\infty\), because every \(p_r\) is
strictly positive.  Also
\begin{equation*}
 F(h)\leq \min_ra_r+\Ee\left[\max_r
 \left|\frac{B_{S,r}(\rho_*)}{\mu_r}\right|\right].
\end{equation*}
Thus \(F(h)\to-\infty\) as \(\lVert h\rVert\to\infty\) on the
budget hyperplane.  Continuity proves that
\(\mathcal H_*\) is nonempty and compact.

For every fixed \(h\in\mathcal H\), define
\(k_t(h)=\lfloor t\lambda_*+\sqrt t\,h\rfloor\), with the floor taken
coordinatewise.  Positivity holds for all large \(t\), and
\(c\cdot k_t(h)\leq c\cdot(t\lambda_*+\sqrt t\,h)\leq t\); hence
\(k_t(h)\in\mathcal K_t\), including when \(c\cdot h=0\).  Proposition~\ref{prop:allocation-expectation} therefore yields
\begin{equation*}
 \liminf_{t\to\infty}\frac{V_t-t\rho_*}{\sqrt t}\geq F(h).
\end{equation*}
Taking the supremum gives the lower bound \(F_*\).  Conversely, localization
places all scaled exact optimizers in one compact set.  From any subsequence
we may extract a further subsequence on which
\((\widehat k_t-t\lambda_*)/\sqrt t\to h\).  Feasibility implies
\(h\in\mathcal H\), and the uniform expansion identifies the corresponding
objective limit as \(F(h)\leq F_*\).  Comparison with the lower bound forces
\(F(h)=F_*\).  Hence every cluster point lies in \(\mathcal H_*\), proving
\eqref{eq:allocation-value} and \eqref{eq:allocation-argmax}.  If the
optimizer is unique, the compactness argument gives
\eqref{eq:allocation-unique}.
\end{proof}

\subsection{The cost-share law for the selected face}

Standardize the Gaussian fluctuation and the allocation correction by
\begin{equation*}
 Q_r=\frac{B_{S,r}(\rho_*)}{\mu_r},
 \qquad
 a_r=\frac{h_r}{\mu_r}.
\end{equation*}
Then the limit problem becomes
\begin{equation}
 \max_{p\cdot a\leq0}f(a),
 \qquad
 f(a)=\Ee[\min_r(a_r-Q_r)].
 \label{eq:allocation-standard-program}
\end{equation}

The objective also measures the loss caused by differences between the
coordinate passage times.  On $p\cdot a=0$, put $U_r=a_r-Q_r$.  Since
$\Ee[\sum_rp_rU_r]=0$,
\begin{equation}
 -f(a)=\Ee\left[\sum_{r=1}^dp_r
       \left(U_r-\min_{1\leq s\leq d}U_s\right)\right].
 \label{eq:allocation-competition-loss}
\end{equation}
Every summand is nonnegative.  Since the optimum is attained and every
$p_r$ is positive, $F_*=0$ holds precisely when some $a$ on $p\cdot a=0$ makes
all the $U_r$ equal almost surely.  Taking expectations forces all $a_r$
to be equal, hence zero.  Equality is therefore equivalent to almost-sure
coincidence of all normalized Gaussian coordinates $Q_r$, or to
$\Gamma=\sigma^2\mu\mu^{\mathsf T}$ for some $\sigma\geq0$.
In every other case $F_*<0$.
Consequently $-F_*$ is the smallest expected weighted discrepancy between
the normalized passage fluctuations, and it is the coefficient of the
square-root loss in the optimal expected count.  Proposition~\ref{prop:uniform-tightness-dichotomy}
shows that exactly the same covariance condition characterizes uniform
tightness of the unperturbed inverse field.  Both conclusions express
whether the coordinates share a common random time displacement.

We next characterize the optimizer through the probabilities of the selected
coordinates.

\Needspace{0.40\textheight}
\begin{proposition}
\label{prop:allocation-characterization}
The function \(f\) is concave, every optimizer in
\eqref{eq:allocation-standard-program} lies on \(p\cdot a=0\), and the
optimizer set is compact.  For \(d\geq2\), assume in addition that the
Gaussian contrast vector
\begin{equation}
 (Q_1-Q_d,\ldots,Q_{d-1}-Q_d)
 \label{eq:allocation-contrast}
\end{equation}
has nonsingular covariance.  Then the optimizer is unique, as it also is
when \(d=1\).  Its counterpart
\begin{equation*}
 h_*=(\mu_1a_{*,1},\ldots,\mu_da_{*,d})
\end{equation*}
is characterized by \(c\cdot h_*=0\) together with
\begin{equation}
 \Pp\left(
 r=\argmin_{1\leq s\leq d}
 \frac{h_{*,s}-B_{S,s}(\rho_*)}{\mu_s}
 \right)
 =p_r=\frac{c_r\mu_r}{c\cdot\mu},
 \qquad 1\leq r\leq d.
 \label{eq:allocation-cost-share}
\end{equation}
\end{proposition}

\begin{proof}
The minimum of affine functions is concave.  The translation identity
and coercivity on the budget hyperplane, proved in
Theorem~\ref{thm:allocation}, place every optimizer on $p\cdot a=0$
and make the optimizer set nonempty and compact.

For \(d=1\), the feasible boundary consists of \(a_1=0\), and the
remaining assertions are immediate.  Let \(d\geq2\).  Nonsingularity of
\eqref{eq:allocation-contrast} gives a Gaussian density
which is strictly positive everywhere on \(\R^{d-1}\).  Let \(a-b\) be
nonconstant.  Choose \(r,s\) such that
\((a_r-b_r)<(a_s-b_s)\).  The open interval
\begin{equation*}
 (a_r-a_s,\,b_r-b_s)
\end{equation*}
is nonempty.  On a nonempty open set of contrast values one can choose
\(Q_r-Q_s\) in this interval and choose all other \(Q_j\)'s sufficiently low
that \(r\) is the unique minimizer at \(a\), whereas \(s\) is the unique
minimizer at \(b\).  That open set has positive probability.  On it the
pointwise concavity inequality for the minimum is strict.  It follows that
\(f\) is strictly concave modulo the common-translation direction.  The
constraint \(p\cdot a=0\) removes that direction, and hence the optimizer is
unique.

The same nonsingularity makes all ties null events.  Dominated directional
differentiation (the difference quotient is bounded by one) gives
\begin{equation}
 \frac{\partial f}{\partial a_r}(a)
 =\Pp(r=\argmin_s(a_s-Q_s)).
 \label{eq:allocation-gradient}
\end{equation}
At the constrained maximum, \(\nabla f(a_*)=\zeta p\) for a scalar \(\zeta\).
The components in \eqref{eq:allocation-gradient} sum to one, as do those of
\(p\), so \(\zeta=1\).  This is exactly
\eqref{eq:allocation-cost-share}.  Conversely, if \(p\cdot a=0\) and
\(\nabla f(a)=p\), concavity gives, for every \(b\) on the budget hyperplane,
\begin{equation*}
 f(b)\leq f(a)+p\cdot(b-a)=f(a).
\end{equation*}
Strict concavity on that hyperplane then identifies \(a\) as the unique
optimizer.
\end{proof}

The nonsingular contrast condition also supplies the local curvature needed
to convert errors in the objective into errors in the allocation.

\begin{proposition}
\label{prop:allocation-curvature}
Let $d\geq2$, and suppose the Gaussian contrast vector in
\eqref{eq:allocation-contrast} has nonsingular covariance.  The function
$f$ in \eqref{eq:allocation-standard-program} is twice continuously
differentiable.  For $r\ne s$, put
\begin{equation*}
 w_{rs}(a)=\phi_{\sigma_{rs}}(a_s-a_r)
 \Pp\!\left(
 Q_j-Q_r\leq a_j-a_r\text{ for all }j\notin\{r,s\}
 \,\middle|\,Q_s-Q_r=a_s-a_r\right),
 \qquad \sigma_{rs}^2=\Var(Q_s-Q_r),
\end{equation*}
where $\phi_\sigma$ is the centered normal density with variance
$\sigma^2$, and the conditional probability is evaluated using the
Gaussian conditional distribution.  Then $w_{rs}=w_{sr}>0$, these weights
are continuous, and
\begin{equation}
 v^{\mathsf T}\nabla^2 f(a)v
 =-\sum_{r<s}w_{rs}(a)(v_r-v_s)^2.
 \label{eq:allocation-hessian}
\end{equation}
Consequently, for every compact $D\subset\mathcal H$ there is a constant
$c_D>0$ such that
\begin{equation}
 F(h_*)-F(h)\geq c_D\|h-h_*\|^2,
 \qquad h\in D.
 \label{eq:allocation-quadratic-growth}
\end{equation}
The conclusion \eqref{eq:allocation-quadratic-growth} also holds when $d=1$.
\end{proposition}

\begin{proof}
For a fixed $r$, the vector $Y^{(r)}=(Q_j-Q_r)_{j\ne r}$ is an
invertible linear transformation of the contrast vector in
\eqref{eq:allocation-contrast}.  It therefore has a nonsingular Gaussian
density.  The gradient formula \eqref{eq:allocation-gradient} becomes
\begin{equation*}
 \partial_r f(a)
 =\Pp\bigl(Y^{(r)}_j\leq a_j-a_r\text{ for every }j\ne r\bigr).
\end{equation*}
Differentiating this Gaussian distribution function in $a_s$, $s\ne r$,
gives $\partial_{sr}f(a)=w_{rs}(a)$.  The Gaussian conditional covariance
of the remaining contrasts is the positive definite Schur complement of
$\Var(Q_s-Q_r)$ in $\Cov(Y^{(r)})$.  Its conditional mean depends
continuously on $a_s-a_r$.  Thus each conditional probability is strictly
positive and continuous; when $d=2$ it is one.  These observations also
justify continuous differentiation of the Gaussian distribution function.
Differentiation in $a_r$ changes every threshold at rate $-1$, so
$\partial_{rr}f=-\sum_{s\ne r}w_{rs}$.  Hence $f$ is $C^2$,
$w_{rs}=w_{sr}$ by symmetry of its Hessian, and expansion of the
quadratic form proves \eqref{eq:allocation-hessian}.

On a compact set $K\subset\R^d$ the minimum
$m_K=\min_{a\in K,r<s}w_{rs}(a)$ is positive.  If $p\cdot v=0$, then
\begin{equation*}
 \sum_{r<s}(v_r-v_s)^2
 =d\|v\|^2-(\one\cdot v)^2
 \geq\frac{\|v\|^2}{\|p\|^2}.
\end{equation*}
Indeed, the projection of $\one$ onto $p^\perp$ has squared norm
$d-\|p\|^{-2}$ because $p\cdot\one=1$.
For any compact convex $K\subset\{a:p\cdot a=0\}$ containing $a_*$,
Taylor's formula with integral remainder, together with
$\nabla f(a_*)=p$, therefore gives
\begin{equation*}
 f(a_*)-f(a)\geq\frac{m_K}{2\|p\|^2}\|a-a_*\|^2,
 \qquad a\in K.
\end{equation*}

The assertion is vacuous if $D$ is empty.  Otherwise, for $h\in D$, let
$u(h)=-c\cdot h/C_\mu\geq0$ and
$b(h)=h+u(h)\mu$.  Then $c\cdot b(h)=0$ and
$F(h)=F(b(h))-u(h)$.  Take $K$ to be the convex hull of $a_*$ and the
normalized vectors $(b_r(h)/\mu_r)_{r\leq d}$, $h\in D$.
This is compact.  The preceding bound implies, with
$\kappa=m_K/(2\|p\|^2\max_r\mu_r^2)>0$,
\begin{equation*}
 F(h_*)-F(h)\geq\kappa\|b(h)-h_*\|^2+u(h).
\end{equation*}
If $U=\max_{h\in D}u(h)=0$, this proves
\eqref{eq:allocation-quadratic-growth} with $c_D=\kappa$.
If $U>0$, use
\begin{equation*}
 \|h-h_*\|^2
 \leq2\|b(h)-h_*\|^2+2Uu(h)\|\mu\|^2
\end{equation*}
and choose $c_D=\min\{\kappa/2,(2U\|\mu\|^2)^{-1}\}$.
For $d=1$, one has $h_*=0$ and $F(h)=h/\mu_1$ on $h\leq0$.
If $R=\max_{h\in D}|h|>0$, then
$F(0)-F(h)\geq |h|^2/(\mu_1R)$; if $R=0$, any positive constant works.
\end{proof}

With $I=[d]$ and $G=B_S(\rho_*)$, let $W_h$ be the selector in
\eqref{eq:gaussian-selector}.  The gradient formula gives a marginal
interpretation of the cost-share law:
$\partial F(h)/\partial h_r=\Pp(W_h=r)/\mu_r$.  Increasing a threshold
improves the limiting expected count exactly on paths for which that
coordinate is limiting.  At the optimum, the marginal gains per unit cost
are equal: $c_r^{-1}\partial F(h_*)/\partial h_r=(c\cdot\mu)^{-1}$.
The drift determines the units of each cumulative quantity, the costs
specify the budget, and the Gaussian contrasts determine the correction
needed to attain these probabilities.

The cost-share identity and the marked renewal theorem determine the
terminal law at an optimal allocation.  When an exact optimizer is not unique,
fix any deterministic selection, such as the lexicographically first one.

\begin{corollary}
\label{cor:allocation-terminal-mixture}
Under Assumptions~\ref{ass:marked-renewal} and
\ref{ass:allocation-moment}, take
$\lambda^0=\lambda_*=\rho_*\mu$, whose active set is
$I=\{1,\ldots,d\}$.  When $d\geq2$, assume additionally that the Gaussian
contrast vector in \eqref{eq:allocation-contrast} has nonsingular
covariance.  For each $t$, let
$\widehat k_t\in\argmax_{k\in\mathcal K_t}\Ee[N(k)]$ be the deterministic
selection fixed above, and let \(M_t\) denote its full terminal mark.  Let
$W_t$ be the selected crossing coordinate of that terminal edge, with the
same tie-breaking convention as in Section~\ref{sec:marked}.  Then
\begin{equation}
 \mathcal L(W_t,M_t)\Longrightarrow
 \sum_{r=1}^d p_r\,\delta_r\otimes\nu_r,
 \qquad
 p_r=\frac{c_r\mu_r}{c\cdot\mu},
 \label{eq:allocation-terminal-mixture}
\end{equation}
where \(\nu_r\) is the coordinate-\(r\) equilibrium-edge law.  In
particular, the marginal terminal-mark law is
$\sum_{r=1}^dp_r\nu_r$.
\end{corollary}

\begin{proof}
Assumption~\ref{ass:marked-renewal}(ii) supplies the displacement
Brownian limit under the fixed initial law on $\mathsf X$.  Set
\begin{equation*}
 h_t=\frac{\widehat k_t-t\lambda_*}{\sqrt t}.
\end{equation*}
The additional contrast condition permits application of
Proposition~\ref{prop:allocation-characterization} when $d\geq2$;
for $d=1$ the maximizing correction is $h_*=0$.  That proposition makes the limiting
optimizer unique, so Theorem~\ref{thm:allocation} gives
\(h_t\to h_*\).  Since \(\widehat k_t\) is integer-valued,
\(\lfloor t\lambda_*+\sqrt t\,h_t\rfloor=\widehat k_t\).  Apply the locally uniform
version of Theorem~\ref{thm:marked-competing-face} with
\(\lambda^0=\lambda_*=\rho_*\mu\); its active set is therefore
\(I=\{1,\ldots,d\}\).  For every $r$, \eqref{eq:uniform-bl-marked-kernels} gives
\begin{equation*}
 d_{\mathrm{BL}}\bigl(\mathsf K^r_{t,h_t},\mathsf K^r_{h_t}\bigr)
 \longrightarrow0.
\end{equation*}
The continuity in $h$ established after
Corollary~\ref{cor:marked-product-tests} gives
$d_{\mathrm{BL}}(\mathsf K^r_{h_t},\mathsf K^r_{h_*})\to0$.
Conditional on selection of coordinate \(r\), the terminal mark therefore
has limit \(\nu_r\); by \eqref{eq:allocation-cost-share}, the limiting
probability of this selection is \(p_r\).  Summing over the disjoint
selected-coordinate events proves
\eqref{eq:allocation-terminal-mixture}.
\end{proof}

\subsection{Closed form in two coordinates}

For \(d=2\), put
\begin{equation*}
 D=Q_1-Q_2,
 \qquad
 \sigma_D^2=\Var(D)>0,
\end{equation*}
and let \(\Phi\) and \(\phi\) denote the standard normal distribution function
and density.  Write \(z=\Phi^{-1}(p_2)\) and
\(\Delta=\sigma_Dz\).

\Needspace{14\baselineskip}
In two coordinates the limiting optimization problem has an explicit
solution.

\begin{corollary}
\label{cor:allocation-two-coordinate}
Under the hypotheses of Theorem~\ref{thm:allocation}, with \(d=2\) and
\(\sigma_D^2>0\),
\begin{equation}
 h_{*,1}=\mu_1p_2\Delta,
 \qquad
 h_{*,2}=-\mu_2p_1\Delta,
 \label{eq:allocation-hstar-two}
\end{equation}
and
\begin{equation}
 F_*=-\sigma_D\phi(\Phi^{-1}(p_2)).
 \label{eq:allocation-Fstar-two}
\end{equation}
Consequently,
\begin{equation*}
 V_t=\frac{t}{c_1\mu_1+c_2\mu_2}
 -\sqrt t\,\sigma_D\phi(\Phi^{-1}(p_2))
 +o(\sqrt t).
\end{equation*}
\end{corollary}

\begin{proof}
Coordinate \(1\) is the unique bottleneck precisely when
\begin{equation*}
 a_1-Q_1<a_2-Q_2,
 \quad\text{or equivalently}\quad
 D>a_1-a_2.
\end{equation*}
Thus \eqref{eq:allocation-cost-share} says
\(\Pp(D>\Delta)=p_1\), whence
\(\Delta=\sigma_D\Phi^{-1}(p_2)\).  Solving
\(a_1-a_2=\Delta\) and \(p_1a_1+p_2a_2=0\) gives
\(a_1=p_2\Delta\) and \(a_2=-p_1\Delta\), proving
\eqref{eq:allocation-hstar-two}.

For completeness, use
\(\min(a_1-Q_1,a_2-Q_2)=a_2-Q_2+\min(\Delta-D,0)\).
Since \(D\) and \(Q_2\) are centered and
\(\Pp(D>\Delta)=p_1\),
\begin{equation*}
 F_*
 =-p_1\Delta+\Ee[(\Delta-D)\one_{\{D>\Delta\}}]
 =-\Ee[D\one_{\{D>\Delta\}}]
 =-\sigma_D\phi(\Delta/\sigma_D),
\end{equation*}
which proves the remaining assertions.
\end{proof}

\subsection{A two-state Markov-modulated calculation}

The following example exhibits three distinct components of the theory: the
long-run covariance changes the square-root allocation, cost shares determine the selected-face probabilities, and the full transition
kernel determines the terminal equilibrium-edge mixture.  This last
dependence retains information beyond \((\mu,\Gamma)\).

The example is based on a fully specified two-state kernel.

\begin{example}
\label{ex:allocation-two-state}
Let \(E=\{0,1\}\), start the modulating chain in stationarity, and take
\begin{equation*}
 P=\begin{pmatrix}0.8&0.2\\0.3&0.7\end{pmatrix},
 \qquad
 \pi=(0.6,0.4).
\end{equation*}
Let \(B_{n,1},B_{n,2}\) be i.i.d. Bernoulli\((1/2)\), independent of the chain.
Conditional on the departure state \(J_{n-1}\), set
\begin{equation*}
 X_n=
 \begin{cases}
  (1+B_{n,1},\,3+B_{n,2}),&J_{n-1}=0,\\
  (2+B_{n,1},\,1+B_{n,2}),&J_{n-1}=1.
 \end{cases}
\end{equation*}
Equivalently, for \(b=(b_1,b_2)\in\{0,1\}^2\), the Markov-renewal kernel is
\begin{equation*}
 q_{0j}(1+b_1,3+b_2)=\frac14P_{0j},\qquad
 q_{1j}(2+b_1,1+b_2)=\frac14P_{1j},
 \qquad j\in\{0,1\},
\end{equation*}
and it vanishes elsewhere.  The kernel is strictly positive in both additive
coordinates, irreducible, aperiodic, finite-state and finite-support.  To
verify the spectral aperiodicity condition, write
\begin{equation*}
 Q(\theta)=\varphi(\theta)D(\theta)P,\qquad
 \varphi(\theta)=\prod_{\ell=1}^2
     \frac{1+e^{\mathrm i\theta_\ell}}2,\qquad
 D(\theta)=\operatorname{diag}\!\left(
 e^{\mathrm i(\theta_1+3\theta_2)},
 e^{\mathrm i(2\theta_1+\theta_2)}\right).
\end{equation*}
If \(\theta\not\equiv0\pmod{2\pi}\), then
\(|\varphi(\theta)|<1\), whereas
\begin{equation*}
 \operatorname{spr}(D(\theta)P)
 \leq \lVert D(\theta)P\rVert_\infty=1.
\end{equation*}
Thus \(\operatorname{spr}Q(\theta)<1\).  Restricting $\theta$ to each
coordinate axis verifies \eqref{eq:scalar-renewal-arithmetic}.

The conditional mean vectors are
\begin{equation*}
 m(0)=(3/2,7/2),
 \qquad
 m(1)=(5/2,3/2),
\end{equation*}
and hence
\begin{equation*}
 \mu=\pi_0m(0)+\pi_1m(1)=(1.9,2.7).
\end{equation*}
We record the covariance calculation because it shows exactly where Markov
modulation enters.  Put \(f(i)=m(i)-\mu\).  Then
\begin{equation*}
 f(0)=(-0.4,0.8),
 \qquad
 f(1)=(0.6,-1.2),
\end{equation*}
and the nontrivial eigenvalue of \(P\) is \(1/2\), with \(Pf=f/2\).  Therefore
\begin{equation*}
 \Ee_\pi[f(J_0)f(J_k)^{\mathsf T}]
 =2^{-k}\Sigma_f,
 \qquad
 \Sigma_f=
 \begin{pmatrix}0.24&-0.48\\-0.48&0.96\end{pmatrix}.
\end{equation*}
The centered Bernoulli noise is temporally independent, independent of the
chain, and has covariance \(0.25I_2\).  It follows that the per-unit-time
long-run covariance is
\begin{equation*}
 \Gamma
 =0.25I_2+
   \left(1+2\sum_{k\geq1}2^{-k}\right)\Sigma_f
 =0.25I_2+3\Sigma_f
 =\begin{pmatrix}0.97&-1.44\\-1.44&3.13\end{pmatrix}.
\end{equation*}

Take equal costs \(c=(1,1)\).  Then
\begin{equation*}
 \rho_*=\frac5{23},
 \qquad
 \lambda_*=\left(\frac{19}{46},\frac{27}{46}\right),
 \qquad
 (p_1,p_2)=\left(\frac{19}{46},\frac{27}{46}\right).
\end{equation*}
Since \(Q_r=B_{S,r}(\rho_*)/\mu_r\),
\begin{equation*}
 \sigma_D^2
 =\rho_*\left\{
   \frac{\Gamma_{11}}{\mu_1^2}
  +\frac{\Gamma_{22}}{\mu_2^2}
  -\frac{2\Gamma_{12}}{\mu_1\mu_2}
 \right\}
 =\frac{1657250}{6052887}\approx0.2738.
\end{equation*}
Equations \eqref{eq:allocation-hstar-two} and
\eqref{eq:allocation-Fstar-two} now give
\begin{equation*}
 h_*\approx(0.1282,-0.1282),
 \qquad
 F_*\approx-0.2038.
\end{equation*}
Thus
\begin{equation*}
 \widehat k_t=t\left(\frac{19}{46},\frac{27}{46}\right)
  +\sqrt t\,h_*+o(\sqrt t)
\end{equation*}
and
\begin{equation*}
 V_t=\frac5{23}t+F_*\sqrt t+o(\sqrt t).
\end{equation*}
The square-root correction transfers budget from coordinate \(2\) to
coordinate \(1\), and makes the two limiting selected-face probabilities
\(19/46\) and \(27/46\).

Finally consider the microscopic terminal edge.  Write \(I_{\mathrm{dep}}\)
for the departure-state coordinate of a terminal mark.  Marginalizing the
equilibrium-edge law \(\nu_r\) over its arrival state, increment, age and
overshoot gives
\begin{equation*}
 \nu_r(I_{\mathrm{dep}}=i)
 =\frac{\pi_i\Ee_i[X_{1,r}]}{\mu_r}.
\end{equation*}
Consequently,
\begin{equation*}
 \nu_1(I_{\mathrm{dep}}=\cdot)
 =\left(\frac9{19},\frac{10}{19}\right),
 \qquad
 \nu_2(I_{\mathrm{dep}}=\cdot)
 =\left(\frac79,\frac29\right).
\end{equation*}
At the optimal allocation, Corollary~\ref{cor:allocation-terminal-mixture}
therefore yields
\begin{equation*}
 \frac{19}{46}\left(\frac9{19},\frac{10}{19}\right)
 +\frac{27}{46}\left(\frac79,\frac29\right)
 =\left(\frac{15}{23},\frac8{23}\right).
\end{equation*}
This differs from the stationary distribution \((3/5,2/5)\).  The difference is the persistent length bias induced by selecting the edge
that crosses the optimally allocated rectangle.
\end{example}

\section{Discussion and further directions}
\label{sec:discussion}

Stratification by the active constraints links the geometry of multi-time
horizons to inverse fluctuations, terminal-edge dependence and optimal
allocation.  The inverse theorem controls changes in the active set and
retains exact values on the tie strata.  The marked theorem combines
Gaussian face selection with scalar equilibrium laws and determines when
several horizons retain a common terminal edge.  Fixed integer offsets
preserve the dependence of stationary scalar level processes.  The
allocation theorem gives a second-order description of optimal deterministic
budgets; with the marked hypotheses and nonsingular Gaussian contrasts,
the limiting terminal mixture is determined by the coordinate cost shares.
These results lead to questions about simultaneous terminal-risk assessment,
quantitative renewal errors and decisions based on the observed trajectory.

\subsection{State spaces and the renewal clock}

The standard Borel formulation includes finite and denumerable modulating
spaces and permits renewal events that leave the state unchanged.  Its
uniform zero-cluster condition controls the expected number of transitions
at each location.  Zero-displacement transitions are incorporated through the resolvent
$R_0=\sum_{m\geq0}q_0^m$ and the positive-edge kernel
$\widehat q_v=R_0q_v$, $v\ne0$.
Propositions~\ref{prop:potential} and \ref{prop:positive-edge-embedding}
give the exact relations
between potentials, counts, invariant laws and drifts for the two clocks.
The observation identity retains the original potential and the actual
departure state of the exiting edge.  A functional limit for the original
additive path retains its original clock; in the positive-edge
representation, the cycle-length reward supplies the additional count
fluctuation, as in Corollary~\ref{cor:raw-count-zero-increments}.

An edge with zero vector displacement cannot cause exit.  Formula
\eqref{eq:coordinate-palm-law} remains valid for the actual exiting edge,
with $\pi$ and $\mu$ referring to the original transition clock.
The marked limits can be obtained directly from the original path using
the scalar and regenerative conditions of Section~\ref{sec:marked}.
In the finite-state specialization, the arithmetic condition also has a
useful matrix expression.  If
$F_r(z)=\sum_vQ_vz^{v_r}$ and
$\widehat F_r(z)=R_0\sum_{v\ne0}Q_vz^{v_r}$, then
\begin{equation*}
 I-F_r(z)=(I-Q_0)\{I-\widehat F_r(z)\}.
\end{equation*}
Thus the absence of a nontrivial unit-circle zero of the scalar renewal
determinant is preserved.  The scalar renewal period is consequently compatible with both choices
of clock.  For example,
deleting the zero increments of a one-state process supported on
$\{0,1\}$ produces deterministic unit increments, which still satisfy the
scalar renewal condition.

State self-transitions with positive displacement are allowed throughout:
they record a renewal without a change of modulating state.  Zero-displacement
self-loops and zero-displacement state changes are allowed under
Assumption~\ref{ass:zero-cluster}.  A closed zero-displacement class is
excluded because it can produce infinitely many transitions at one location.
Finite irreducibility and positive drift exclude such a class and give the
uniform zero-cluster bound.  On general state spaces, including
denumerable ones, the uniform condition remains an explicit assumption
for the structural kernel identities.

\subsection{Joint observations across horizon scales}

Theorem~\ref{thm:marked-competing-face} and
Corollary~\ref{cor:marked-fixed-offsets} connect two horizon scales.
Perturbations of order $\sqrt t$ determine the Gaussian selection of a
coordinate, while fixed integer offsets retain a stationary process of scalar terminal
marks within that coordinate.  The next goal is a uniform
approximation for a growing grid of critical horizons that preserves the
random partition of observations according to their terminal edges.
Such an approximation would support simultaneous assessments of terminal
costs and remaining resources over an entire family of admissible budgets.

The stratified inverse theorem describes the deterministic tie geometry,
while the critical marked theorem resolves Gaussian face selection near it.
For growing horizon families, the main difficulty is to control this selection
together with dependence between nearby scalar levels.  Within one coordinate,
the kernel $K_r$ propagates the exact local dependence.  Across coordinates,
the normalized Gaussian contrasts locate the interfaces.  As the grid becomes
denser, observations can approach these interfaces at the same scale as
local passage fluctuations, and different coordinates can select transitions
that occur only a few steps apart.  A uniform theory must control these
interactions together with the scalar equilibrium approximation.  The quantitative mixing methods developed by D{\"o}ring and
Trottner~\cite{DoringTrottner2023} for overshoots of continuous-time
Markov-additive processes with finite modulation motivate a corresponding
analysis of the scalar level chains used here.

For averages of terminal observables over such grids, the common Gaussian
fluctuation can produce a leading random mean.  A fluctuation
theorem should identify a corresponding finite-$t$ centering and the scale
of the remaining variation.  Formula~\eqref{eq:terminal-covariance} separates
the Gaussian contribution from the contribution of a shared edge.  The offset
corollary adds the correlations of the stationary level chains: for two
observations of one chain at levels $0$ and $m\geq0$, this covariance is
$\nu_r(f_1K_r^mf_2)-\nu_r(f_1)\nu_r(f_2)$.
A limiting covariance for grid averages would have to combine these local
correlations with fluctuations of the interfaces themselves.  Establishing
the normalization and a usable covariance formula would turn the present
joint laws into uncertainty estimates for aggregate terminal losses and
simultaneous horizon comparisons.

\subsection{Quantitative renewal estimates and periodic observations}

The general-state marked theorem uses only pointwise scalar equilibration
and controls successive passage states through their limiting edge laws.
Quantitative results for growing horizon families require more: the errors
from an increasing number of conditioning steps must be controlled uniformly
as the family grows.
A useful objective is a bound with a measurable weight $V\geq1$ and
a rate $\varepsilon_r(b)\to0$:
\begin{equation*}
 \sup_{k\in\N,\ k\geq b}
 \|\mathcal L_x(V_r(k))-\nu_r\|_{\rm TV}
 \leq\varepsilon_r(b)V(x).
\end{equation*}
The aim is an explicit rate with control of the accumulated state weights.
Fuh's uniform scalar renewal estimates~\cite{Fuh2004}, obtained in a
nonlattice setting, provide a methodological comparison for this problem.
For the present lattice edge process on a general state space, one would
seek rates in terms of drift, joint-kernel minorization and increment
moments, together with control of the weights accumulated at successive
passages.
The arrival-state moment
\begin{equation*}
 \nu_r^+(V)=\frac1{\mu_r}\int\pi(\dd x)
       \sum_v\int v_rV(y)q_v(x,\dd y)
\end{equation*}
shows how passage sampling weights the states visited after large
increments.  Bounds compatible with this weighting would support
growing grids and explicit terminal-risk error estimates.

The scalar aperiodicity conditions remove persistent arithmetic phases.
When they fail, the appropriate observation retains its residue or
state-dependent lattice offset.  A joint rectangular limit should then
combine these phases with Gaussian face selection and the stationary
level-chain dependence.  Identifying this enlarged equilibrium law would
extend the theory to periodically sampled systems and make the phase
contribution to finite-budget corrections explicit.

\subsection{Statistical allocation and sequential decisions}

The allocation theorem suggests a statistical question in which the balanced
corner must itself be learned.  Write $C_\mu=c\cdot\mu$ and
$\widehat\lambda=\widehat\mu/(c\cdot\widehat\mu)$.
If an estimator based on $n_t$ observations satisfies
$\sqrt{n_t}(\widehat\mu-\mu)\Rightarrow Z$ and
$n_t/t\to\kappa\in(0,\infty)$, the ordinary delta method gives
\begin{equation*}
 \sqrt t(\widehat\lambda-\lambda_*)\Rightarrow\kappa^{-1/2}AZ,
 \qquad A=C_\mu^{-1}I-C_\mu^{-2}\mu c^{\mathsf T},\qquad c^{\mathsf T}A=0.
\end{equation*}
Thus estimation uncertainty enters exactly at the scale of the critical
allocation correction and remains tangent to the budget constraint.
When the same trajectory supplies both the estimator and the terminal
observation, their joint dependence becomes part of the decision problem.
A useful objective is a joint limit for estimation, Gaussian face selection
and terminal marks, with inferential guarantees uniform near ties.
The exact terminal survival factors in semi-Markov estimation
\citep{TrevezasLimnios2011} and stochastic-programming sensitivity
\citep{Shapiro1991,Shapiro1993} connect naturally to this question.

A quantitative theory should turn asymptotic allocation into a
finite-budget loss guarantee.  Under the nonsingular contrast condition,
let $D$ be a compact subset of $\mathcal H$ containing $h_*$ and, for all
sufficiently large $t$, the scaled correction of an exact lattice
optimizer.  Such a set exists by Theorem~\ref{thm:allocation}.  Put
\begin{equation*}
 F_t(h)=\frac{\Ee[N(\lfloor t\lambda_*+\sqrt t\,h\rfloor)]-t\rho_*}
                   {\sqrt t}.
\end{equation*}
Proposition~\ref{prop:allocation-expectation}, together with the rounding
bound
\begin{equation*}
 \left|F\!\left(\frac{\lfloor t\lambda_*+\sqrt t\,h\rfloor-t\lambda_*}{\sqrt t}\right)
       -F(h)\right|\leq\frac1{\sqrt t\,\min_r\mu_r},
\end{equation*}
gives $\epsilon_t:=\sup_{h\in D}|F_t(h)-F(h)|\to0$, and
Proposition~\ref{prop:allocation-curvature} supplies $c_D>0$.
Every maximizer $h_t$ of $F_t$ over $D$ satisfies
\begin{equation*}
 F(h_*)-F(h_t)\leq2\epsilon_t,
 \qquad \|h_t-h_*\|\leq(2\epsilon_t/c_D)^{1/2}.
\end{equation*}
The same comparison gives
\begin{equation*}
 0\leq V_t-\Ee[N(\lfloor t\lambda_*+\sqrt t\,h_*\rfloor)]
       \leq2\sqrt t\,\epsilon_t.
\end{equation*}
The remaining quantitative problem is to obtain usable rates and constants
for the uniform renewal error and optimizer localization.  The curvature
constant can be bounded through the Gaussian conditional probabilities in
\eqref{eq:allocation-hessian}.  A refinement to order one would also have
to retain lattice phases and the terminal-edge contribution alongside
higher-order path fluctuations.  Identity~\eqref{eq:allocation-competition-loss}
identifies the leading loss that such a refinement would sharpen.

Sequential decisions pose a further challenge: observations at earlier
horizons can be used to choose subsequent coordinatewise larger thresholds.
The future-quadrant kernel supplies the conditional law for these threshold
decisions, and the critical chart identifies their fluctuation scale.
Interventions that change the dynamics at renewal epochs require an
action-indexed family of joint state--increment kernels and uniform control
of their drift, moments and scalar renewal equilibration.  A controlled
renewal theory would seek policies with asymptotically optimal expected
counts and quantified terminal risk, including the effect of learning on
those policies.

\subsection{Joint crossings for discontinuous cumulative paths}

Cumulative quantities with jump-dominated scaling limits lead to a
different geometric question: one jump can exhaust several resources and
can be the terminal transition for many horizons.  For a vector subordinator
$D$ with unbounded coordinate paths, define
\begin{equation*}
 \tau(\lambda)=\inf\{s:D(s)\notleqd\lambda\},\qquad
 C_D(\lambda)=\{r:D_r(\tau(\lambda))>\lambda_r\}.
\end{equation*}
At a horizon crossed by a jump, the natural limiting observation retains
$\tau(\lambda)$, the two positions $D(\tau(\lambda)-)$ and
$D(\tau(\lambda))$, and the complete crossed-coordinate set
$C_D(\lambda)$.  A field-level theorem must also specify one-sided
boundary conventions at horizons lying on the range of $D$.
Coupled walk--clock limits and pre- and post-renewal observations are
studied by Straka and Henry~\cite{StrakaHenry2011} under a scalar-clock
formulation.  For rectangular horizons, the objective is the
joint marked field and the random partition of horizon space according to
the jump that causes exit.

Such a theory would describe clusters of simultaneous failures or resource
exhaustions through the joint jump measure.  It would also provide a basis
for allocation criteria that account for rare transitions affecting several
coordinates at once.  The required convergence must preserve shared jumps,
pre- and post-jump positions, and the dependence between incomparable
horizons; these features determine the appropriate topology and terminal
observation law.

\appendix
\section{Supporting results for the multi-time observation field}
\label{app:field-structure}

This appendix proves the exact structural statements used in
Sections~\ref{sec:kernel} and \ref{sec:augmentation} and gives the full
positive-edge representation when zero displacements occur.  Throughout,
Assumption~\ref{ass:zero-cluster} holds on the standard Borel state space
introduced in Section~\ref{sec:kernel}.

\subsection{The kernel potential}

We prove Proposition~\ref{prop:potential}.

\begin{proof}
Induction using the deterministic-time kernel identity
\eqref{eq:kernel-transition} gives the convolution law
\eqref{eq:convolution-law}.

We next remove the zero clusters.  Since
$\sum_{v\neq0}q_v\one=(I-q_0)\one$, monotone convergence and telescoping give
\begin{equation*}
 \sum_{v\neq0}\widehat q_v\one=R_0(I-q_0)\one=\one.
\end{equation*}
Thus $\widehat q$ is stochastic and strict.  Every path contributing to
$\psi_k$ has at most $|k|_1$ nonzero increments.  Decompose it uniquely into
a zero cluster preceding each nonzero increment and a final zero cluster.
Tonelli's theorem gives
\begin{equation*}
 \psi_k=\sum_{j=0}^{|k|_1}(\widehat q^{(*j)})_kR_0
       =\widehat\psi_kR_0.
\end{equation*}
The strict potential $\widehat\psi_k$ is a finite sum of subprobability
kernels, and $R_0$ is bounded by \eqref{eq:zero-cluster-bound}; hence
$\psi_k$ is finite.  Summing
$q^{(*n)}=q*q^{(*(n-1))}=q^{(*(n-1))}*q$ and using monotone convergence gives
\eqref{eq:resolvent}.

If $V=e+q*V$ and $V$ is nonnegative, iteration gives
$V\geq\sum_{j=0}^m q^{(*j)}$ for every $m$.  Letting $m\to\infty$ at
coefficient $k$ proves $V\geq\psi$ and hence minimality.  The right renewal
equation is treated by the same iteration.
\end{proof}

The zero-cluster assumption also gives bounds uniform in the initial state.

\begin{lemma}
\label{lem:zero-cluster-bounds}
Under Assumption~\ref{ass:zero-cluster}, there are constants
$a>0$ and $C<\infty$ such that the number $L$ of transitions up to and
including the first nonzero displacement satisfies
\begin{equation*}
 \sup_{x\in\mathsf X}\Ee_x[e^{aL}]\leq C.
\end{equation*}
Moreover, for every $k\in\N^d$,
\begin{equation*}
 \sup_x\psi_k(x,\mathsf X)\leq C_0,
 \qquad
 \sup_x\Ee_x[T(k)]\leq C_0(|k|_1+1).
\end{equation*}
\end{lemma}

\begin{proof}
The tail of $L$ is $\Pp_x(L>n)=q_0^n\one(x)$.  Since these tails
decrease in $n$, the zero-cluster bound gives
$(n+1)\sup_xq_0^n\one(x)\leq C_0$.  Choose an integer $m\geq1$
such that $C_0/(m+1)\leq1/2$.  Kernel composition then gives
\begin{equation*}
 \sup_x\Pp_x(L>jm+r)\leq2^{-j},
 \qquad j\geq0,\quad 0\leq r<m.
\end{equation*}
For $0<a<(\log2)/m$, the tail-sum identity
\begin{equation*}
 \Ee_x[e^{aL}]
 =1+(e^a-1)\sum_{n\geq0}e^{an}\Pp_x(L>n)
\end{equation*}
proves the first assertion, uniformly in $x$.

The positive-edge path visits each exact location $k$ at most once,
because its $\ell^1$ displacement increases strictly at every edge.
Consequently $\widehat\psi_k\one\leq1$.  The factorization
$\psi_k=\widehat\psi_kR_0$ therefore yields
$\psi_k\one\leq C_0$.
The indices $\tau_j$ of successive nonzero increments are stopping
times, since
\begin{equation*}
 \{\tau_j\leq n\}
 =\left\{\sum_{i=1}^n\one_{\{X_i\ne0\}}\geq j\right\}.
\end{equation*}
Applying the deterministic-time transition identity on each event
$\{\tau_{j-1}=n\}$ gives
$\Ee[L_j\mid\mathcal F_{\tau_{j-1}}]
=R_0\one(J_{\tau_{j-1}})\leq C_0$.
At most $|k|_1$ nonzero increments can be included in the rectangle;
thus $T(k)\leq\tau_{|k|_1+1}$ and the second bound follows by summing
these conditional expectations.
\end{proof}

\subsection{Positive-edge embedding and the original transition count}

The factorization in Proposition~\ref{prop:potential} does not identify the
number of original transitions with the number of positive-displacement
edges.  The following exact representation records the difference and also
retains the state immediately before a positive edge.

Let $\tau_0=0$ and, for $m\geq1$, define
\begin{equation*}
 \tau_m=\inf\{n>\tau_{m-1}:X_n\neq0\},\quad
 L_m=\tau_m-\tau_{m-1},\quad
 \widehat J_m=J_{\tau_m},\quad \widehat S_m=S_{\tau_m}.
\end{equation*}
Put $\widehat J_0=J_0$, $\widehat S_0=0$, and
\begin{equation*}
 M(k)=\sup\{m\geq0:\widehat S_m\leqd k\}.
\end{equation*}
For $\ell\geq1$ and $v\neq0$, the marked positive-edge kernel is
\begin{equation}
 \mathcal C_{\ell,v}(x,\dd h,\dd y)
 =q_0^{\ell-1}(x,\dd h)q_v(h,\dd y).
 \label{eq:positive-cycle-kernel}
\end{equation}
Its mark $h$ is the actual state immediately before the positive edge.

\begin{proposition}
\label{prop:positive-edge-embedding}
Under Assumption~\ref{ass:zero-cluster}, the variables $\tau_m$ are finite
almost surely, the unmarked transition kernel of
$(\widehat J_m,\widehat S_m)$ is $\widehat q$ from
\eqref{eq:collapsed-kernel}, and
\begin{equation}
 T(k)=\tau_{M(k)+1},\qquad
 N(k)=\sum_{m=1}^{M(k)+1}L_m-1.
 \label{eq:raw-positive-count-relation}
\end{equation}
On the standard Borel state space, let
$\widehat P=R_0(P-q_0)$ and suppose that $\widehat\pi$ is an invariant
probability distribution for $\widehat P$.  Then, with
\begin{equation*}
 \overline\ell=\widehat\pi R_0\one,\qquad
 \pi=\frac{\widehat\pi R_0}{\overline\ell},
\end{equation*}
$\overline\ell=\Ee_{\widehat\pi}[L_1]$, and $\pi$ is invariant for $P$.
For this invariant distribution, define the mean displacements per original
transition and per positive edge by setting
\begin{equation*}
 D(x)=\sum_{v\in\N^d}q_v(x,\mathsf X)v,\qquad
 \mu=\pi D,\qquad \widehat\mu=\widehat\pi R_0D,
\end{equation*}
where the kernels act componentwise on the vector-valued function $D$.
If $\mu\in(0,\infty)^d$, then, for every
$\lambda\in(0,\infty)^d$,
\begin{equation}
 \widehat\mu=\overline\ell\,\mu,\qquad
 \overline\ell\min_r\frac{\lambda_r}{\widehat\mu_r}
 =\min_r\frac{\lambda_r}{\mu_r}.
 \label{eq:raw-collapsed-rates}
\end{equation}
\end{proposition}

\begin{proof}
The mass of paths having $\ell-1$ zero increments followed by a nonzero
increment is \eqref{eq:positive-cycle-kernel}.  Summing over $\ell$ and the
pre-edge state gives $R_0q_v=\widehat q_v$.  Proposition~\ref{prop:potential}
shows that $\widehat q$ is stochastic.  Moreover,
$q_0^n\one(x)\to0$ for every $x$, because
$\sum_nq_0^n\one(x)<\infty$.  Conditional on the state at each
$\tau_{m-1}$, the next positive edge is therefore reached almost surely;
induction gives $\tau_m<\infty$ for every $m$.  The $(M(k)+1)$st positive edge is the first
one leaving the rectangle.  Counting all zero and nonzero transitions up to,
but not including, that edge gives \eqref{eq:raw-positive-count-relation}.

For the invariant-law identities, $w=\widehat\pi R_0$ is a finite
measure by Assumption~\ref{ass:zero-cluster}.  The kernel resolvent
identities give $w(I-q_0)=\widehat\pi$ and
$\widehat\pi=\widehat\pi R_0(P-q_0)=w(P-q_0)$; subtraction yields
$wP=w$.
The tail-sum identity gives $\Ee_x[L_1]=R_0\one(x)$, so
$\Ee_{\widehat\pi}[L_1]=\overline\ell$.  Normalization by
$w\one=\overline\ell$ proves invariance of $\pi$.
Furthermore,
\begin{equation*}
 \widehat\mu=\widehat\pi R_0D
 =\overline\ell\,\pi D
 =\overline\ell\,\mu.
\end{equation*}
The second relation in \eqref{eq:raw-collapsed-rates} follows by dividing
$\widehat\mu=\overline\ell\mu$ componentwise.
\end{proof}

\subsection{Conditioning on the lower-rectangle past}

The following factorization justifies conditioning at the last included
index without assuming that this index is a stopping time.

\begin{lemma}
\label{lem:censoring}
Let $\mathcal F_n=\sigma(J_0,S_0,\ldots,J_n,S_n)$ and
\begin{equation*}
\mathcal G_k^Y=\sigma\{Y_m=(Z_m,U_m):m\leqd k\}.
\end{equation*}
For every cylinder event $A\in\mathcal G_k^Y$ and every $n,u$, there is a
bounded $\mathcal F_n$-measurable variable $H_{n,u,A}$ such that
\begin{equation}
 \one_A\one_{\{N(k)=n,U_k=u\}}
 =H_{n,u,A}\one_{\{S_n=k-u\}}
  \one_{\{X_{n+1}\notleqd u\}}.
 \label{eq:cylinder-factorization}
\end{equation}
Thus, on the event $\{N(k)=n,U_k=u\}$, the only information about the next
edge contributed by the lower-rectangle past is its censoring by
$\{X_{n+1}\notleqd u\}$.  In particular, for every bounded measurable
$a(y,v)$,
\begin{equation}
\Ee\!\left[
 a(J_{n+1},X_{n+1})\one_{\{X_{n+1}\notleqd u\}}
 \mid\mathcal F_n\right]
=\sum_{v\notleqd u}\int_{\mathsf X}a(y,v)q_v(J_n,\dd y).
\label{eq:censoring-unnormalized}
\end{equation}
On $\{\overline H_u(J_n)>0\}$, the corresponding conditional integral after
censoring is
\begin{equation}
 q^{\langle u\rangle}a(J_n)
 =\frac{1}{\overline H_u(J_n)}
   \sum_{v\notleqd u}\int_{\mathsf X}a(y,v)q_v(J_n,\dd y).
\label{eq:censoring-law}
\end{equation}
\end{lemma}

\begin{proof}
For fixed $n$ the equality
\begin{equation*}
\{N(k)=n\}=\{S_n\leqd k,\ X_{n+1}\notleqd k-S_n\}
\end{equation*}
follows from nonnegative increments.  To make the measurability explicit,
for $m\leqd k$ define the truncated index
\begin{equation*}
 N^{[n]}(m)=\max\{0\leq j\leq n:S_j\leqd m\}
\end{equation*}
and the corresponding $\mathcal F_n$-measurable pair
$Y_m^{[n]}=(J_{N^{[n]}(m)},m-S_{N^{[n]}(m)})$.  If
$S_n=k-u$ and $X_{n+1}\notleqd u$, then $S_{n+1}\notleqd k$; it is therefore
outside every lower rectangle with corner $m\leqd k$, and later sums cannot
re-enter.  Hence $Y_m=Y_m^{[n]}$ for all $m\leqd k$.  For a cylinder $A$,
take $H_{n,u,A}$ to be its indicator evaluated at the finitely many
$Y_m^{[n]}$ that define $A$.  This proves
\eqref{eq:cylinder-factorization} without conditioning on a point value of
$J_n$.

For a bounded $\mathcal F_n$-measurable variable $H$, the deterministic-time
Markov property gives
\begin{equation*}
 \Ee\bigl[H a(J_{n+1},X_{n+1}) \one_{\{X_{n+1}\notleqd u\}}\bigr] = \Ee\left[H\sum_{v\notleqd u} \int a(y,v)q_v(J_n,\dd y)\right].
\end{equation*}
This proves \eqref{eq:censoring-unnormalized}.  Dividing by the conditional
probability $\overline H_u(J_n)$ on the admissible set gives
\eqref{eq:censoring-law}.  Testing the resulting identity after multiplication by
$H_{n,u,A}\one_{\{S_n=k-u\}}$, then summing over $n,u$, proves the asserted
conditional identity for all past cylinders.  A monotone-class argument
extends it to $\mathcal G_k^Y$.
\end{proof}

We now prove Proposition~\ref{thm:future-quadrant}, including the count-augmented assertion.

\begin{proof}
We first treat a path with its original initial law.  Fix $k$, a past
cylinder event $A\in\mathcal G_k^Y$, and a bounded future cylinder
function $F$.  For $n\geq0$ and $u\leqd k$, set
\begin{equation*}
 D_{n,u}=\{S_n=k-u\},\qquad
 E_{n,u}=D_{n,u}\cap\{X_{n+1}\notleqd u\}.
\end{equation*}
Nonnegative increments give
$E_{n,u}=\{N(k)=n,U_k=u\}$; these events partition the sample space
up to a null set.  By Lemma~\ref{lem:censoring}, there is an
$\mathcal F_n$-measurable indicator $H_{n,u,A}$ such that
$\one_A\one_{E_{n,u}}
=H_{n,u,A}\one_{D_{n,u}}\one_{\{X_{n+1}\notleqd u\}}$.

On $D_{n,u}$, every corner $k+h$ contains the renewal point $S_n$.
Hence, on $E_{n,u}$, its last included index and state depend only on
$J_n$ and the path after index $n$.  Translate that suffix by $k$.
It starts at $-u$, and its first edge is constrained by
$X_{n+1}\notleqd u$.  Iterating
\eqref{eq:kernel-transition} for cylinder functions of the suffix gives
\begin{equation*}
 \one_{D_{n,u}}\Ee\!\left[
  \one_{\{X_{n+1}\notleqd u\}}
  F((Y_{k+h})_h)\mid\mathcal F_n\right]
 =\one_{D_{n,u}}\overline H_u(J_n)\mathsf R_{J_n,u}F.
\end{equation*}
Here the right side is defined to be zero when
$\overline H_u(J_n)=0$; the left side is then zero as well.
For completeness, the factor $\overline H_u(J_n)$ is the total mass
of the allowed first edges, division by it gives
$q^{\langle u\rangle}$, and every subsequent transition has kernel
$q$.  Bounded measurable suffix functionals follow from the cylinder
identity by the monotone-class theorem.
It follows that
\begin{eqnarray*}
 \Ee[\one_A\one_{E_{n,u}}F((Y_{k+h})_h)]
 &=&\Ee[H_{n,u,A}\one_{D_{n,u}}
       \overline H_u(J_n)\mathsf R_{J_n,u}F]\\
 &=&\Ee[\one_A\one_{E_{n,u}}\mathsf R_{Y_k}F].
\end{eqnarray*}
The second equality follows by conditioning its right side on
$\mathcal F_n$ and using $Y_k=(J_n,u)$ on $E_{n,u}$.
Summing over $n,u$ is justified by boundedness and the disjointness
of the events $E_{n,u}$.  It proves
\begin{equation*}
 \Ee[\one_AF((Y_{k+h})_h)]
 =\Ee[\one_A\mathsf R_{Y_k}F].
\end{equation*}
Two applications of the monotone-class theorem extend this identity
first to all $A\in\mathcal G_k^Y$ and then to all bounded Borel $F$
on $\mathcal E^{\N^d}$.  The latter product has a countable index set,
and its Borel sigma-field is generated by its cylinders.  This proves
\eqref{eq:future-quadrant}.

For a residual start at $(x,u_0)$, use the same partition according to
the last included suffix index.  On the terms with index at least one,
the first residual edge is already in the past, and all future
transitions have kernel $q$, so the calculation above applies.  On
the term with index zero, the departure state is $x$, the age at $k$
is $u_0+k$, and every smaller corner also has index zero.
The first-edge law is $q^{\langle u_0\rangle}$, with the additional
constraint $v\notleqd u_0+k$.  Since
$\{v\notleqd u_0+k\}\subseteq\{v\notleqd u_0\}$, its normalized
law is exactly $q^{\langle u_0+k\rangle}$ whenever the additional
event has positive probability.  This proves the same tested
identity at every residual start.  Integration over the initial
pair proves the assertion for every initial distribution on
$\mathcal E$.

For the count-augmented assertion, include the truncated indices
$N^{[n]}(m)$ in the past cylinders of Lemma~\ref{lem:censoring}.
They are $\mathcal F_n$-measurable, so the same factorization holds.
On $E_{n,u}$, couple the residual path with the actual suffix after
index $n$; its count of included new edges at corner $h$ then satisfies
\begin{equation*}
 N(k+h)=n+\mathcal N_h^{J_n,u},\qquad h\in\N^d.
\end{equation*}
The preceding tested-expectation calculation, with the absolute
count $n$ retained in the future functional, identifies its
conditional law as $\mathsf R^N_{n,J_n,u}$.  Summing over $n,u$ and
using monotone classes proves the count-augmented assertion.
\end{proof}

\subsection{The one-point semigroup and state projection}

We first prove Proposition~\ref{prop:one-point-semigroup}.

\begin{proof}
Measurability follows from measurability of $\overline H$, $q_v$ and $K_t$,
and the countable sum.  The two disjoint sets
\begin{equation*}
\{v\notleqd u+h\},
\qquad
\{v\leqd u+h,\ v\notleqd u\}
\end{equation*}
partition $\{v\notleqd u\}$.  Hence $P_h^Y((x,u),\mathcal E)=1$; the same
decomposition shows that its support is admissible.  Taking
$F(\omega)=f(\omega_h)$ in Proposition~\ref{thm:future-quadrant} proves
\eqref{eq:one-point-markov}.  Starting from an arbitrary residual pair,
apply this identity successively at $0$ and $h$ and use the tower property
to obtain $P_{h+r}^Y=P_h^YP_r^Y$ pointwise on $\mathcal E$.  Reversing $h$ and
$r$ gives $P_{h+r}^Y=P_r^YP_h^Y$.
\end{proof}

We next prove Proposition~\ref{thm:lumpability}.

\begin{proof}
Assume (L2).  For bounded measurable $f$, Proposition
\ref{prop:one-point-semigroup} gives
\begin{equation*}
\Ee[f(Z_{k+h})\mid\mathcal G_k^Y]
=\Lambda_hf(Y_k)=P_h^Zf(Z_k).
\end{equation*}
The last variable is $\mathcal G_k^Z$-measurable.  Conditioning once more on
$\mathcal G_k^Z$ proves (L1).  The semigroup identity follows by applying
\eqref{eq:commuting-semigroup} to functions constant in the age coordinate.

Conversely, start the residual field first from $(x,u)$ and then from
$(x,u')$.  The common kernel required in (L1) gives
\begin{equation*}
\Lambda_h((x,u),A)=P_h^Z(x,A)=\Lambda_h((x,u'),A),
\end{equation*}
which is (L2).  Evaluation at age zero supplies the measurable common
kernel on $\mathsf X_{\mathcal E}$.

If every $P_{e_i}^Y$ maps functions constant on the fibers
$\{x\}\times\N^d$ to fiber-constant functions, then so does every product
$(P_{e_1}^Y)^{h_1}\cdots(P_{e_d}^Y)^{h_d}=P_h^Y$.  This proves sufficiency of the
coordinate test.  Necessity follows by specializing \eqref{eq:lumpability} to
$h=e_i$, for each $i$.
\end{proof}

\section{Functional-limit consequences and sufficient conditions}
\label{app:probabilistic-input}

The main inverse theorem takes a continuous functional limit as its input.
This appendix first records its elementary consequences, then treats the
original count when zero-displacement transitions are removed and
establishes the Brownian functional limit for finite-state and regenerative
kernels.

\subsection{Consequences of a continuous path limit}

\begin{lemma}
\label{lem:fclt-consequences}
Under Assumption~\ref{ass:general-fclt}, for every $T<\infty$,
\begin{eqnarray}
\sup_{0\leq s\leq T}\|G_t(s)\|&=&O_{\Pp}(1),
\label{eq:G-bounded}\\
\sup_{0\leq s\leq T}
 \left\|\frac{S_{\lfloor ts\rfloor}}{t}-s\mu\right\|
 &\xrightarrow{\Pp}&0,
\label{eq:functional-lln}\\
\max_{1\leq n\leq\lfloor tT\rfloor}
 \frac{\|X_n-\mu\|}{a_t}
 &\xrightarrow{\Pp}&0.
\label{eq:maximal-jump}
\end{eqnarray}
Moreover, $G_t$ is asymptotically uniformly equicontinuous on $[0,T]$:
for every $\varepsilon>0$,
\begin{equation}
\lim_{\delta\downarrow0}\limsup_{t\to\infty}
\Pp\left(
\sup_{\substack{r,s\in[0,T]\\ |r-s|\leq\delta}}
\|G_t(r)-G_t(s)\|>\varepsilon
\right)=0.
\label{eq:stoch-equicontinuity}
\end{equation}
\end{lemma}

\begin{proof}
In a Skorokhod representation of the $J_1$ convergence, continuity of the
limiting path allows the time changes to be removed, so the represented
paths converge uniformly on $[0,T]$.  The resulting supremum and modulus
bounds prove \eqref{eq:G-bounded} and
\eqref{eq:stoch-equicontinuity}.  Since
\begin{equation*}
\frac{S_{\lfloor ts\rfloor}}{t}-s\mu
=\frac{a_t}{t}G_t(s)
 +\left(\frac{\lfloor ts\rfloor}{t}-s\right)\mu,
\end{equation*}
\eqref{eq:scale} and \eqref{eq:G-bounded} give
\eqref{eq:functional-lln}.  Finally, every vector
$(X_n-\mu)/a_t$ with $n\leq\lfloor tT\rfloor$ is a jump of $G_t$ on
$[0,T]$.  The largest jump of a sequence converging in $J_1$ to a continuous
limit tends to zero in probability.  This proves
\eqref{eq:maximal-jump}.
\end{proof}

\subsection{Zero increments and raw transition counts}
\label{subsec:zero-increment-counts}

When $q_0$ is nonzero, the positive-edge inverse $M(k)$ from
Section~\ref{sec:kernel} does not by itself determine the fluctuations of the
raw transition count $N(k)$.  The random lengths of the intervening
zero-increment clusters contribute an additional, generally correlated,
term.  The following corollary states precisely the probabilistic input
needed to retain that term.

Let $(\widehat J_m,\widehat S_m)_{m\geq0}$ and $(L_m)_{m\geq1}$ be the
positive-edge embedding of Proposition~\ref{prop:positive-edge-embedding}, and let
$\widehat\mu\in(0,\infty)^d$ and $\overline\ell\in(0,\infty)$.  Define
the centered processes
\begin{equation*}
 \widehat G_t(s)
 =\frac{\widehat S_{\lfloor ts\rfloor}
        -\lfloor ts\rfloor\widehat\mu}{\sqrt t},
 \qquad
 H_t^L(s)
 =\frac{1}{\sqrt t}\sum_{m=1}^{\lfloor ts\rfloor}
       (L_m-\overline\ell).
\end{equation*}
Suppose that, under the initial law under consideration,
\begin{equation}
 (\widehat G_t,H_t^L)
 \Rightarrow(\widehat B_S,B_L)
 \quad\text{in }D_{\mathrm{loc}}([0,\infty),\R^{d+1}),
 \label{eq:positive-edge-joint-fclt}
\end{equation}
where $(\widehat B_S,B_L)$ is a centered Brownian motion, with arbitrary
cross-covariance between its two components.  We also make explicit the
maximal-cycle condition
\begin{equation}
 \max_{1\leq m\leq\lfloor tT\rfloor+1}
       \frac{L_m}{\sqrt t}\xrightarrow{\Pp}0
 \qquad\text{for every }T<\infty.
 \label{eq:maximal-positive-cycle}
\end{equation}
The condition in \eqref{eq:maximal-positive-cycle} follows from
\eqref{eq:positive-edge-joint-fclt} whenever the second coordinate is formed
from the displayed step-function partial sums, but recording it separately
makes the treatment of the terminal positive-edge cycle transparent.

The raw transition count then has the following locally uniform critical
limit.

\Needspace{0.42\textheight}
\begin{corollary}
\label{cor:raw-count-zero-increments}
Assume \eqref{eq:positive-edge-joint-fclt} and
\eqref{eq:maximal-positive-cycle}.  Fix $\widehat\rho>0$, and, for
$h\in\R^d$, set
\begin{equation*}
 k_t(h)=\left\lfloor
 t\widehat\rho\,\widehat\mu+\sqrt t\,h
 \right\rfloor,
\end{equation*}
where the floor is coordinatewise.  For every compact $H\subset\R^d$,
\begin{equation}
 \left(
 \frac{N(k_t(h))-t\overline\ell\,\widehat\rho}{\sqrt t}
 \right)_{h\in H}
 \Rightarrow
 \left(
 B_L(\widehat\rho)
 +\overline\ell\min_{1\leq r\leq d}
   \frac{h_r-\widehat B_{S,r}(\widehat\rho)}{\widehat\mu_r}
 \right)_{h\in H}
 \label{eq:raw-count-critical-limit}
\end{equation}
in $\ell^\infty(H)$.  Thus the cycle lengths contribute an additional
Brownian term, jointly dependent on the displacement limit, which cannot
in general be omitted.  More generally, for
$\lambda^0\in(0,\infty)^d$, define
\begin{equation*}
 \widehat\rho_0=\min_r\frac{\lambda_r^0}{\widehat\mu_r},
 \qquad
 \widehat I_0
 =\left\{r:\frac{\lambda_r^0}{\widehat\mu_r}
                  =\widehat\rho_0\right\}.
\end{equation*}
If $k_t(h)=\lfloor t\lambda^0+\sqrt t\,h\rfloor$, the same assertion holds
in $\ell^\infty(H)$ in the explicit form
\begin{equation}
 \left(
 \frac{N(k_t(h))-t\overline\ell\,\widehat\rho_0}{\sqrt t}
 \right)_{h\in H}
 \Rightarrow
 \left(
 B_L(\widehat\rho_0)
 +\overline\ell\min_{r\in\widehat I_0}
   \frac{h_r-\widehat B_{S,r}(\widehat\rho_0)}
        {\widehat\mu_r}
 \right)_{h\in H}.
 \label{eq:raw-count-general-critical-limit}
\end{equation}
\end{corollary}

\begin{proof}
Write
\begin{equation*}
 R_t(h)=\frac{M(k_t(h))-t\widehat\rho}{\sqrt t}.
\end{equation*}
Monotonicity of $\widehat S_m$, tightness in
\eqref{eq:positive-edge-joint-fclt}, and positivity of every coordinate of
$\widehat\mu$ first give
$\sup_{h\in H}|R_t(h)|=O_{\Pp}(1)$ by comparing, uniformly over $h\in H$,
the two deterministic indices
$\lfloor t\widehat\rho\pm L\sqrt t\rfloor$ and then letting $L$ increase.
For a fixed such $L$, stochastic equicontinuity gives
\begin{equation*}
 \sup_{|s-\widehat\rho|\leq(L+2)/\sqrt t}
 \|\widehat G_t(s)-\widehat G_t(\widehat\rho)\|
 \xrightarrow{\Pp}0.
\end{equation*}
At an integer index
$m=t\widehat\rho+\sqrt t\,z+O(1)$ with $|z|\leq L+1$, the coordinatewise
inequality $\widehat S_m\leq_d k_t(h)$ is equivalent to
\begin{equation*}
 z\widehat\mu_r
 +\widehat G_{t,r}\left(\widehat\rho+\frac{z}{\sqrt t}
                         +O(t^{-1})\right)
 \leq h_r+O(t^{-1/2}),
 \qquad 1\leq r\leq d.
\end{equation*}
Define
\begin{equation*}
 A_t(h)=\min_{1\leq r\leq d}
 \frac{h_r-\widehat G_{t,r}(\widehat\rho)}{\widehat\mu_r}.
\end{equation*}
Fix $\delta>0$.  On the preceding tightness and equicontinuity event, take
$m_-(h)=\lfloor t\widehat\rho+\sqrt t\{A_t(h)-\delta\}\rfloor$ and
$m_+(h)=\lceil t\widehat\rho+\sqrt t\{A_t(h)+\delta\}\rceil$.
The displayed inequalities hold in every coordinate at $m_-(h)$, whereas
at $m_+(h)$ they fail in a coordinate attaining the minimum in $A_t(h)$,
uniformly for $h\in H$.  Monotonicity therefore places $M(k_t(h))$ between
these two indices.  First letting $t\to\infty$ and then
$\delta\downarrow0$, while the rounding errors in the index and threshold are
$O(t^{-1/2})$, yields
\begin{equation}
 R_t(h)
 =\min_{1\leq r\leq d}
   \frac{h_r-\widehat G_{t,r}(\widehat\rho)}{\widehat\mu_r}
   +o_{\Pp}(1)
 \quad\text{in }\ell^\infty(H).
 \label{eq:positive-edge-count-expansion}
\end{equation}

The exact identity \eqref{eq:raw-positive-count-relation} gives
\begin{equation}
 \frac{N(k_t(h))-t\overline\ell\,\widehat\rho}{\sqrt t}
 =H_t^L\left(\frac{M(k_t(h))+1}{t}\right)
  +\overline\ell R_t(h)
  +\frac{\overline\ell-1}{\sqrt t}.
 \label{eq:raw-count-time-change}
\end{equation}
Uniformly on $H$, the random time in
\eqref{eq:raw-count-time-change} converges in probability to
$\widehat\rho$.  Joint convergence, stochastic equicontinuity of $H_t^L$,
and \eqref{eq:maximal-positive-cycle} show that the first term converges to
$B_L(\widehat\rho)$, jointly with the right side of
\eqref{eq:positive-edge-count-expansion}.  Substitution proves
\eqref{eq:raw-count-critical-limit}.  At a general $\lambda^0$, every
coordinate outside $\widehat I_0$ has a deterministic slack of order $t$.  It is
therefore inactive uniformly on bounded $h$-sets, and the same argument gives
\eqref{eq:raw-count-general-critical-limit}.
\end{proof}

\subsection{Finite-state sufficient conditions}

We give a finite-state condition for the Brownian input.  It supplies the
functional limit used in Corollary~\ref{cor:finite-marked-input}.

Let $\mathsf X=\{1,\ldots,m\}$ and write
\begin{equation*}
P_{ij}=\sum_{v\in\N^d}q_v(i,\{j\}).
\end{equation*}
Let $P$ be irreducible and aperiodic, with invariant distribution $\pi$.  For
$i\in\mathsf X$ assume
\begin{equation}
\sum_{j=1}^m\sum_{v\in\N^d}
q_v(i,\{j\})
\left(\|v\|^2+\|g(i,j,v)\|^2\right)<\infty.
\label{eq:finite-second-moment}
\end{equation}
Define
\begin{eqnarray}
\mu&=&\sum_{i=1}^m\pi_i\sum_{j=1}^m\sum_v
v q_v(i,\{j\}),
\label{eq:finite-mu}\\
\bar g&=&\sum_{i=1}^m\pi_i\sum_{j=1}^m\sum_v
g(i,j,v)q_v(i,\{j\}).
\label{eq:finite-gbar}
\end{eqnarray}
We continue to assume that every coordinate of $\mu$ is strictly positive.

The next proposition establishes the Brownian functional limit in the
finite-state setting.

\begin{proposition}
\label{prop:finite-state-fclt}
Let $P$ be irreducible and aperiodic.  If
$\mu\in(0,\infty)^d$ and \eqref{eq:finite-second-moment} holds, then
Assumption~\ref{ass:brownian-fclt} is satisfied for every initial
distribution.  If
\begin{equation*}
\zeta_n=
\begin{pmatrix}X_n-\mu\\
g(J_{n-1},J_n,X_n)-\bar g
\end{pmatrix},
\end{equation*}
then the covariance matrix $\boldsymbol\Gamma$ also admits the representation
\begin{equation}
\boldsymbol\Gamma
=\gamma_0+\sum_{h=1}^{\infty}(\gamma_h+\gamma_h^{\mathsf T}),
\qquad
\gamma_h=\Cov_\pi(\zeta_1,\zeta_{1+h}).
\label{eq:covariance-series}
\end{equation}
The series in \eqref{eq:covariance-series} converges absolutely entrywise.
\end{proposition}

\begin{proof}
Let $r=d+p$ and set
\begin{equation*}
b(i)=\Ee_i[\zeta_1]\in\R^r.
\end{equation*}
By \eqref{eq:finite-mu}--\eqref{eq:finite-gbar}, $\pi b=0$.  Since $P$ is
an irreducible finite-state transition matrix, the Poisson equation
\begin{equation}
h-Ph=b,
\qquad \pi h=0,
\label{eq:poisson-finite}
\end{equation}
has a unique solution $h:\mathsf X\to\R^r$.  With
$\mathcal F_n=\sigma(J_0,S_0,\ldots,J_n,S_n)$, define
\begin{equation*}
D_n=\zeta_n-b(J_{n-1})+h(J_n)-Ph(J_{n-1}).
\end{equation*}
The Markov-additive transition rule gives
$\Ee[D_n\mid\mathcal F_{n-1}]=0$.  By \eqref{eq:poisson-finite},
\begin{equation}
\sum_{j=1}^n\zeta_j
=\sum_{j=1}^nD_j+h(J_0)-h(J_n).
\label{eq:martingale-coboundary}
\end{equation}
The last two terms are bounded because the state space is finite.

Put $V(i)=\Ee_i[D_1D_1^{\mathsf T}]$.  The ergodic theorem gives, uniformly
on bounded rescaled time intervals,
\begin{equation*}
\frac1n\sum_{j=1}^{\lfloor ns\rfloor}
\Ee[D_jD_j^{\mathsf T}\mid\mathcal F_{j-1}]
=\frac1n\sum_{j=1}^{\lfloor ns\rfloor}V(J_{j-1})
\longrightarrow s\sum_i\pi_iV(i)
\end{equation*}
in probability.  Uniformity follows from the deterministic implication
$A_k/k\to0\Rightarrow n^{-1}\max_{k\leq nT}|A_k|\to0$,
applied entrywise to the centered ergodic sums.  For every $T,\varepsilon>0$,
the conditional Lindeberg bound is
\begin{equation*}
 \frac1n\sum_{j=1}^{\lfloor nT\rfloor}
 \Ee\!\left[\|D_j\|^2
       \one_{\{\|D_j\|>\varepsilon\sqrt n\}}
       \mid\mathcal F_{j-1}\right]
 \leq (T+1)\max_{i\in\mathsf X}
       \Ee_i\!\left[\|D_1\|^2
       \one_{\{\|D_1\|>\varepsilon\sqrt n\}}\right]
 \longrightarrow0.
\end{equation*}
The maximum tends to zero because the state space is finite and each
conditional law has a finite second moment.
The martingale functional central limit theorem \citep{Rebolledo1980} and
\eqref{eq:martingale-coboundary} now prove the claimed convergence, with
\begin{equation*}
\boldsymbol\Gamma=\sum_i\pi_iV(i).
\end{equation*}

For $h\geq1$, conditioning on the first transition gives
\begin{equation*}
 \gamma_h
   =\Ee_\pi\!\left[
      \zeta_1\{P^{h-1}b(J_1)\}^{\mathsf T}\right].
\end{equation*}
Since $\pi b=0$, finite-state irreducibility and aperiodicity give
constants $C<\infty$ and $\vartheta\in(0,1)$ such that
$\max_i\|P^{h-1}b(i)\|\leq C\vartheta^{h-1}\|b\|_\infty$.
Consequently,
\begin{equation*}
 \|\gamma_h\|\leq
 C\vartheta^{h-1}\|b\|_\infty\Ee_\pi[\|\zeta_1\|],
\end{equation*}
so the covariance series is absolutely convergent.  Stationarity gives
\begin{equation*}
 \frac1n\Var_\pi\!\left(\sum_{j=1}^n\zeta_j\right)
 =\gamma_0+
   \sum_{h=1}^{n-1}\left(1-\frac hn\right)
       (\gamma_h+\gamma_h^{\mathsf T}).
\end{equation*}
Dominated convergence identifies its limit with the displayed series.
If $M_n=\sum_{j=1}^nD_j$ and $R_n=h(J_0)-h(J_n)$, then
$\sup_n\Ee_\pi[\|R_n\|^2]<\infty$ and
$\Ee_\pi[\|M_n\|^2]=n\sum_i\pi_i\operatorname{tr}V(i)$.
Cauchy--Schwarz therefore shows that the covariance matrices of
$M_n+R_n$ and $M_n$, divided by $n$, differ by $O(n^{-1/2})$.
The latter normalized covariance is $\sum_i\pi_iV(i)$, proving
the claimed equality of the two formulas.
\end{proof}

Positive definiteness is not needed for the global inverse theorem.
The marked theorem in Section~\ref{sec:marked} requires positive variances
only for the normalized coordinate contrasts that determine face selection.

\subsection{Regenerative sufficient conditions and boundary dependence}
\label{subsec:regenerative-fclt}

We next give sufficient conditions on a general state space and describe
the boundary dependence emphasized by Alsmeyer~\cite{Alsmeyer1994} in the Markov renewal theorem.
Our purpose is to carry that dependence
through the joint displacement--reward covariance used below.  Suppose that a
Nummelin splitting \citep{Nummelin1978Splitting} of the
\emph{joint} state--increment kernel has produced an atom
$\mathfrak a$ in an enlarged state space.  The defining property is that all
representatives of $\mathfrak a$ have the same outgoing law for the next
state, increment, and split variable.  Let
\begin{equation*}
\widetilde T_0<\widetilde T_1<\widetilde T_2<\cdots
\end{equation*}
be the first entrance time and the successive return times to the atom.  We
assume these times are almost surely finite, and regard the segment before
$\widetilde T_0$ as the delayed initial piece.  For $m\geq1$, put
\begin{equation*}
\widetilde{\mathcal L}_m=\widetilde T_m-\widetilde T_{m-1},\qquad
\widetilde{\mathcal V}_m=S_{\widetilde T_m}-S_{\widetilde T_{m-1}}.
\end{equation*}
\begin{eqnarray}
\widetilde{\mathcal R}_m(g)&=&
\sum_{n=\widetilde T_{m-1}+1}^{\widetilde T_m}
g(J_{n-1},J_n,X_n),\nonumber\\
\widetilde{\mathcal R}_m^*(g)&=&
\sum_{n=\widetilde T_{m-1}+1}^{\widetilde T_m}
\|g(J_{n-1},J_n,X_n)\|.
\label{eq:cycle-reward}
\end{eqnarray}
The half-open convention in \eqref{eq:cycle-reward} assigns the first edge
leaving the atom to the new cycle.  That edge still records its origin
$J_{\widetilde T_{m-1}}$, which is the terminal representative produced by
the preceding cycle.

The first reward block may therefore depend on the terminal representative of
the delayed piece.  We absorb the path through $\widetilde T_1$ into one
delayed term and set $\Delta=\widetilde T_1$.  Define the shifted processes
\begin{equation*}
 J_n^*=J_{\Delta+n},\qquad
 S_n^*=S_{\Delta+n}-S_\Delta,\qquad
 A_n^*(g)=A_{\Delta+n}(g)-A_\Delta(g),\qquad n\geq0,
\end{equation*}
and their return clock
\begin{equation*}
 T_0=0,\qquad T_m=\widetilde T_{m+1}-\Delta,\quad m\geq1.
\end{equation*}
For the remaining blocks set
\begin{equation*}
 (\mathcal L_m,\mathcal V_m,\mathcal R_m(g),\mathcal R_m^*(g))
 =(\widetilde{\mathcal L}_{m+1},\widetilde{\mathcal V}_{m+1},
   \widetilde{\mathcal R}_{m+1}(g),\widetilde{\mathcal R}_{m+1}^*(g)),
 \qquad m\geq1.
\end{equation*}
Thus every return clock and every block variable below belongs to the same
reindexed stationary sequence; the contribution through
$\widetilde T_1$ is part of the delayed term.

The following proposition gives the dependence structure and the resulting
long-run covariance.  Its moment condition uses the sum of absolute reward
magnitudes within each cycle, so that it controls incomplete cycles as well
as the centered totals of completed cycles.

\Needspace{9\baselineskip}
\begin{proposition}
\label{prop:regen-fclt}
For the split construction and reindexed blocks defined above, the following
assertions hold.

\begin{enumerate}[label=\textup{(\roman*)}]
\item The pairs $(\mathcal L_m,\mathcal V_m)$ are i.i.d.  For a general
transition reward, the
sequence
\begin{equation*}
\bigl(\mathcal L_m,\mathcal V_m,\mathcal R_m(g)\bigr)_{m\geq1}
\end{equation*}
is stationary and $1$-dependent, but need not be independent.

\item Assume that
\begin{equation}
\Ee\left[
\left(\mathcal L_1+\|\mathcal V_1\|+\mathcal R_1^*(g)\right)^2
\right]<\infty,
\label{eq:cycle-second-moment}
\end{equation}
with an almost surely finite initial delay $\Delta$.  No moment condition
on this single delayed segment is required.  Define
\begin{equation}
\ell=\Ee[\mathcal L_1],
\qquad
\mu=\frac{\Ee[\mathcal V_1]}{\ell},
\qquad
\bar g=\frac{\Ee[\mathcal R_1(g)]}{\ell}.
\label{eq:cycle-means}
\end{equation}
Assume that every coordinate of $\mu$ is strictly positive, and set
\begin{equation*}
C_m=
\begin{pmatrix}
\mathcal V_m-\mu\mathcal L_m\\
\mathcal R_m(g)-\bar g\mathcal L_m
\end{pmatrix}.
\end{equation*}
Then Assumption~\ref{ass:brownian-fclt} holds and its long-run covariance is
\begin{equation}
\boldsymbol\Gamma
=\frac1\ell\left\{
\Cov(C_1)+\Cov(C_1,C_2)+\Cov(C_2,C_1)
\right\}.
\label{eq:regenerative-long-run-covariance}
\end{equation}
\end{enumerate}
\end{proposition}

\begin{proof}
We first make the dependence structure explicit.  Let $E_0$ record the
excursion from $\widetilde T_0$ through $\widetilde T_1$, excluding its
departure representative at $\widetilde T_0$.  For $m\geq1$, let $E_m$ be
the complete fresh excursion record sampled after the return at $T_{m-1}$.  It
contains the first outgoing state--increment pair, all subsequent innovations
up to the next return, and the terminal representative of the atom at $T_m$;
it does not include the representative already present at $T_{m-1}$.  The
atom property and the strong Markov construction of the split chain imply
that $(E_m)_{m\geq0}$ are i.i.d.  Both $\mathcal L_m$ and
$\mathcal V_m$ are measurable functions
of $E_m$.

Let $\xi_m=J^*_{T_m}$ denote the terminal representative.  It is a function of
$E_m$.  All terms of $\mathcal R_m(g)$ except possibly its first reward term
are functions of $E_m$ alone.  That term also uses
$J^*_{T_{m-1}}=\xi_{m-1}$.  Consequently
there is a measurable map $\Phi_g$ such that
\begin{equation*}
\mathcal R_m(g)=\Phi_g(E_{m-1},E_m),\qquad m\geq1.
\end{equation*}
Variables belonging to collections of cycle
indices at mutual distance at least two are functions of disjoint collections
of i.i.d. excursion records.  This proves stationarity and $1$-dependence, and
also proves the i.i.d. assertion for $(\mathcal L_m,\mathcal V_m)$.

By \eqref{eq:cycle-means} and \eqref{eq:cycle-second-moment}, the $C_m$ are
centered and square-integrable.  A stationary
$1$-dependent invariance principle is classical \citep{Herrndorf1984};
the excursion representation gives a direct argument here.
Write $C_m=F(E_{m-1},E_m)$ and define
\begin{equation*}
 a(x)=\Ee[F(x,E_1)],\qquad b(y)=\Ee[F(E_0,y)],\qquad
 r(x,y)=F(x,y)-a(x)-b(y).
\end{equation*}
The variables
\begin{equation*}
 D_m=a(E_m)+b(E_m)+r(E_{m-1},E_m)
\end{equation*}
are stationary square-integrable martingale differences for
$\mathcal H_m=\sigma(E_0,\ldots,E_m)$, since integrating $r$ in either
argument gives zero.  Moreover,
\begin{equation*}
 \sum_{m=1}^j C_m=\sum_{m=1}^j D_m+a(E_0)-a(E_j).
\end{equation*}
The conditional covariance of $D_m$ is an integrable function of
$E_{m-1}$.  The ergodic theorem therefore gives convergence of the
rescaled conditional covariance sums, uniformly on compact time intervals,
to $s\Ee[D_1D_1^{\mathsf T}]$.  For every $T,\varepsilon>0$, the expected
conditional Lindeberg sum is bounded by
$(T+1)\Ee[\|D_1\|^2\one_{\{\|D_1\|>\varepsilon\sqrt n\}}]$, which tends
to zero.  The martingale functional central limit theorem
\citep{Rebolledo1980} consequently applies.  Finally, square integrability
and the union bound give
\begin{equation*}
 \Pp\left(\max_{0\leq j\leq\lfloor nT\rfloor}\|a(E_j)\|
       >\varepsilon\sqrt n\right)
 \leq(nT+1)\Pp(\|a(E_0)\|>\varepsilon\sqrt n)\longrightarrow0.
\end{equation*}
Thus
\begin{equation}
 \left(\frac1{\sqrt n}\sum_{m=1}^{\lfloor ns\rfloor}C_m\right)_{s\geq0}
 \Rightarrow W_K,
 \label{eq:cycle-fclt}
\end{equation}
where $W_K$ is Brownian motion with covariance per unit cycle
\begin{equation*}
 K=\Cov(C_1)+\Cov(C_1,C_2)+\Cov(C_2,C_1).
\end{equation*}
Indeed, stationarity and $1$-dependence give this limit for
$n^{-1}\Var(\sum_{m=1}^nC_m)$; the bounded $L^2$ norm of the coboundary
shows that it also equals $\Ee[D_1D_1^{\mathsf T}]$.

It remains to transfer cycle time to transition time.  Write
\begin{equation*}
\mathcal M(n)=\max\{m:T_m\leq n\}.
\end{equation*}
The ordinary renewal law of large numbers for the i.i.d. positive variables
$\mathcal L_m$ yields
\begin{equation}
\sup_{0\leq s\leq T}
\left|\frac{\mathcal M(\lfloor ns\rfloor)}n-\frac{s}{\ell}\right|
\xrightarrow{\Pp}0
\label{eq:cycle-time-change}
\end{equation}
for every $T<\infty$.  The standard random-time-change argument
\citep{Whitt1980,GlynnWhitt1993}, applied to
\eqref{eq:cycle-fclt} and \eqref{eq:cycle-time-change}, gives Brownian
covariance $K/\ell$.

To control the incomplete-cycle terms under the stated moment condition, set
\begin{equation*}
W_m=\mathcal L_m+\|\mathcal V_m\|+\mathcal R_m^*(g).
\end{equation*}
For every fixed $c,\varepsilon>0$, stationarity and the union bound give
\begin{equation*}
\Pp\left(\max_{1\leq m\leq\lfloor cn\rfloor}W_m>
\varepsilon\sqrt n\right)
\leq cn\,\Pp(W_1>\varepsilon\sqrt n)\longrightarrow0,
\end{equation*}
because $x^2\Pp(W_1>x)\to0$ when $\Ee[W_1^2]<\infty$.
Nonnegativity bounds each partial displacement by a constant multiple
of $\|\mathcal V_m\|$, and each partial absolute reward by
$\mathcal R_m^*(g)$.  Thus, for a finite constant $C$ depending only on
the norms, $\mu$ and $\bar g$,
\begin{equation*}
 \sup_{0\leq s\leq T}
 \left\|
 \begin{pmatrix}
 S^*_{\lfloor ns\rfloor}-\mu\lfloor ns\rfloor\\
 A^*_{\lfloor ns\rfloor}(g)-\bar g\lfloor ns\rfloor
 \end{pmatrix}
 -\sum_{j=1}^{\mathcal M(\lfloor ns\rfloor)}C_j
 \right\|
 \leq C\max_{1\leq j\leq\mathcal M(\lfloor nT\rfloor)+1}W_j.
\end{equation*}
The preceding maximal bound and \eqref{eq:cycle-time-change} make this
remainder $o_{\Pp}(\sqrt n)$.  This proves the functional limit for the
shifted processes.  The initial quantity
\begin{equation*}
 \Delta+\|S_\Delta\|
       +\sum_{j=1}^{\Delta}\|g(J_{j-1},J_j,X_j)\|
\end{equation*}
is finite almost surely, because $\Delta$ is finite and $g$ is finite-valued.
Its ratio to $\sqrt n$ therefore tends to zero almost surely; no
integrability is needed.  Also $\Delta/n\to0$ almost surely.  Continuity
of the Brownian limit permits restoring the delayed segment and original
clock.  This proves the original-chain functional central limit theorem
and \eqref{eq:regenerative-long-run-covariance}.
\end{proof}

The covariance correction in
\eqref{eq:regenerative-long-run-covariance} is necessary for general
transition rewards.

\begin{remark}
\label{rem:false-independent-cycles}
All representatives of the atom have the same outgoing transition law, but
the first reward in a cycle may depend on the representative at which the
preceding excursion arrived.  Thus independence of
$(\mathcal L_m,\mathcal V_m)$ does not imply independence of
$(\mathcal L_m,\mathcal V_m,\mathcal R_m(g))$.  Treating the reward blocks as
independent would omit the two
lag-one matrices in \eqref{eq:regenerative-long-run-covariance} and can give
an incorrect limiting variance.  This possible one-dependence is discussed
by Alsmeyer~\cite{Alsmeyer1994}; the formula above gives its effect on the
joint covariance used here.  The lag terms vanish, for example, when the
atom is a singleton or when the first-edge reward is insensitive to the atom
representative, but that is an additional property of the reward.
\end{remark}

The following two-state example exhibits a nonzero lag-one correction.

\begin{example}
\label{ex:lag-one-reward}
Let $\mathsf X=\{0,1\}$, let both rows of $P$ equal $(1/2,1/2)$, and take
$X_n=1$.  The entire state space is an atom and every cycle has length and
displacement one.  In stationarity the states $(J_n)$ are i.i.d. Bernoulli
$1/2$.  For the edge reward $g(i,j,1)=ij$,
\begin{equation*}
\mathcal R_m(g)=J_{m-1}J_m.
\end{equation*}
Hence the reward blocks are stationary and $1$-dependent, with
\begin{equation*}
\Var(\mathcal R_1)=\frac{3}{16},
\qquad
\Cov(\mathcal R_1,\mathcal R_2)=\frac1{16}.
\end{equation*}
The long-run reward variance is therefore $5/16$; an independence
approximation would give $3/16$.
\end{example}

The inverse limits in Section~\ref{sec:stratified} are stated in terms of
$G$, or of $(B_S,B_A)$ and their long-run covariance.  Regeneration and the
finite-state Poisson equation give two sets of sufficient conditions for
these functional limits.

\section{Auxiliary topological and discretization details}
\label{app:auxiliary}

This appendix establishes metrizability of the trace enhancement,
measurability of the limit parametrization, and the error bound for
discrete-time interpolation.  It also extends the directional formulation
from the simplex to the positive orthant.

\subsection{A metric and measurable limit parametrization}

Let $d_{\mathcal F}$ be a bounded metric inducing the Fell topology on the
closed subsets of the locally compact separable space
$\Delta^\circ\times\R$.  For locally bounded
$f,g:\Delta^\circ\to\R$, define the global hypi semimetric
\begin{equation*}
 d_{\mathrm{hypi}}^\Delta(f,g)
 =d_{\mathcal F}\!\left(
       \operatorname{cl}(\operatorname{epi}f),
       \operatorname{cl}(\operatorname{epi}g)\right)
 \vee d_{\mathcal F}\!\left(
       \operatorname{cl}(\operatorname{hypo}f),
       \operatorname{cl}(\operatorname{hypo}g)\right).
\end{equation*}
This is a semimetric on function representatives: its zero classes are
precisely the functions with the same lower and upper semicontinuous hulls.

Enumerate the nonempty subsets of $[d]$ as $I_1,\ldots,I_{2^d-1}$.  Put
\begin{equation}
 d_{\mathrm{str}}(f,g)
=(1\wedge d_{\mathrm{hypi}}^\Delta(f,g))
 +\sum_{\ell=1}^{2^d-1}\sum_{m\geq1}
 2^{-\ell-m}
 \left\{1\wedge
 \sup_{\lambda\in\mathcal D_{I_\ell,m}}
 |f(\lambda)-g(\lambda)|\right\},
 \label{eq:explicit-stratified-metric}
\end{equation}
with a zero term for an empty compact set.  Although the first coordinate is
only a semimetric, $d_{\mathrm{str}}$ is a metric: if it vanishes, the trace
terms show that $f=g$ on every stratum, and the strata partition
$\Delta^\circ$.

We first record the stability under locally uniform perturbations used
in the inverse theorem.

\begin{lemma}
\label{lem:stratified-local-perturbation}
Let $f_n,g_n$ be locally bounded real functions on $\Delta^\circ$.  If
$\sup_K|f_n-g_n|\to0$ for every compact $K\subset\Delta^\circ$, then
$d_{\mathrm{str}}(f_n,g_n)\to0$.
\end{lemma}

\begin{proof}
The Fell hyperspace of the locally compact separable metric space
$\Delta^\circ\times\R$, including the empty set, is compact and
metrizable.  Along any subsequence, pass to a further subsequence on which
both closed epigraphs have Fell limits, say $E$ and $E'$.  For
$(\lambda,z)\in E$, the lower Fell condition and approximation by the
unclosed epigraph provide
$(\lambda_n,z_n)\in\operatorname{epi}f_n$ converging to
$(\lambda,z)$.  The points $\lambda_n$ eventually belong to a compact
neighborhood $K$ of $\lambda$.  With
$e_n=\sup_K|f_n-g_n|$, the points
$(\lambda_n,z_n+e_n)$ belong to $\operatorname{epi}g_n$ and have the
same limit.  The upper Fell condition gives $(\lambda,z)\in E'$.
Interchanging $f_n$ and $g_n$ gives $E=E'$.  Thus their closed epigraphs
are asymptotically at zero distance in $d_{\mathcal F}$.  The identical
argument, with vertical shifts $-e_n$, treats the closed hypographs.
Every trace set $\mathcal D_{I,m}$ is compact, so each trace distance also
tends to zero.  Truncation of the summable weights proves the claim.
\end{proof}

The following proposition makes the topology and its parametrization
explicit.

\begin{proposition}
\label{prop:stratified-metric}
The metric in \eqref{eq:explicit-stratified-metric} induces
Definition~\ref{def:stratified-topology}.  For continuous
$g:[0,T]\to\R^d$ with $T>\rho_{\max}$, and continuous
$b:\Delta^\circ\to\R^d$, the map
\begin{equation*}
 (g,b)\longmapsto\Phi_{g,b},\qquad
 \Phi_{g,b}(\lambda)=
 \min_{r\in I(\lambda)}
 \frac{b_r(\lambda)-g_r(\rho(\lambda))}{\mu_r},
\end{equation*}
is continuous from locally uniform convergence into
$d_{\mathrm{str}}$.  Consequently, if $G$ is a Borel random element in the
path space, then $\Phi_{G,b}$ is a Borel random element and has separable
range.
\end{proposition}

\begin{proof}
A weighted sum of bounded component semimetrics induces their common topology,
proving the first assertion.  The use of one global Fell coordinate is
essential: taking semicontinuous hulls only after restriction to arbitrary
compact sets could introduce artificial boundary values.

For a compact $K\subset\Delta^\circ$, the finite-minimum inequality gives
\begin{equation*}
 \sup_{\lambda\in K}|\Phi_{g_n,b_n}(\lambda)-\Phi_{g,b}(\lambda)|
 \leq\max_{r\leq d}\frac{
 \sup_{\lambda\in K}|b_{n,r}(\lambda)-b_r(\lambda)|
 +\sup_{s\leq T}|g_{n,r}(s)-g_r(s)|}{\mu_r}.
\end{equation*}
Thus locally uniform convergence of $(g_n,b_n)$ implies locally uniform
convergence of the displayed fields, hence hypi convergence and convergence
of every trace.  Truncating the summable weights proves continuity in
$d_{\mathrm{str}}$.  The continuous-function spaces on the present locally
compact separable domains are separable.  Their continuous image is
separable, and continuity gives the asserted Borel measurability.
\end{proof}

\begin{remark}
The topology does not introduce independently varying values on intersecting
strata.  Compatibility is automatic because every coordinate in
\eqref{eq:explicit-stratified-metric} is obtained by restricting the same
function.  What it adds to the hypi equivalence class is the exact trace,
which ordinary epi/hypograph convergence intentionally discards.
\end{remark}

The inverse fields used in Theorem~\ref{thm:stratified-fclt} are themselves
Borel random elements of this metric space.  First, the functional limit
implies $S_{n,r}/n\to\mu_r>0$ in probability.  Monotonicity gives, for
each finite $M$,
\begin{equation*}
 \Pp\bigl(\sup_nS_{n,r}\leq M\bigr)
 \leq\Pp(S_{n,r}\leq M)\longrightarrow0.
\end{equation*}
Thus every coordinate tends to infinity almost surely, and all $N(k)$ are
finite simultaneously.  For fixed $t$, view $(N(k))_{k\in\N^d}$ as a
random element of the countable product of discrete copies of $\N$.
The deterministic map
\begin{equation*}
 (n_k)_{k\in\N^d}\longmapsto
 \left(\lambda\longmapsto
            \frac{n_{k_t(\lambda)}-t\rho(\lambda)}{a_t}\right)
\end{equation*}
is continuous for local uniform convergence: on each compact,
$k_t$ takes finitely many values, and product convergence eventually fixes
all the corresponding coordinates.  The preceding lemma gives continuity
into $d_{\mathrm{str}}$.  Consequently $X_t$ is Borel and has separable
range.  The field $F_t$ in the theorem's proof depends continuously, for
local uniform convergence, on the finitely many increments needed to
evaluate $G_t$ up to $\rho_{\max}$; it has the same measurability
properties.  On a compact critical chart, only finitely many integer
thresholds occur, giving Borel measurability in $\ell^\infty(H)$ as well.
Finally, continuity of $(g,b)\mapsto\Phi_{g,b}$ sends compact sets of
continuous paths to compact subsets of the target metric space.
The limiting law is therefore tight, not merely concentrated on a
separable subset.  These observations justify the ordinary weak-convergence
and asymptotic-equivalence arguments for these fields.

\subsection{Interpolation of a monotone path}

Let $(s_n)_{n\geq0}\subset\R_+^d$ be coordinatewise nondecreasing.  Let
$\widetilde s$ be
its piecewise linear interpolation, parametrized so that
$\widetilde s(n)=s_n$, and define
\begin{equation*}
 N_s(k)=\sup\{n\in\N:s_n\leq_d k\},\qquad
 \mathcal I_t(k)=\sup\{u\geq0:t^{-1}\widetilde s(tu)\leq_d t^{-1}k\}.
\end{equation*}

The discrete and interpolated inverses differ by at most one transition.

\begin{lemma}
\label{lem:interpolation-error}
For every threshold $k$ with $s_0\leq_d k$ and $N_s(k)<\infty$,
\begin{equation}
 \frac{N_s(k)}t\leq \mathcal I_t(k)\leq\frac{N_s(k)+1}{t}.
 \label{eq:interpolation-error}
\end{equation}
Consequently, at any fluctuation scale $a_t\to\infty$,
\begin{equation*}
 \frac{t\mathcal I_t(k)-N_s(k)}{a_t}\longrightarrow0
\end{equation*}
uniformly over thresholds $k$ satisfying these conditions.
\end{lemma}

\begin{proof}
At time $N_s(k)$ the interpolated path is inside the rectangle.  At time
$N_s(k)+1$ it is outside, by maximality and monotonicity, and after that it
cannot re-enter.  The last interpolated time inside is therefore in the
single interval between those two integer times.  This proves
\eqref{eq:interpolation-error}.  Multiplying it by $t/a_t$ and using
$a_t\to\infty$ gives the normalized conclusion.
\end{proof}

Applied with $s_n=S_n$, Assumption~\ref{ass:zero-cluster} guarantees
finiteness even when the path has flat segments.  Under the functional
limit assumption, finiteness also follows directly from the preceding
monotonicity argument, without a zero-cluster hypothesis.  The
coordinatewise floor in $k_t(\lambda)$ contributes at most $a_t^{-1}$ to
$b_t$, and Lemma~\ref{lem:interpolation-error} contributes at most
$a_t^{-1}$ to the inverse.  Both discretization errors vanish at the scale used in
Theorem~\ref{thm:stratified-fclt}.

\subsection{Radial directions}

Normalization to $\Delta^\circ$ loses no information.  If
$\lambda=r\theta$ with $r=|\lambda|_1>0$ and
$\theta\in\Delta^\circ$, then
\begin{equation*}
 \rho(r\theta)=r\rho(\theta),\qquad I(r\theta)=I(\theta).
\end{equation*}
On a compact subset of $(0,\infty)^d$, the change of variables
$\lambda\mapsto(r,\theta)$ and its inverse are uniformly continuous.
For the joint extension in the radial and directional variables, fix a
compact set $K\subset(0,\infty)^d$ and choose $T>\sup_{\lambda\in K}\rho(\lambda)$.
The uniform inverse bounds apply directly on $K$, with the original
time parameter $t$ and scale $a_t$.  The active-set argument also
applies on the positive orthant: each tied coordinate can be made uniquely
active by a sufficiently small perturbation within that orthant.
Consequently the compact-uniform, hypi and fixed-stratum conclusions extend
jointly to the radial and directional variables.  The corresponding trace
exhaustions may be taken as
\begin{equation*}
 \left\{\lambda\in(0,\infty)^d:I(\lambda)=I,\quad
 m^{-1}\leq\lambda_r\leq m\ (r\leq d),\quad
 \min_{s\notin I}d_s(\lambda)\geq m^{-1}\right\},
\end{equation*}
with the last condition omitted for $I=[d]$.
This argument uses no regular-variation assumption on $a_t$ and no
rescaling identity for $G$.  The simplex presentation displays the finite
polyhedral stratification without a redundant scaling direction.

\bibliographystyle{preprint}
\bibliography{references}

\end{document}